\UseRawInputEncoding 
\documentclass[12pt,reqno]{amsart}
\usepackage{fullpage}

\newtheorem{theorem}{Theorem}[section]
\newtheorem{lemma}[theorem]{Lemma}
\newtheorem{prop}[theorem]{Proposition}
\newtheorem{cor}[theorem]{Corollary}
\newtheorem{conj}[theorem]{Conjecture}

\newtheorem{prob}[theorem]{Problem}
\usepackage{graphicx}
\usepackage{color}
\usepackage[dvipsnames]{xcolor}
\usepackage{subfigure}
\usepackage{amssymb}
\usepackage{amsmath,mathrsfs}
\usepackage{colonequals}
\usepackage{hyperref}
\theoremstyle{definition}
\newtheorem{definition}[theorem]{Definition}

\newtheorem{remark}[theorem]{Remark}

\renewcommand{\subset}{\subseteq}
\renewcommand{\supset}{\supseteq}
\renewcommand{\epsilon}{\varepsilon}
\renewcommand{\nu}{v}

\newcommand{\abs}[1]{\left|#1\right|}                   % Absolute value notation
\newcommand{\absf}[1]{|#1|}                             % small absolute value signs
\newcommand{\vnorm}[1]{\left\|#1\right\|}    % norm notation
\newcommand{\vnormf}[1]{\|#1\|}                         % norm notation, forced to be small
\newcommand{\vnormt}[1]{\left\|#1\right\|}    % norm notation
\newcommand{\vnormtf}[1]{\|#1\|}                         % norm notation, forced to be small
\newcommand{\E}{\mathbb{E}}

\newcommand{\R}{\mathbb{R}}
\newcommand{\C}{\mathbb{C}}

\newcommand{\embolden}[1]{\textbf {#1}}

\newcommand{\redA}{\partial^{*}\Omega}
\newcommand{\redb}{\partial^{*}}
\newcommand{\sdimn}{n}
\newcommand{\adimn}{n+1}
\newcommand{\scon}{\lambda}
\newcommand{\pcon}{\delta}

\begin{document}

\title{Symmetric Convex Sets with Minimal\\ Gaussian Surface Area}

\author{Steven Heilman}
\address{Department of Mathematics, University of Southern California, Los Angeles, CA 90089-2532}
\email{stevenmheilman@gmail.com}
\date{\today}
\thanks{Supported by NSF Grant DMS 1708908.}
%60E15, 60G15, 53A10, 58E30

%For talk
% Let B be a subset of Euclidean space, such that B has minimal Gaussian surface area, subject to the constraint that B=-B and the Gaussian volume of B is fixed.  It is shown that if B is convex, and if the average of the squared length of the second fundamental form of the boundary of B is not close to 1, then the boundary of B is a round cylinder.  That is, we essentially resolve the convex case of a question of Barthe from 2001.
%
%The main tool is the Colding-Minicozzi theory for Gaussian minimal surfaces, which studies eigenfunctions of an Ornstein-Uhlenbeck type operator associated to the a surface.  A key new ingredient is the use of a randomly chosen degree 2 polynomial in the second variation formula for the Gaussian surface area.  Our actual results are a bit more general than the above statement.  Also, some of our results hold without the assumption of convexity.

\begin{abstract}
Let $\Omega\subset\mathbb{R}^{n+1}$ have minimal Gaussian surface area among all sets satisfying $\Omega=-\Omega$ with fixed Gaussian volume.  Let $A=A_{x}$ be the second fundamental form of $\partial\Omega$ at $x$, i.e. $A$ is the matrix of first order partial derivatives of the unit normal vector at $x\in\partial\Omega$.  For any $x=(x_{1},\ldots,x_{n+1})\in\mathbb{R}^{n+1}$, let $\gamma_{n}(x)=(2\pi)^{-n/2}e^{-(x_{1}^{2}+\cdots+x_{n+1}^{2})/2}$.  Let $\|A\|^{2}$ be the sum of the squares of the entries of $A$, and let $\|A\|_{2\to 2}$ denote the $\ell_{2}$ operator norm of $A$.

It is shown that if $\Omega$ or $\Omega^{c}$ is convex, and if either
$$\int_{\partial\Omega}(\|A_{x}\|^{2}-1)\gamma_{n}(x)dx>0\qquad\mbox{or}\qquad
\int_{\partial\Omega}\Big(\|A_{x}\|^{2}-1+2\sup_{y\in\partial\Omega}\|A_{y}\|_{2\to 2}^{2}\Big)\gamma_{n}(x)dx<0,$$
then $\partial\Omega$ must be a round cylinder.  That is, except for the case that the average value of $\|A\|^{2}$ is slightly less than $1$, we resolve the convex case of a question of Barthe from 2001.

The main tool is the Colding-Minicozzi theory for Gaussian minimal surfaces, which studies eigenfunctions of the Ornstein-Uhlenbeck type operator $L= \Delta-\langle x,\nabla \rangle+\|A\|^{2}+1$ associated to the surface $\partial\Omega$.  A key new ingredient is the use of a randomly chosen degree 2 polynomial in the second variation formula for the Gaussian surface area.  Our actual results are a bit more general than the above statement.  Also, some of our results hold without the assumption of convexity.
\end{abstract}
\maketitle
%
%
% arxiv subjects: math.PR, math.DG, cs.CC?
%
%  MSC:    60E15, 60G15, 53A10, 58E30
%
%  keywords: convex, symmetric, Gaussian, minimal surface, calculus of variations

\section{Introduction}\label{secintro}

In a landmark investigation of mean curvature flow \cite{colding12a}, Colding and Minicozzi studied a maximal version of the Gaussian surface area of an $\sdimn$-dimensional hypersurface $\Sigma$ in $\R^{\adimn}$.  They called this quantity
\begin{equation}\label{one1}
\sup_{a>0,b\in\R^{\adimn}}\int_{\Sigma}a^{-\frac{\sdimn}{2}}\gamma_{\sdimn}((x-b)a^{-1/2})dx
\end{equation}
the ``entropy'' of $\Sigma$.  The Colding-Minicozzi entropy \eqref{one1} is of interest since it monotonically decreases under the mean curvature flow.  For this reason, \cite{colding12a} studied minimizers of \eqref{one1}.  Here, with $m=n+1$, we define
$$\gamma_{\sdimn}(x)\colonequals (2\pi)^{-\sdimn/2}e^{-\vnormt{x}^{2}/2},\qquad
\vnormt{x}^{2}\colonequals\sum_{i=1}^{m}x_{i}^{2},\qquad
\forall\,x=(x_{1},\ldots,x_{m})\in\R^{m}.
$$
%We also define
$$\int_{\Sigma}\gamma_{\sdimn}(x)dx\colonequals\liminf_{\epsilon\to0^{+}}
\frac{1}{2\epsilon}\int_{\{x\in\R^{\adimn}\colon\exists\,y\in\Sigma,\,\vnormt{x-y}<\epsilon\}}\gamma_{\sdimn}(x)dx.$$

In the context of mean curvature flow, the Colding-Minicozzi entropy \eqref{one1} is an analogue of Perelman's reduced volume for Ricci flow.  It was conjectured in \cite{colding12} and ultimately proven in \cite{zhu16} that, among all compact $\sdimn$-dimensional hypersurfaces $\Sigma\subset\R^{\adimn}$ with $\partial\Sigma=\emptyset$, the round sphere minimizes the quantity \eqref{one1}.

Mean curvature flow refers to a set of orientable hypersurfaces $\{\Sigma_{s}\}_{s\geq0}$ such that $(d/ds)x=-H(x)N(x)$, $\forall$ $x\in\Sigma_{s}$, $\forall$ $s\geq0$, where $H(x)$ is the mean curvature of $x\in\Sigma$ and $N(x)$ is the exterior unit normal vector at $x\in\Sigma_{s}$.  (See Section \ref{seccurvature} for more detailed definitions.)

% first variation formula for Euclidean surface area is
%  (d/dt)|_{t=0}\int_{\Sigma_{t}}dx=\int_{\Sigma}\langle (d/dt)|_{t=0}x,Hn\rangle
% So, the flow defined by  (d/dt)x=-Hn decreases volume fastest.  (mean curvature flow)

Influenced by the methods of \cite{colding12a}, we study minimizers of the Gaussian surface area itself, over symmetric hypersurfaces $\Sigma\subset\R^{\adimn}$ enclosing a fixed Gaussian volume.  We say a hypersurface $\Sigma$ is \embolden{symmetric} if $\Sigma=-\Sigma$.  Without the symmetry assumption, it is well-known that the set $\Omega\subset\R^{\adimn}$ of fixed Gaussian volume $\int_{\Omega}\gamma_{\adimn}(x)dx$ and of minimal Gaussian surface area $\int_{\partial\Omega}\gamma_{\sdimn}(x)dx$ is a half space.  That is, $\Omega$ is the set lying on one side of a hyperplane \cite{sudakov74}.  This result has been elucidated and strengthened over the years \cite{borell85,ledoux94,ledoux96,bobkov97,burchard01,borell03,mossel15,mossel12,eldan13,mcgonagle15,barchiesi16}.  However, all of these proof methods (with the exception of \cite{mcgonagle15,barchiesi16}) seem unable to handle the additional constraint that the set $\Omega$ is symmetric. i.e. that $\Omega=-\Omega$.  That is, new methods are needed to find symmetric sets $\Omega\subset\R^{\adimn}$ of fixed Gaussian volume and minimal Gaussian surface area.  In this work, we demonstrate that the calculus of variations techniques of \cite{colding12a,mcgonagle15,barchiesi16} succeed in this task, where other proof strategies seem insufficient.  Informally, the calculus of variations is a ``local'' proof strategy, whereas other proof strategies such as in \cite{mossel12} or \cite{eldan13} either directly or indirectly use ``global'' translation invariance of the problem at hand, in the sense that the translation of a half space is still a half space.  So, the other methods cannot deal with the constraint $\Sigma=-\Sigma$, since a translation of such a $\Sigma$ is no longer symmetric.

It was suggested by Barthe in 2001 \cite{barthe01} that the symmetric set $\Omega$ of fixed Gaussian volume and minimal Gaussian surface area could be a symmetric strip bounded by two parallel hyperplanes.  It was also expressed in \cite{chakrabarti10,odonnell12} that a Euclidean ball centered at the origin or its complement could minimize Gaussian surface area.  A simple calculation demonstrates that the symmetric strip does not minimize Gaussian surface area for certain volume constraints.  If $t>0$ satisfies $\gamma_{1}([-t,t])=1/2$, then the Gaussian surface area of $[-t,t]$ is $\sqrt{2\pi}\gamma_{1}(\{-t\}\cup\{t\})\approx 1.5932$.  If $r>0$ and if $B(0,r)=\{(x_{1},x_{2})\in\R^{2}\colon x_{1}^{2}+x_{2}^{2}\leq r^{2}\}$ satisfies $\gamma_{2}(B(0,r))=1/2$, then $\int_{\partial B(0,r)}\gamma_{1}(x)dx\approx1.4757$.  Also, the ball $B(0,s)$ in $\R^{3}$ with $\gamma_{3}(B(0,s))=1/2$ satisfies $\int_{\partial B(0,s)}\gamma_{2}(x)dx\approx 1.4496$.  So, at least for symmetric sets of Gaussian measure $1/2$, the interval or the strip bounded by two hyperplanes does not minimize Gaussian surface area.  Moreover, it even appears that the $n$-dimensional ball of Gaussian measure $1/2$ has a decreasing surface area as $n$ increases.

Define $S^{\sdimn}\subset\R^{\adimn}$ so that
$$S^{\sdimn}\colonequals\{(x_{1},\ldots,x_{\adimn})\in\R^{\adimn}\colon x_{1}^{2}+\cdots+x_{\adimn}^{2}=1\}.$$
From the Central Limit Theorem with error bound (also known as the Edgeworth Expansion) \cite[XVI.4.(4.1)]{feller71}, for any $s\in\R$, the following asymptotic expansion holds as $n\to\infty$:
$$\gamma_{n}\left(B\left(0,\sqrt{n+s\sqrt{2n}}\,\right)\right)=\int_{-\infty}^{s}e^{-t^{2}/2}dt/\sqrt{2\pi}+\frac{(1-s^{2})e^{-s^{2}/2}/\sqrt{2\pi}}{\sqrt{n}}+o(n^{-1/2}).$$
Moreover, from the Chain rule, (denoting $B(0,r)=\{x\in\R^{n}\colon\vnormt{x}\leq r\}$),
\begin{flalign*}
\frac{d}{ds}\Big|_{s=0}\gamma_{n}\left(B\left(0,\sqrt{n+s\sqrt{2n}}\,\right)\right)
&=\frac{d}{dr}\Big|_{r=\sqrt{n}}\,\gamma_{n}(B(0,r))\cdot\frac{d}{ds}\Big|_{s=0}\sqrt{n+s\sqrt{2n}}\\
&=\frac{1}{\sqrt{2\pi}}+o(n^{-1/2}).
\end{flalign*}
That is,
$$\lim_{n\to\infty}\frac{d}{dr}\Big|_{r=\sqrt{n}}\,\gamma_{n}(B(0,r))=\frac{1}{\sqrt{\pi}}\approx .5642.$$
So, for symmetric sets of Gaussian measure $1/2$, it seems plausible that their Gaussian surface area is at least $1/\sqrt{\pi}$.

Morally, the results of \cite{colding12,bernstein16,zhu16} imply this result, but we cannot see a formal way of proving this statement.

In summary,  it is believed that solid round cylinders (or their complements) minimize Gaussian surface area \cite{barthe01,chakrabarti10,odonnell12}.  We state this as a conjecture below.

\begin{prob}\label{prob1}
Fix $0<c<1$.  Minimize $$\int_{\partial \Omega}\gamma_{\sdimn}(x)dx$$ over all subsets $\Omega\subset\R^{\adimn}$ satisfying $\Omega=-\Omega$ and $\gamma_{\adimn}(\Omega)=c$.
\end{prob}
\begin{remark}\label{crk}
If $\Omega$ minimizes Problem \ref{prob1}, then $\Omega^{c}$ also minimizes Problem \ref{prob1}, with $c$ replaced by $1-c$.
\end{remark}

\begin{conj}[{\cite{barthe01,chakrabarti10,odonnell12}}]\label{conj0}
Suppose $\Omega\subset\R^{\adimn}$ minimizes Problem \ref{prob1}.  Then, after rotating $\Omega$, $\exists$ $r>0$ and $\exists$ $0\leq k\leq n$ such that
$$\partial\Omega=r S^{k}\times\R^{\sdimn-k}.$$
\end{conj}
Except for the case that the average value of the squared length of the second fundamental form is close to $1$, we resolve the case of Conjecture \ref{conj0} where $\Omega$ or $\Omega^{c}$ is convex.  See Theorem \ref{mainthm} below.

Besides the relation of this problem to mean curvature flows \cite{colding12a,colding15a}, recent interest for Gaussian isoperimetric inequalities has developed in theoretical computer science \cite{khot07,mossel10,isaksson11}.  A typical application reduces the computational hardness of a discrete computational problem to the solution of a Gaussian isoperimetric inequality.  The resolution of a discrete problem using a continuous inequality can be surprising.  For example, Borell's isoperimetric inequality \cite{borell85,mossel12,eldan13} states that, among all $\Omega\subset\R^{\adimn}$ with fixed Gaussian volume, the one with largest noise stability is a half space.  Here the noise stability of $\Omega\subset\R^{\adimn}$ with parameter $0<\rho<1$ is
$$(1-\rho^{2})^{-\frac{\adimn}{2}}(2\pi)^{-(\adimn)}\int_{\Omega}\int_{\Omega}e^{\frac{-\vnormt{x}^{2}-\vnormt{y}_{2}^{2}+2\rho\langle x,y\rangle}{2(1-\rho^{2})}}dxdy.$$
This inequality of Borel generalizes the fact that half spaces minimize Gaussian surface area.  And Borell's inequality implies \cite{khot07} that the Goemans-Williamson algorithm for MAX-CUT \cite{goemans95} is the best possible polynomial time algorithm, assuming the Unique Games Conjecture \cite{khot02}.

The addition of the constraint that $\Omega=-\Omega$ has particularly been related to the communication complexity of the Gap-Hamming-Distance problem \cite{chakrabarti10,sherstov12,vidick13a}.  As mentioned above, looking for the symmetric set $\Omega\subset\R^{\adimn}$ of fixed Gaussian volume and minimal surface area is intrinsically interesting since previous proof strategies fail to solve this problem, with the exception of recent calculus of variations techniques \cite{colding12a,colding12,mcgonagle15,barchiesi16}.

The so-called S-inequality of \cite{latala99} seems superficially related to Problem \ref{prob1}.  One part of this inequality says: if $\Omega\subset\R^{\adimn}$ is a symmetric convex set, and if $P$ is a symmetric strip lying between two hyperplanes such that $\gamma_{\adimn}(\Omega)=\gamma_{\adimn}(P)$, then for any $t\geq1$, $\gamma_{\adimn}(t\Omega)\geq\gamma_{\adimn}(tP)$.  This result of \cite{latala99} seems to have a rather different nature than Problem \ref{prob1}, since the first step of the proof of \cite{latala99} reduces to the case $\adimn=2$.  As discussed above, Problem \ref{prob1} cannot have such a reduction since symmetric strips do not minimize Problem \ref{prob1} for the measure constraint $c=1/2$.  One of the difficulties of Problem \ref{prob1} is dealing with higher-dimensional sets.

\subsection{Colding-Minicozzi Theory for Mean Curvature Flow}

 For any $x=(x_1,\ldots,x_{\adimn})$ and $y=(y_1,\ldots,y_{\adimn})\in \R^{\adimn}$ let $\langle x,y\rangle \colonequals\sum_{i=1}^{\adimn} x_{i} y_{i}$ denote their standard inner product, and let $\vnormt{x}\colonequals\sqrt{\langle x,x\rangle}$ denote the corresponding Euclidean norm.

The Colding-Minicozzi theory \cite{colding12a,colding12} focuses on orientable hypersurfaces $\Sigma$ with $\partial\Sigma=\emptyset$ satisfying
\begin{equation}\label{one2}
H(x)=\langle x, N(x)\rangle,\qquad\forall\,x\in\Sigma.
\end{equation}
Here $N(x)$ is the unit exterior pointing normal to $\Sigma$ at $x$, and $H(x)$ is the mean curvature of $\Sigma$ at $x$.  Below, we will often omit the $x$ arguments of $H,N$ for brevity.  $H(x)$ is the sum of principal curvatures of $\Sigma$ at $x$, or equivalently $H(x)=\mathrm{div}(N(x))$.
Here $H$ is chosen so that, if $r>0$, then the surface $r S^{\sdimn}$ satisfies $H(x)=n/r$ for all $x\in r S^{\sdimn}$.  A hypersurface $\Sigma$ satisfying \eqref{one2} is called a \embolden{self-shrinker}, since it is self-similar under the mean curvature flow.  Examples of self-shrinkers include a hyperplane through the origin, the sphere $\sqrt{\sdimn}S^{\sdimn}$, or more generally, round cylinders $\sqrt{k} S^{k}\times S^{\sdimn -k}$, where $0\leq k\leq n$, and also cones with zero mean curvature.

Also, self-shrinkers model singularities of mean curvature flow.  And $\Sigma$ is a self-shrinker if and only if it is a critical point of Gaussian surface area, in the following sense: for any differentiable $a\colon(-1,1)\to\R$ with $a(0)=1$, for any differentiable $b\colon(-1,1)\to\R^{\adimn}$ with $b(0)=0$, and for any normal variation $\{\Sigma_{s}\}_{s\in(-1,1)}$ of $\Sigma$ (as in \eqref{one4.8}), we have
$$\frac{\partial}{\partial s}\Big|_{s=0}\int_{\Sigma_{s}}(a(s))^{-\frac{\sdimn}{2}}\gamma_{\sdimn}((x-b(s))(a(s))^{-1/2})dx=0.$$
This equivalence was shown in \cite[Proposition 3.6]{colding12a}.  Put another way, self-shrinkers are critical points of Gaussian surface area, if we mod out by translations and dilations.  In this paper, we instead study critical points of the Gaussian surface area itself.  In this case, $\exists$ $\scon\in\R$ such that $H(x)=\langle x,N(x)\rangle+\scon$ for all $x\in\Sigma$ if and only if, for any normal variation $\{\Sigma_{s}\}_{s\in(-1,1)}$ of $\Sigma$,
$$\frac{\partial}{\partial s}\Big|_{s=0}\int_{\Sigma_{s}}\gamma_{\sdimn}(x)dx=0.$$
This fact is well-known, and reproven in Lemma \ref{varlem} below.  The special case $\scon=0$ recovers the self-shrinker equation \eqref{one2}.

The papers \cite{colding12a,colding12} led to several investigations.  On the one hand, it was conjectured that, among all compact hypersurfaces, the round sphere minimizes the entropy \eqref{one1}.  This conjecture was studied in \cite{colding12} and \cite{bernstein16} until its ultimate resolution in \cite{zhu16}.  One main technical contribution of \cite{zhu16} was to extend the Colding-Minicozzi theory to handle perturbations of cones.  Also, the main result of \cite[Theorem 0.12]{colding12a} shows that round cylinders are the only $C^{\infty}$ self-shrinkers that locally minimize the entropy \eqref{one1}.

On the other hand, it is natural to study a generalization of surfaces satisfying \eqref{one2}
\cite{mcgonagle15,guang15,cheng14,cheng15,cheng16,barchiesi16}.  That is, several papers have studied surfaces $\Sigma$ such that there exists $\scon\in\R$ such that
\begin{equation}\label{one3}
H(x)=\langle x, N(x)\rangle+\scon,\qquad\forall\,x\in\Sigma.
\end{equation}
Surfaces satisfying \eqref{one3} are called $\lambda$-hypersurfaces.

As we just mentioned, the condition \eqref{one3} is natural in the study of sets minimizing Gaussian surface area, since \eqref{one3} holds if and only if $\Sigma$ is a critical point of the Gaussian surface area (see Lemma \ref{varlem}).
%That is, if a set $\Omega\subset\R^{\adimn}$ minimizes the Gaussian surface area among all sets of fixed Gaussian volume, then $\Sigma\colonequals\partial\Omega$ satisfies \eqref{one3}.

A key aspect of the Colding-Minicozzi theory is the study of eigenfunctions of the differential operator $L$, defined for any $C^{\infty}$ function $f\colon\Sigma\to\R$ by
\begin{equation}\label{one3.5}
Lf(x)\colonequals \Delta f(x)-\langle x,\nabla f(x)\rangle+\vnormt{A_{x}}^{2}f(x)+f(x),\qquad\forall\,x\in\Sigma.
\end{equation}
Here $\Delta$ is the Laplacian associated to $\Sigma$, $\nabla$ is the gradient associated to $\Sigma$, $A=A_{x}$ is the second fundamental form of $\Sigma$ at $x$, and $\vnormt{A_{x}}^{2}$ is the sum of the squares of the entries of the matrix $A_{x}$.  Note that $L$ is an Ornstein-Uhlenbeck-type operator.  In particular, if $\Sigma$ is a hyperplane, then $A_{x}=0$ for all $x\in\Omega$, so $L$ is exactly the usual Ornstein-Uhlenbeck operator, plus the identity map.  (More detailed definitions will be given in Section \ref{seccurvature} below.)

The work \cite{colding12a} made the following crucial observation about the operator $L$.  If \eqref{one2} holds, then $H$ is an eigenfunction of $L$ with eigenvalue $2$:
\begin{equation}\label{one4}
\eqref{one2}\quad\Longrightarrow\quad LH=2H.
\end{equation}
(See \eqref{three9} below.  Note that our definition of $L$ differs from that of \cite{colding12a} since our Gaussian measure has a factor of $2$, whereas their Gaussian measure has a factor of $4$.  Consequently, their $L$ operator has different eigenvalues than ours.)  The Colding-Minicozzi theory can readily solve Problem \ref{prob1} in the special case that \eqref{one2} holds (which is more restrictive than \eqref{one3}).  For illustrative purposes, we now sketch this argument, which closely follows \cite[Theorem 4.30]{colding12a}.  In particular, we use the following key insights of \cite{colding12a}.

\begin{itemize}
\item $H$ is an eigenfunction of $L$ with eigenvalue $2$.  (That is, \eqref{one4} holds.)
\item The second variation formula for Gaussian surface area \eqref{nine2.9} is a quadratic form involving $L$.
\item If $H$ changes sign, then an eigenfunction of $L$ exists with eigenvalue larger than $2$.
\end{itemize}

\begin{prop}[\embolden{Special Case of Conjecture \ref{conj0}}]\label{prop2}
Let $\Omega\subset\R^{\adimn}$ minimize Problem \ref{prob1}.  By Lemma \ref{varlem} below, $\exists$ $\scon\in\R$ such that \eqref{one3} holds.  Assume $\scon=0$.  Assume also that $\Sigma\colonequals\partial\Omega$ is a compact, $C^{\infty}$ hypersurface.  Then $\exists$ $r>0$ such that $\partial\Omega=r S^{\sdimn}$.
\end{prop}
\begin{proof}
Let $H$ be the mean curvature of $\Sigma$.  If $H\geq0$, then Huisken's classification \cite{huisken90,huisken93} \cite[Theorem 0.17]{colding12a} of compact surfaces satisfying \eqref{one2} implies that $\Sigma$ is a round sphere ($\exists$ $r>0$ such that $\Sigma= r S^{\sdimn}$.).  So, we may assume that $H$ changes sign.  As noted in \eqref{one4}, $LH=2H$.  Since $H$ changes sign, $2$ is not the largest eigenvalue of $L$, by spectral theory \cite[Lemma 6.5]{zhu16} (e.g. using that $(L-\vnormt{A}^{2}-2)^{-1}$ is a compact operator).  That is, there exists a $C^{2}$ function $g\colon\Sigma\to\R$ and there exists $\pcon>2$ such that $Lg=\pcon g$.  Moreover, $g>0$ on $\Sigma$.  Since $g>0$ and $\Sigma=-\Sigma$, it follows by \eqref{one3.5} that $g(x)+g(-x)$ is an eigenfunction of $L$ with eigenvalue $\pcon$.  That is, we may assume that $g(x)=g(-x)$ for all $x\in\Sigma$.

Since $\Sigma$ is not a round sphere, it suffices to find a nearby hypersurface of smaller Gaussian surface area.  For any $C^{2}$ function $f\colon\Sigma\to\R$, and for any $s\in(-1,1)$, consider the hypersurface
\begin{equation}\label{one4.8}
\Sigma_{s}\colonequals\{x+sN(x)f(x)\colon x\in\Sigma\}.
\end{equation}

From the second variation formula, Lemma \ref{varlem2} below,
$$\frac{d^{2}}{d s^{2}}|_{s=0}\int_{\Sigma_{s}}\gamma_{\sdimn}(x)dx
=-\int_{\Sigma}f(x)Lf(x)\gamma_{\sdimn}(x)dx.$$

So, to complete the proof, it suffices by Lemma \ref{varlem2} to find a $C^{2}$ function $f$ such that
\begin{itemize}
\item $f(x)=f(-x)$ for all $x\in\Sigma$.  ($f$ preserves symmetry.)
\item $\int_{\Sigma}f(x)\gamma_{\sdimn}(x)dx=0$.  ($f$ preserves Gaussian volume.)
\item $\int_{\Sigma}f(x)Lf(x)\gamma_{\sdimn}(x)dx>0$.  ($f$ decreases Gaussian surface area.)
\end{itemize}
We choose $g$ as above so that $Lg=\pcon g$, $\pcon>2$ and so that $\int_{\Sigma}(H(x)+g(x))\gamma_{\sdimn}(x)dx=0$.  (Since $H$ changes sign and $g\geq0$, $g$ can satisfy the last equality by multiplying it by an appropriate constant.)  We then define $f\colonequals H+g$.  Then $f$ satisfies the first two properties.  So, it remains to show that $f$ satisfies the last property.  Note that, since $H$ and $g$ have different eigenvalues, they are orthogonal, i.e. $\int_{\Sigma}H(x)g(x)\gamma_{\sdimn}(x)dx=0$.  Therefore,
\begin{flalign*}
\int_{\Sigma}f(x)Lf(x)\gamma_{\sdimn}(x)dx
&\stackrel{\eqref{one4}}{=}\int_{\Sigma}(H(x)+g(x))(2H(x)+\pcon g(x))\gamma_{\sdimn}(x)dx\\
&=2\int_{\Sigma}(H(x))^{2}\gamma_{\sdimn}(x)dx
+\pcon\int_{\Sigma}(g(x))^{2}\gamma_{\sdimn}(x)dx
>0.
\end{flalign*}
(Since $H(x)=\langle x,N(x)\rangle$ for all $x\in\Sigma$, and $\Sigma$ is compact, both $\int_{\Sigma}(H(x))^{2}\gamma_{\sdimn}(x)dx$ and $\int_{\Sigma}H(x)\gamma_{\sdimn}(x)dx$ exist.)
\end{proof}
\begin{remark}
The case that $\partial\Omega$ is not compact can also be dealt with \cite[Lemmas 9.44 and 9.45]{colding12a}, \cite[Proposition 6.11]{zhu16}.  Instead of asserting the existence of $g$, one approximates $g$ by a sequence of Dirichlet eigenfunctions on the intersection of $\Sigma$ with large compact balls.  Since Proposition \ref{prop2} was presented only for illustrative purposes, and since the assumption \eqref{one2} is too restrictive to resolve Problem \ref{prob1}, we will not present the details.
\end{remark}

Unfortunately, the proof of Proposition \ref{prop2} does not extend to the more general assumption \eqref{one3}.  In order to attack Problem \ref{prob1}, we can only assume that \eqref{one3} holds, instead of the more restrictive \eqref{one2}.

Under the assumption of \eqref{one3}, the proof of Proposition \ref{prop2} breaks in at least two significant ways.  First, $H$ is no longer an eigenfunction of $L$ when \eqref{one3} holds with $\scon\neq0$ (see \eqref{three9} below).

Second, Huisken's classification no longer holds \cite{huisken90,huisken93}.  Indeed, it is known that, for every integer $m\geq3$, there exists $\scon=\scon_{m}<0$ and there exists a convex embedded curve $\Gamma_{m}\subset\R^{2}$ satisfying \eqref{one3} and such that $\Gamma_{m}$ has $m$-fold symmetry (and $\Gamma_{m_{1}}\neq\Gamma_{m_{2}}$ if $m_{1}\neq m_{2}$)  \cite[Theorem 1.3, Proposition 3.2]{chang17}.  Consequently, $\Gamma_{m}\times\R^{\sdimn-2}\subset\R^{\adimn}$ also satisfies \eqref{one3}.  That is, Huisken's classification cannot possibly hold, at least when $\scon<0$ in \eqref{one3}.

\subsection{Our contribution}

For any hypersurface $\Sigma\subset\R^{\adimn}$, we define (using \eqref{one3.5} and \eqref{c0def})
$$\pcon=\pcon(\Sigma)
\colonequals\sup_{f\in C_{0}^{\infty}(\Sigma)}\frac{\int_{\Sigma}fLf\gamma_{\sdimn}(x)dx}{\int_{\Sigma}f^{2}\gamma_{\sdimn}(x)dx}.$$
The quantity $\pcon(\Sigma)$ is denoted $-\mu_{1}$ in \cite[Corollary 5.15]{colding12a}.

In Section \ref{secdif} below, we show that Huisken's classification does actually hold for surfaces satisfying \eqref{one3} if $\scon>0$, if the surface $\Sigma$ encloses a convex region, and if $\pcon(\Sigma)<\infty$.  Due to the following Lemma (see Lemma \ref{lemma95.1} below), we may always assume that $\pcon(\partial\Omega)<\infty$.

\begin{lemma}\label{lemma95}
If $\Omega\subset\R^{\adimn}$ minimizes Problem \ref{prob1}, then $\pcon(\partial\Omega)<\infty$.
\end{lemma}

The assumption that $\pcon(\partial\Omega)<\infty$ is similar to assuming that $\int_{\partial\Omega}\vnormt{A}^{2}\gamma_{\sdimn}(x)dx<\infty$.

\begin{theorem}[\embolden{Huisken-type classification, $\scon>0$}]\label{thm1}
Let $\Omega\subset\R^{\adimn}$ be a convex set.  Let $\scon>0$.  Assume that $H(x)=\langle x,N(x)\rangle+\scon$ for all $x\in\partial\Omega$.  Assume $\pcon(\partial\Omega)<\infty$.  Then, after rotating $\Omega$, $\exists$ $r>0$ and $\exists$ $0\leq k\leq\sdimn$ such that $\partial\Omega= r S^{k}\times\R^{\sdimn-k}$.
\end{theorem}

Related to Huisken's classification \cite{huisken90,huisken93} \cite[Theorem 0.17]{colding12a} are Bernstein theorems.  If a hypersurface $\Sigma$ satisfies \eqref{one2} and $\Sigma$ can be written as the graph of a function, then $\Sigma$ is a hyperplane \cite{ecker89} \cite{wang11}.  Also, if a hypersurface $\Sigma$ satisfies \eqref{one3}, if $\Sigma$ has polynomial volume growth and if $\Sigma$ be written as the graph of a function, then $\Sigma$ is a hyperplane \cite[Theorem 1.6]{guang15} \cite[Theorem 1.3]{cheng14}.  In particular, if $\Omega=-\Omega$, if $\partial\Omega$ satisfies \eqref{one3}, and if $\Omega$ can be separated by a hyperplane into two sets, each of which is the graph of a function, then $\partial\Omega$ must consists of two parallel hyperplanes.  In this sense, the symmetric strip separated by two parallel hyperplanes (or its complement) are ``isolated critical points'' in Problem \ref{prob1}.

Due to the Bernstein-type theorems of \cite{guang15,cheng14}, the main difficulty of Problem \ref{prob1} occurs when $\Sigma$ is not the graph of a function.  Also, by Theorem \ref{thm1}, Lemma \ref{lemma95} and Lemma \ref{varlem} below, in order to solve the convex case of Problem \ref{prob1}, it suffices to restrict to surfaces $\Sigma$ such that there exists $\scon<0$ and such that $H(x)=\langle x,N(x)\rangle+\scon$, for all $x\in\Sigma$.  As discussed above, the case $\scon<0$ is most interesting, since a Huisken-type classification cannot possibly hold when $\scon<0$.  To deal with the case $\scon<0$, we use second variation arguments, as in \cite{colding12a,colding12}.

We begin by using the mean curvature minus its mean in the second variation formula for Gaussian surface area.
%We let $\pcon$ be the largest eigenvalue of $L$, so that
%$$\pcon\colonequals\sup_{f\in C^{\infty}(\Sigma)}\frac{\int_{\Sigma}fLf\gamma_{\sdimn}(x)dx}{\int_{\Sigma}f^{2}\gamma_{\sdimn}(x)dx}.$$
%If $g$ is the eigenfunction corresponding to $\pcon$, we then subtract a constant from $g$ to get a mean zero function, which we then plug into the second variation formula for %Gaussian surface area, yielding the following.
%
\begin{theorem}[\embolden{Second Variation Using an Eigenfunction of $L$, $\scon<0$}]\label{thm2}
Let $\Omega$ minimize Problem \ref{prob1} and let $\Sigma\colonequals\partial\Omega$.  Then by Lemma \ref{varlem}, $\exists$ $\scon\in\R$ such that $H(x)=\langle x,N(x)\rangle+\scon$, for any $x\in\Sigma$.  Assume $\scon<0$.  If
$$\int_{\Sigma}(\vnormt{A_{x}}^{2}-1)\gamma_{\sdimn}(x)dx>0,$$
then, after rotating $\Omega$, $\exists$ $r>0$ and $0\leq k\leq \sdimn$ so that $\Sigma=r S^{k}\times\R^{\sdimn-k}$.
\end{theorem}

Theorem \ref{thm2} follows from a slightly more general inequality in Lemma \ref{lemma74} below.  In the case that $\scon<0$ and $\Omega$ is convex, the largest eigenvalue of $L$ is at most $2$, as shown in Lemma \ref{lemma39}.  With an eigenvalue bound smaller than $2$, Lemma \ref{lemma74} would improve Theorem \ref{thm2}.

%
%As shown in Lemma \ref{lemma45} below, if $\Omega$ is convex and if $\scon<0$, then $\pcon\leq2$.  So, Theorem \ref{thm2} has the following Corollary.
%
%\begin{theorem}[\embolden{Second Variation Using $H$}]\label{thm2}
%Let $\Omega$ minimize Problem \ref{prob1} and let $\Sigma\colonequals\partial\Omega$.  Let $b\in\R$ such that $\int_{\Sigma}(H(x)+b)\gamma_{\sdimn}(x)dx=0$.  If
%$$
%\int_{\Sigma}\Big(2(H(x)+b)^{2}+(b+\scon)^{2}[\vnormt{A}^{2}-1]\Big)\gamma_{\sdimn}(x)dx
%-\scon\int_{\Omega}(\sdimn-\vnormt{x}^{2})\gamma_{\sdimn}(x)dx\int_{\Sigma}\gamma_{\sdimn}(y)dy
%>0,
%$$
%then, after rotating $\Omega$, $\exists$ $r>0$ and $\exists$ $0\leq k\leq \sdimn$ so that $\Sigma=r S^{k}\times\R^{\sdimn-k}$.
%\end{theorem}
%
%If $\Omega$ is convex and if $\scon<0$, then the right term is nonnegative, so we get the following.
%
%\begin{cor}[\embolden{Second Variation Using $H$}]\label{cor1}
%Let $\Omega$ minimize Problem \ref{prob1} and let $\Sigma\colonequals\partial\Omega$, so that $\exists$ $\scon\in\R$ such that $H(x)=\langle x,N(x)\rangle+\scon$ for all $x\in\Sigma$.  Assume that $\scon<0$.  If
%$$\int_{\Sigma}(\vnormt{A_{x}}^{2}-1)\gamma_{\sdimn}(x)dx>0,$$
%then, after rotating $\Omega$, $\exists$ $r>0$ and $\exists$ $0\leq k\leq \sdimn$ so that $\Sigma=r S^{k}\times\R^{\sdimn-k}$.
%\end{cor}

To handle the case when the average curvature of $\Sigma$ is less than $1$, we use our intuition about the sphere itself.  On the sphere, the mean zero symmetric eigenfunctions of $L$ which maximize the second variation of Gaussian surface area are degree two homogeneous spherical harmonics.  This was observed in \cite{lamanna17}.  A similar observation was made in the context of noise stability in \cite{heilman15}.  In fact, if $v,w\in S^{\sdimn}$, if $\langle v,w\rangle=0$ and if $\Sigma= S^{\sdimn}$ then $\langle v,N\rangle\langle w,N\rangle$ is an eigenfunction of $L$.  So, intuitively, if $\exists$ $\scon\in\R$ such that $H(x)=\langle x,N(x)\rangle+\scon$, then $\langle v,N\rangle\langle w,N\rangle$ should also be an eigenfunction of $L$.  Unfortunately, this does not seem to be true.  Nevertheless, if we average over all possible choices of $v,w\in S^{\sdimn}$, then we can obtain a good bound in the second variation formula.  And then there must exist $v,w\in S^{\sdimn}$ whose second variation exceeds this average value.

If $y\in\Sigma$, then we define $\vnorm{A_{y}}_{2\to 2}^{2}\colonequals\sup_{v\in S^{\sdimn}}\vnormt{A_{y}v}^{2}$ to be the $\ell_{2}$ operator norm of $A_{y}$.  Also, $\Pi_{y}\colon\R^{\adimn}\to\R^{\sdimn}$ denotes the linear projection onto the tangent space of $y\in\Sigma$.  (So $\vnormt{\Pi_{y}v}^{2}=1-\langle N(y),v/\vnormt{v}\rangle^{2}$ for any $v\in\R^{\adimn}\setminus\{0\}$.)

As noted in Proposition \ref{prop2}, Problem \ref{prob1} reduces to finding functions $f\colon\Sigma\to\R$ such that $\int_{\Sigma} fLf \gamma_{\sdimn}(x)dx$ is as large as possible.

\begin{theorem}[\embolden{Second Variation Using a Random Bilinear Function}]\label{cor1.5}
Let $\Sigma\subset\R^{\adimn}$ be an orientable hypersurface with $\partial\Sigma=\emptyset$.  Suppose $\exists$ $\scon\in\R$ such that $H(x)=\langle x,N(x)\rangle+\scon$ for all $x\in\Sigma$.  Let $p\colonequals\int_{\Sigma}\gamma_{\sdimn}(x)dx$.

There exists $v,w\in S^{\sdimn}$ so that, if $m\colonequals\frac{1}{p}\int_{\Sigma}\langle v,N\rangle\langle w,N\rangle\gamma_{\sdimn}(y)dy$, we have
\begin{flalign*}
&(\adimn)^{2}\int_{\Sigma}(\langle v,N\rangle\langle w,N\rangle-m)L(\langle v,N\rangle\langle w,N\rangle-m)\gamma_{\sdimn}(x)dx\\
&\geq\frac{1}{p^{2}}\int_{\Sigma\times\Sigma\times\Sigma}(1-\vnormt{A_{x}}^{2}-2\vnorm{A_{y}}_{2\to 2}^{2})\vnormt{\Pi_{y}(N(z))}^{2}
\gamma_{\sdimn}(x)\gamma_{\sdimn}(y)\gamma_{\sdimn}(z)dxdydz.
\end{flalign*}
\end{theorem}
Note that convexity is not assumed in Theorem \ref{cor1.5}.

Theorem \ref{cor1.5} actually follows from a slightly more general statement, Lemma \ref{lemma85} below.

 Theorem \ref{cor1.5} is sharp for spheres, as observed by \cite{lamanna17}.  If $r>0$, and if $\Sigma=r S^{\sdimn}$, then $\vnormt{A_{x}}^{2}=\sdimn/r^{2}$ and $\vnorm{A_{y}}_{2\to2}^{2}=1/r^{2}$, so
$$1-\vnormt{A_{x}}^{2}-2\vnorm{A_{y}}_{2\to2}^{2}=\frac{r^{2}-\sdimn-2}{r^{2}}.$$
If $v,w\in S^{\sdimn}$ satisfy $\langle v,w\rangle=0$, then $m=0$ and $\int_{\Sigma}\langle v,N\rangle\langle w,N\rangle L(\langle v,N\rangle\langle w,N\rangle)\gamma_{\sdimn}(x)dx\geq0$ if and only if $r\geq\sqrt{\sdimn+2}$ \cite[Proposition 1]{lamanna17}.

Since Theorem \ref{cor1.5} gives a bound on the second variation, Theorem \ref{cor1.5} implies the following.

\begin{cor}[\embolden{Second Variation Using a Random Bilinear Function}]\label{cor2}
Let $\Omega$ minimize Problem \ref{prob1} and let $\Sigma\colonequals\partial\Omega$, so that $\exists$ $\scon\in\R$ such that $H(x)=\langle x,N(x)\rangle+\scon$ for all $x\in\Sigma$.  If
$$
\int_{\Sigma\times\Sigma\times\Sigma}(1-\vnormt{A_{x}}^{2}-2\vnorm{A_{y}}_{2\to2}^{2})\vnormt{\Pi_{y}(N(z))}^{2}
\gamma_{\sdimn}(x)\gamma_{\sdimn}(y)\gamma_{\sdimn}(z)dxdydz>0,$$
then, after rotating $\Omega$, $\exists$ $r>0$ and $\exists$ $0\leq k\leq \sdimn$ so that $\Sigma=r S^{k}\times\R^{\sdimn-k}$.
\end{cor}

The combination of Remark \ref{crk}, Theorems \ref{thm1} and \ref{thm2} and Corollary \ref{cor2} implies the following.

\begin{theorem}[\embolden{Main Result}]\label{mainthm}
Let $\Omega$ minimize Problem \ref{prob1} and let $\Sigma\colonequals\partial\Omega$.  Assume that $\Omega$ or $\Omega^{c}$ is convex.  If
$$\int_{\Sigma}(\vnormt{A_{x}}^{2}-1)\gamma_{\sdimn}(x)dx>0,$$
or
$$\int_{\Sigma}(\vnormt{A_{x}}^{2}-1+2\sup_{y\in\Sigma}\vnormt{A_{y}}_{2\to 2}^{2})\gamma_{\sdimn}(x)dx<0,$$
then, after rotating $\Omega$, $\exists$ $r>0$ and $\exists$ $0\leq k\leq \sdimn$ so that $\Sigma=r S^{k}\times\R^{\sdimn-k}$.
\end{theorem}
So, except for the case that the average value of $\vnormt{A}^{2}$ is slightly less than $1$, we resolve the convex case of Barthe's Conjecture \ref{conj0}.

In Section \ref{secdil} we adapt an argument of \cite{colding12a} that allows the computation of the second variation of Gaussian volume preserving normal variations, which simultaneously can dilate the hypersurface $\Sigma$.  When we use the function $H-\scon$ in this second variation formula, we get zero.  This suggests the intriguing possibility that the fourth variation of $H-\scon$ could help to solve Problem \ref{prob1}.  Instead of embarking on a rather technical enterprise of computing Gaussian volume preserving fourth variations, we instead put the function $H-\scon+t$, $t\in\R$ into this second variation formula, and we then differentiate twice in $t$.  We then arrive at the following interesting inequality.

\begin{theorem}\label{lastthm}
Let $\Omega$ minimize Problem \ref{prob1} and let $\Sigma\colonequals\partial\Omega$.  Assume also that $\Sigma$ is a compact, $C^{\infty}$ hypersurface and $\Omega$ is convex.  Then
$$
\int_{\Sigma}\Big(-\vnormt{A_{x}}^{2}
+H(x)\frac{\int_{\Sigma}\gamma_{\sdimn}(z)dz}{\int_{\Sigma}\langle y,N\rangle\gamma_{\sdimn}(y)dy}
\Big)\gamma_{\sdimn}(x)dx\geq0.
$$
\end{theorem}

This inequality is rather interesting since it is equal to zero exactly when $\Sigma=r S^{\sdimn}$, for any $r>0$, since then $\vnormt{A_{x}}^{2}=\sdimn/r^{2}$, $H(x)=\sdimn/r$, and $\langle x,N\rangle=r$ for all $x\in rS^{\sdimn}$.  So, one might speculate that round spheres are the only compact $C^{\infty}$ hypersurfaces, where this quantity is nonnegative, and where $\exists$ $\scon\in\R$ such that $H(x)=\langle x,N(x)\rangle+\scon$ for all $x\in\Sigma$

Finally, we show that Theorem \ref{mainthm} can be partially generalized to the non-convex case.

\begin{theorem}[\embolden{Weak Main Result, Without Convexity}]\label{mainthm2}
Let $\Omega$ minimize Problem \ref{prob1} and let $\Sigma\colonequals\partial\Omega$.  From Lemma \ref{varlem} below, $\exists$ $\scon\in\R$ such that $H(x)=\langle x,N(x)\rangle+\scon$, $\forall$ $x\in\Sigma$.  If
$$\int_{\Sigma}(\vnormt{A_{x}}^{2}-1)\gamma_{\sdimn}(x)dx>0\,\,\mbox{and}\,\,-\scon\int_{\Sigma}\langle x,N(x)\rangle\gamma_{\sdimn}(x)dx>0,$$
or
$$\int_{\Sigma}(\vnormt{A_{x}}^{2}-1+2\sup_{y\in\Sigma}\vnormt{A_{y}}_{2\to 2}^{2})\gamma_{\sdimn}(x)dx<0,$$
then, after rotating $\Omega$, $\exists$ $r>0$ and $\exists$ $0\leq k\leq \sdimn$ so that $\Sigma=r S^{k}\times\R^{\sdimn-k}$.
\end{theorem}

Conjecture \ref{conj0} and our results for it only specify that some cylinder minimizes the Gaussian surface area among all sets of fixed Gaussian volume.  That is, the dimension of the cylinder that minimizes Gaussian surface area is not specified.  Upon seeing our initial preprint, Frank Morgan suggested the following strengthened version of Conjecture \ref{conj0}, which appears to be verified by numerical computations, at least when $n\leq6$.  In Conjecture \ref{conj0n} below, the cylinder of minimal Gaussian surface area is identified.

\begin{conj}[\embolden{Morgan's Conjecture}]\label{conj0n}
There exists a sequence of real numbers $1=a_{0}>a_{1}>a_{2}>\cdots>1/2$ such that $(1/2,1]=\cup_{k=0}^{\infty}(a_{k+1},a_{k}]$ and such that the following holds.  Fix $1/2<c<1$.  Let $n\geq0$.  Suppose $\Omega_{\adimn}\subset\R^{\adimn}$ minimizes
$$\int_{\partial \Omega}\gamma_{\sdimn}(x)dx$$
over all subsets $\Omega\subset\R^{\adimn}$ satisfying $\Omega=-\Omega$ and $\gamma_{\adimn}(\Omega)=c$.  Let $k$ be the unique nonnegative integer such that $c\in(a_{k+1},a_{k}]$.  Then,
$$\int_{\partial\Omega_{k+1}}\gamma_{k}(x)dx=\min_{n\geq0}\int_{\partial\Omega_{\adimn}}\gamma_{\sdimn}(x)dx.$$
Also there exists $r=r(c,k)>0$ such that
$$\Omega_{k+1}=\{(x_{1},\ldots,x_{k+1})\in\R^{k+1}\colon x_{1}^{2}+\cdots+x_{k+1}^{2}\leq r^{2}\}.$$
That is, the minimum Gaussian surface area of all (measurable) sets of Gaussian measure $1/2<c<1$ occurs for the ball in $\R^{k+1}$ centered at the origin when $c\in(a_{k+1},a_{k}]$.  By Remark \ref{crk}, the above statement holds for any $0<c<1/2$ by taking complements.  That is, the minimum Gaussian surface area of all (measurable) sets of Gaussian measure $0<c<1/2$ occurs for the complement of the ball in $\R^{k}$ centered at the origin when $1-c\in(a_{k+1},a_{k}]$.

Lastly, in the case $c=1/2$, the ball (or its complement) minimizes Gaussian surface area, asymptotically as $n\to\infty$:
$$
\inf_{n\geq0}\int_{\partial\Omega_{\adimn}}\gamma_{\sdimn}(x)dx
=\lim_{n\to\infty}\int_{\partial\Omega_{\adimn}}\gamma_{\sdimn}(x)dx
=\lim_{n\to\infty}\int_{\sqrt{\adimn}S^{\sdimn}}\gamma_{\sdimn}(x)dx
=\frac{1}{\sqrt{\pi}}
\approx.56419\ldots .
$$
\end{conj}

At present there seems to be no sensible way to analytically find the numbers $a_{1},a_{2},\ldots$.
%try to draw a_1, a_2,.. in the picture

\begin{figure}[h!]
\centering
\def\svgwidth{.75\textwidth}
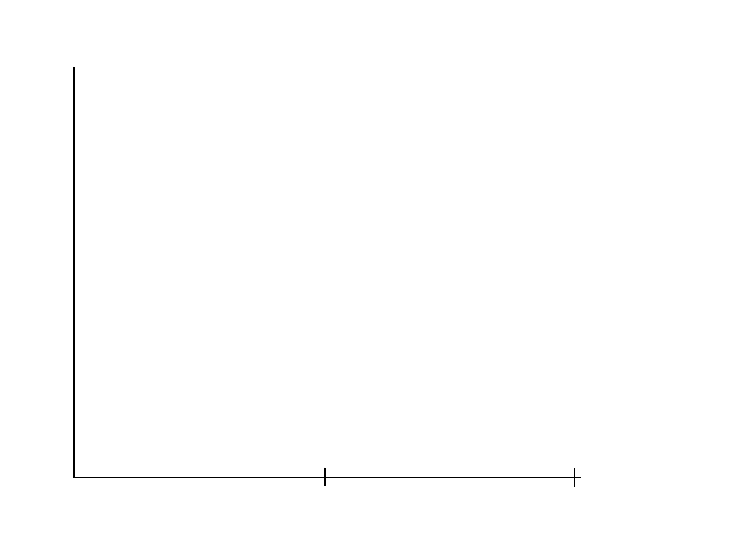
\caption{Conjecture \ref{conj0n} of Morgan.}\label{bloop}
\end{figure}

\subsection{Organization}

\begin{itemize}
\item Preliminary details are covered in Sections \ref{secpre} through \ref{seceig}.
\item  A rather technical section on curvature bounds is given in Section \ref{seccurv}.
\item Theorem \ref{thm1} is proven in Section \ref{secdif}.  Theorem \ref{thm2} is proven in Section \ref{sech}.
\item Theorem \ref{cor1.5}, Corollary \ref{cor2} and Theorem \ref{mainthm} are proven in Section \ref{secrand}
\item Theorem \ref{lastthm} is proven in Section \ref{secdil}.
\item Theorem \ref{mainthm2} is proven in Section \ref{secal}.
\end{itemize}

\subsection{Discussion}

As in the proof of Proposition \ref{prop2}, the proof of Theorem \ref{mainthm} tries to find a function $f\colon\Sigma\to\R$ with $f(x)=f(-x)$, $\forall$ $x\in\Sigma$, $\int_{\Sigma}f(x)\gamma_{\sdimn}(x)dx=0$, and such that $\int_{\Sigma}f(x)Lf(x)\gamma_{\sdimn}(x)dx>0$.  In Proposition \ref{prop2}, this is achieved by letting $f$ be the sum of two distinct eigenfunctions of $L$ with positive eigenvalues.  It could occur that $L$ has only one symmetric eigenfunction with a positive eigenvalue, but still we could find a symmetric $f$ with zero mean and $\int_{\Sigma}f(x)Lf(x)\gamma_{\sdimn}(x)dx>0$.  It would be interesting to explore this possibility, since the proof strategy of Proposition \ref{prop2} fails in this case.  And indeed, in the proof of the Main Theorem, we have to compensate for the fact that we cannot find explicit eigenfunctions of $L$.  Also, it would be interesting to see if any $\Omega$ exists that evades all constraints put upon it by the results in this work (e.g. by Theorems \ref{mainthm}, \ref{mainthm2} or Corollary \ref{cor11}).  To the author's knowledge, no such $\Omega$ is known to exist.

\section{Preliminaries}\label{secpre}

We say that $\Sigma\subset\R^{\adimn}$ is an $\sdimn$-dimensional $C^{\infty}$ manifold with boundary if $\Sigma$ can be locally written as the graph of a $C^{\infty}$ function.

%\snote{Define $C_{0}^{\infty}$ to avoid singularities?  Also recall this definition in the curvature bound section?}

For any $(\adimn)$-dimensional $C^{\infty}$ manifold $\Omega\subset\R^{\adimn}$ with boundary, we denote
\begin{equation}\label{c0def}
\begin{aligned}
C_{0}^{\infty}(\Omega;\R^{\adimn})
&\colonequals\{f\colon \Omega\to\R^{\adimn}\colon f\in C^{\infty}(\Omega;\R^{\adimn}),\, f(\partial \Omega)=0,\\
&\qquad\qquad\qquad\exists\,r>0,\,f(\Omega\cap(B(0,r))^{c})=0\}.
\end{aligned}
\end{equation}
We also denote $C_{0}^{\infty}(\Omega)\colonequals C_{0}^{\infty}(\Omega;\R)$.  We let $\mathrm{div}$ denote the divergence of a vector field in $\R^{\adimn}$.  For any $r>0$ and for any $x\in\R^{\adimn}$, we let $B(x,r)\colonequals\{y\in\R^{\adimn}\colon\vnormt{x-y}\leq r\}$ be the closed Euclidean ball of radius $r$ centered at $x\in\R^{\adimn}$.

\begin{definition}[\embolden{Reduced Boundary}]
A measurable set $\Omega\subset\R^{\adimn}$ has \embolden{locally finite surface area} if, for any $r>0$,
$$\sup\left\{\int_{\Omega}\mathrm{div}(X(x))dx\colon X\in C_{0}^{\infty}(B(0,r),\R^{\adimn}),\, \sup_{x\in\R^{\adimn}}\vnormt{X(x)}\leq1\right\}<\infty.$$
Equivalently, $\Omega$ has locally finite surface area if $\nabla 1_{\Omega}$ is a vector-valued Radon measure such that, for any $x\in\R^{\adimn}$, the total variation
$$
\vnormt{\nabla 1_{\Omega}}(B(x,1))
\colonequals\sup_{\substack{\mathrm{partitions}\\ C_{1},\ldots,C_{m}\,\mathrm{of}\,B(x,1) \\ m\geq1}}\sum_{i=1}^{m}\vnormt{\nabla 1_{\Omega}(C_{i})}
$$
is finite \cite{cicalese12}.

If $\Omega\subset\R^{\adimn}$ has locally finite surface area, we define the \embolden{reduced boundary} $\redA$ of $\Omega$ to be the set of points $x\in\R^{\adimn}$ such that
$$N(x)\colonequals-\lim_{r\to0^{+}}\frac{\nabla 1_{\Omega}(B(x,r))}{\vnormt{\nabla 1_{\Omega}}(B(x,r))}$$
exists, and it is exactly one element of $S^{\sdimn}$.
\end{definition}
For more background on the reduced boundary and its regularity, we refer to the discussion in Section 2 of \cite{barchiesi16}, \cite{ambrosio00} and \cite{maggi12}.

The following argument is essentially identical to \cite[Proposition 1]{barchiesi16}, so we omit the proof.

\begin{lemma}[\embolden{Existence}]\label{lemma51p}
There exists a set $\Omega\subset\R^{\adimn}$ minimizing Problem \ref{prob1}.
\end{lemma}
%\begin{proof}
%Consider a sequence $\Omega_{1},\Omega_{2},\ldots\subset\R^{\adimn}$ of sets of locally finite surface area such that $\lim_{m\to\infty} \int_{\partial^{*}\Omega_{m}}$ is equal to the infimum of the Gaussian surface area over all sets of locally finite surface area.  Then, for any bounded open set $C$, $\sup_{m\geq1}\int_{(\partial^{*} \Omega_{m})\cap C}dx<\infty$, the compactness theorem for BV functions \cite[Theorem 3.23]{ambrosio00} ensures the existence of a Borel set $\Omega\subset\R^{\adimn}$ such that, there exists a subsequence of positive integers $m_{1}<m_{2}<\cdots$ such that $1_{\Omega_{m_{1}}},1_{\Omega_{m_{2}}},\ldots$ converges to $1_{\Omega}$ locally in $L_{1}(\R^{\adimn})$.
%From the Divergence Theorem,
%\begin{flalign*}
%\int_{\redA}\gamma_{\sdimn}(x)dx
%&=\sqrt{2\pi}\sup\Big\{\int_{\Omega}[\mathrm{div}(X(x))-\langle X(x),x\rangle]\gamma_{\sdimn}(x)dx\\
%&\quad\qquad\qquad\qquad\qquad\colon X\in C_{0}^{\infty}(B(0,r),\R^{\adimn}),\, r>0,\,\sup_{x\in\R^{\adimn}}\vnormt{X(x)}\leq1\Big\}.
%\end{flalign*}
%Therefore, $\Omega\mapsto\int_{\redA}\gamma_{\adimn}(x)dx$ is lower
%semicontinuous with respect to the local $L_{1}$ convergence of sets.  That is, $\int_{\redA}\gamma_{\adimn}(x)dx \leq\liminf_{p\to\infty}\int_{\redb \Omega_{m_{p}}}\gamma_{\sdimn}(x)dx$.  Therefore, $\int_{\redA}\gamma_{\adimn}(x)dx\leq \int_{\partial^{*}D}\gamma_{\adimn}(x)dx$ for any set $D\subset\R^{\adimn}$ of locally finite surface area, as desired.
%\end{proof}

\subsection{Submanifold Curvature}\label{seccurvature}

Here we cover some basic definitions from differential geometry of submanifolds of Euclidean space.

Let $\nabla$ denote the standard Euclidean connection, so that if $X,Y\in C_{0}^{\infty}(\R^{\adimn},\R^{\adimn})$, if $Y=(Y_{1},\ldots,Y_{\adimn})$, and if $u_{1},\ldots,u_{\adimn}$ is the standard basis of $\R^{\adimn}$, then $\overline{\nabla}_{X}Y\colonequals\sum_{i=1}^{\adimn}(X (Y_{i}))u_{i}$.  Let $N$ be the outward pointing unit normal vector of an $\sdimn$-dimensional hypersurface $\Sigma\subset\R^{\adimn}$.  For any vector $x\in\Sigma$, we write $x=x^{T}+x^{N}$, so that $x^{N}\colonequals\langle x,N\rangle N$ is the normal component of $x$, and $x^{T}$ is the tangential component of $x\in\Sigma$.  We let $\nabla^{\Sigma}\colonequals(\nabla)^{T}$ denote the tangential component of the Euclidean connection.

Let $e_{1},\ldots,e_{\sdimn}$ be an orthonormal frame of $\Sigma\subset\R^{\adimn}$.  That is, for a fixed $x\in\Sigma$, there exists a neighborhood $U$ of $x$ such that $e_{1},\ldots,e_{\sdimn}$ is an orthonormal basis for the tangent space of $\Sigma$, for every point in $U$ \cite[Proposition 11.17]{lee03}.

Define the \embolden{mean curvature}
\begin{equation}\label{three0.5}
H\colonequals\mathrm{div}(N)=\sum_{i=1}^{\sdimn}\langle\nabla_{e_{i}}N,e_{i}\rangle.
\end{equation}

Define the \embolden{second fundamental form} $A=(a_{ij})_{1\leq i,j\leq\sdimn}$ so that
\begin{equation}\label{three1}
a_{ij}=\langle\nabla_{e_{i}}e_{j},N\rangle,\qquad\forall\,1\leq i,j\leq \sdimn.
\end{equation}
Compatibility of the Riemannian metric says $a_{ij}=\langle\nabla_{e_{i}}e_{j},N\rangle=-\langle e_{j},\nabla_{e_{i}}N\rangle+ e_{i}\langle N,e_{j}\rangle=-\langle e_{j},\nabla_{e_{i}}N\rangle$, $\forall$ $1\leq i,j\leq \sdimn$.  So, multiplying by $e_{j}$ and summing this equality over $j$ gives
\begin{equation}\label{three2}
\nabla_{e_{i}}N=-\sum_{j=1}^{\sdimn}a_{ij}e_{j},\qquad\forall\,1\leq i\leq \sdimn.
\end{equation}

%Given a vector $v\in\R^{\adimn}$, define $v^{N}\colonequals \langle v,N\rangle N$, and define $v^{T}\colonequals v-v^{N}$, so that $v=v^{N}+v^{T}$.

Using $\langle\nabla_{N}N,N\rangle=0$,
\begin{equation}\label{three4}
H\stackrel{\eqref{three0.5}}{=}\sum_{i=1}^{\sdimn}\langle \nabla_{e_{i}} N,e_{i}\rangle
\stackrel{\eqref{three2}}{=}-\sum_{i=1}^{\sdimn}a_{ii}.
\end{equation}

\subsection{First and Second Variation}

We will apply the calculus of variations to solve Problem \ref{prob1}. Here we present the rudiments of the calculus of variations.

The results of this section are well known to experts in the calculus of variations, and many of these results were re-proven in \cite{barchiesi16}.
%We provide this section for the sake of completeness and for the sake of providing more detail than \cite{barchiesi16} in some cases, but this section could be skipped on a first reading.

Let $\Omega\subset\R^{\adimn}$ be an $(\adimn)$-dimensional $C^{2}$ submanifold with reduced boundary $\Sigma\colonequals\redA$.  Let $N\colon\redA\to S^{\sdimn}$ denote the unit exterior normal to $\redA$.  Let $X\colon\R^{\adimn}\to\R^{\adimn}$ be a vector field.  Unless otherwise stated, we always assume that $X(x)$ is parallel to $N(x)$ for all $x\in\redA$.  That is,
\begin{equation}\label{nine2.4}
X(x)=\langle X(x),N(x)\rangle N(x),\qquad\forall\, x\in\redA.
\end{equation}
Let $\mathrm{div}$ denote the divergence of a vector field.  We write $X$ in its components as $X=(X_{1},\ldots,X_{\adimn})$, so that $\mathrm{div}X=\sum_{i=1}^{\adimn}\frac{\partial}{\partial x_{i}}X_{i}$.  Let $\Psi\colon\R^{\adimn}\times(-1,1)\to\R^{\adimn}$ such that
\begin{equation}\label{nine2.3}
\Psi(x,0)=x,\qquad\qquad\frac{d}{ds}\Psi(x,s)=X(\Psi(x,s)),\quad\forall\,x\in\R^{\adimn},\,s\in(-1,1).
\end{equation}
For any $s\in(-1,1)$, let $\Omega_{s}\colonequals\Psi(\Omega,s)$.  Note that $\Omega_{0}=\Omega$.  Let $\Sigma_{s}\colonequals\redb\Omega_{s}$.
\begin{definition}
We call $\{\Omega_{s}\}_{s\in(-1,1)}$ as defined above a \embolden{normal variation} of $\Omega\subset\R^{\adimn}$.  We also call $\{\Sigma_{s}\}_{s\in(-1,1)}$ a \embolden{normal variation} of $\Sigma=\partial\Omega$.
\end{definition}

\begin{lemma}[\embolden{First Variation}]\label{lemma10}  Let $X\in C_{0}^{\infty}(\R^{\adimn},\R^{\adimn})$.  Let $f(x)=\langle X(x),N(x)\rangle$ for any $x\in\redA$.  Then
\begin{equation}\label{nine1}
\frac{d}{ds}|_{s=0}\gamma_{\adimn}(\Omega_{s})=\int_{\redA}f(x) \gamma_{\adimn}(x)dx.
\end{equation}
\begin{equation}\label{nine2}
\frac{d}{ds}|_{s=0}\int_{\redb \Omega_{s}}\gamma_{\sdimn}(x)dx
=\int_{\redA}(H(x)-\langle N(x),x\rangle)f(x)\gamma_{\sdimn}(x)dx.
\end{equation}
\end{lemma}
 Lemma \ref{lemma41} below (with $h=0$) implies \eqref{nine1} and Lemma \ref{lemma107} below (with $t_{s}=0$ for all $s\in(-1,1)$) implies \eqref{nine2}.

\begin{lemma}[\embolden{Second Variation}]\label{lemma10.5}  Let $X\in C_{0}^{\infty}(\R^{\adimn},\R^{\adimn})$.  Let $f(x)=\langle X(x),N(x)\rangle$ for all $x\in\redA$.  Then
\begin{equation}\label{nine1.9}
\frac{d^{2}}{ds^{2}}|_{s=0}\gamma_{\adimn}(\Omega_{s})=\int_{\redA}f(\nabla_{N}f+f(H-\langle x,N\rangle))\gamma_{\adimn}(x)dx.
\end{equation}
\begin{equation}\label{nine2.9}
\frac{d^{2}}{ds^{2}}|_{s=0}\int_{\redb \Omega_{s}}\gamma_{\sdimn}(x)dx
=\int_{\redA}\Big(-fLf+(H-\langle x,N\rangle)(f\nabla_{N}f+f^{2}(H-\langle x,N\rangle))\Big)\gamma_{\sdimn}(x)dx.
\end{equation}
\end{lemma}

Lemma \ref{lemma42} (with $h=h'=0$) implies \eqref{nine1.9} and Lemma \ref{lemma40} (with $h=h'=0$ and $f'=f\nabla_{N}f$ by Lemma \ref{lemma5}) implies \eqref{nine2.9}.

\section{Variations and Regularity}

In this section, we show that a minimizer of Problem \ref{prob1} exists, and the boundary of the minimizer is $C^{\infty}$ except on a set of Hausdorff dimension at most $n-7$.

Much of this section is a modification of corresponding parts of \cite{barchiesi16}.

Unless otherwise stated, all sets $\Omega\subset\R^{\adimn}$ below are assumed to be measurable sets of locally finite surface area, and such that the Gaussian surface area of $\Omega$, $\int_{\redA}\gamma_{\sdimn}(x)dx$ is finite.

\begin{lemma}[\embolden{First Variation for Minimizers}]\label{varlem}
Let $\Omega\subset\R^{\adimn}$ minimize Problem \ref{prob1}.  Let $\Sigma\colonequals\redA$.  Then there exists $\scon\in\R$ such that, for any $x\in\redA$, $H(x)-\langle x,N(x)\rangle=\scon$.
\end{lemma}
\begin{proof}
Let $f\colon\redA\to\R$ with $\int_{\redA}f(x)\gamma_{\sdimn}(x)dx=0$ and $f(x)=f(-x)$ for all $x\in\redA$.  From Lemma \ref{lemma10},
$$\int_{\redA}(H(x)-\langle N(x),x\rangle)f(x)\gamma_{\sdimn}(x)dx=0.$$
Since $\Omega=-\Omega$, this becomes
$$\int_{(\redA)\cap\{x\in\R^{\adimn}\colon x_{1}\geq0\}}(H(x)-\langle N(x),x\rangle)f(x)\gamma_{\sdimn}(x)dx=0.$$
This equality is true for any function $f$ such that $\int_{(\redA)\cap\{x\in\R^{\adimn}\colon x_{1}\geq0\}}f(x)\gamma_{\sdimn}(x)dx=0$.  So, there exists $\scon\in\R$ such that, for any $x\in\redA$ with $x_{1}\geq0$, $H(x)-\langle N(x),x\rangle=\scon$.  Since $\Omega=-\Omega$, we then have $H(x)-\langle N(x),x\rangle=\scon$ for any $x\in\redA$ with $x_{1}\leq0$.
\end{proof}

\begin{lemma}[\embolden{Second Variation for Minimizers}]\label{varlem2}
Let $\Omega\subset\R^{\adimn}$ minimize Problem \ref{prob1}.  Let $\Sigma\colonequals\redA$.  Then, for any $f\in C_{0}^{\infty}(\Sigma)$ such that $\int_{\Sigma}f(x)\gamma_{\sdimn}(x)dx=0$, and such that $f(x)=f(-x)$ for all $x\in\Sigma$, we have
$$\int_{\Sigma}f(x)Lf(x)\gamma_{\sdimn}(x)dx\leq0.$$
\end{lemma}
\begin{proof}
%We repeat the reasoning of Proposition \ref{prop2}.

From Lemma \ref{varlem} there exists $\scon\in\R$ such that, for any $x\in\redA$, $H(x)-\langle x,N(x)\rangle=\scon$.  Let $f\colon\Sigma\to\R$ satisfy $\int_{\Sigma}f(x)\gamma_{\sdimn}(x)dx=0$, and such that $f(x)=f(-x)$ for all $x\in\Sigma$.  We extend $f$ to a neighborhood $U\subset\R^{\adimn}$ of $\Sigma$ (by e.g. Whitney extension \cite{stein70}), and we denote this extension by $f$ also, so that
\begin{equation}\label{two9}
\nabla_{N}f(x)\colonequals-\scon f(x),\qquad\forall\,x\in\Sigma.
\end{equation}
Then $\nabla_{N}f(x)=\nabla_{N}f(-x)$ for all $x\in\Sigma$.

For any $x\in\R^{\adimn}$, denote $\mathrm{dist}(x,\redA)\colonequals\inf\{\vnormt{y-x}\colon y\in\redA\}$.  Define the signed distance function $d_{\Omega}\colon\R^{\adimn}\to\R$ by
$$
d_{\Omega}(x)
\colonequals
\begin{cases}
\mathrm{dist}(x,\redA), & \mbox{if}\,\, x\in\R^{\adimn}\setminus \Omega\\
-\mathrm{dist}(x,\redA), & \mbox{if}\,\, x\in \Omega.
\end{cases}
$$

We then define $X\colon\R^{\adimn}\to\R^{\adimn}$ so that
\begin{equation}\label{nine4.67}
X(x)\colonequals
\begin{cases}
 f(x)\nabla d_{\Omega}(x) &\mbox{, if}\, x\in U\\
 0 & \mbox{, if}\,x\in\R^{n}\setminus U.
 \end{cases}
 \end{equation}

Let $\{\Omega_{s}\}_{s\in(-1,1)}$ be the normal variation of $\Omega$ associated to $X$.

Since $f(x)=f(-x)$ for all $x\in\Sigma$ and $\nabla_{N}f(x)=\nabla_{N}f(-x)$ for all $x\in\Sigma$, sets $\Omega_{s}$ are symmetric to first and second order in $s$ near $s=0$ (by \eqref{tayloreq} and Lemma \ref{lemma5}.)

By \eqref{nine2} and \eqref{nine1}, and using the assumption that $f$ has mean zero,
$$\frac{d}{d s}|_{s=0}\int_{\Sigma_{s}}\gamma_{\sdimn}(x)dx
=\int_{\Sigma_{s}}(H(x)-\langle x,N(x)\rangle)f(x)\gamma_{\sdimn}(x)dx
=\scon\int_{\Sigma_{s}}f(x)\gamma_{\sdimn}(x)dx=0.$$
$$\frac{d}{d s}|_{s=0}\int_{\Omega_{s}}\gamma_{\adimn}(x)dx
=\int_{\Sigma}f(x)\gamma_{\adimn}(x)dx=0.$$

Also, by \eqref{nine1.9},
\begin{flalign*}
\frac{d^{2}}{d s^{2}}|_{s=0}\int_{\Omega_{s}}\gamma_{\adimn}(x)dx
&=\int_{\Sigma}f(x)(\nabla_{N}f(x)+f(x)(H(x)-\langle x,N(x)\rangle))\gamma_{\adimn}(x)dx\\
&=\int_{\Sigma}f(x)(\nabla_{N}f(x)+\scon f(x))\gamma_{\adimn}(x)dx\stackrel{\eqref{two9}}{=}0.
\end{flalign*}

In summary, the vector field $X$ preserves the symmetry of $\Omega_{s}$ to first and second order at $s=0$, and the vector field $X$ preserves the Gaussian volume of $\Omega_{s}$ to second order at $s=0$.  Since $\Omega$ minimizes Problem \ref{prob1}, we must therefore have
$$\frac{d^{2}}{ds^{2}}|_{s=0}\int_{\redb \Omega_{s}}\gamma_{\sdimn}(x)dx\geq0.$$

Finally, by \eqref{nine2.9},
\begin{flalign*}
0
&\leq\frac{d^{2}}{ds^{2}}|_{s=0}\int_{\redb \Omega_{s}}\gamma_{\sdimn}(x)dx\\
&=\int_{\redA}\Big(-fLf+f(H-\langle x,N\rangle)(\nabla_{N}f+f(H-\langle x,N\rangle))\Big)\gamma_{\sdimn}(x)dx
\stackrel{\eqref{two9}}{=}\int_{\redA}-fLf\gamma_{\sdimn}(x)dx.
\end{flalign*}
\end{proof}

Let $g\colon\R^{\adimn}\to\R$.  We define $\vnormt{g}_{\infty}\colonequals\sup\{t\geq0\colon \gamma_{\adimn}(x\in\R^{\adimn}\colon \abs{g(x)}>t)>0\}$.  Also, for any $0<\sigma<1$, we define
$$\vnormt{g}_{C^{1,\sigma}}\colonequals
\vnormt{g}_{\infty}+\vnormt{\vnormt{\nabla g}}_{\infty}
+\max_{i=1,\ldots,n}\sup_{x,y\in\R^{\adimn}\colon x\neq y}\frac{\absf{\frac{\partial}{\partial x_{i}}g(x)-\frac{\partial}{\partial x_{i}}g(y)}}{\vnormt{x-y}^{\sigma}}.$$

\begin{lemma}[\embolden{Existence and Regularity}]\label{lemma51}
The minimum value of Problem \ref{prob1} exists.

That is, there exists a set $\Omega\subset\R^{\adimn}$ %satisfying Constraint \ref{constraint2}
such that $\Omega$ achieves the minimum value of Problem \ref{prob1}.  Also, $\redA$ is a $C^{\infty}$ manifold.  Moreover, if $\sdimn<7$, then $\partial\Omega\setminus\redA=\emptyset$, and if $n\geq7$, then the Hausdorff dimension of $\partial\Omega\setminus\redA$ is at most $n-7$.
\end{lemma}
\begin{proof}
Existence was shown in Lemma \ref{lemma51p}.  Now, note that $\redA$ is locally the graph of a $C^{1,\sigma}$ function $g$, for some $0<\sigma<1$.  Also, in any neighborhood of $x\in\redA$, $H(x)$ can be written as \cite{evans93}
$$\sum_{i,j=1}^{\adimn}\left(1_{\{i=j\}}-\frac{\frac{\partial}{\partial x_{i}}g(x)\frac{\partial}{\partial x_{j}}g(x)}{\vnormt{\nabla g(x)}^{2}}\right)
\frac{\partial^{2}}{\partial x_{i}\partial x_{j}}g(x).$$
That is, the equation $H(x)-\langle x,N(x)\rangle-\scon=0$ can locally be written as an elliptic equation.  So, ``classical Schauder estimates'' imply that $\redA$ is locally the graph of a $C^{\infty}$ function.  The final statement concerning Hausdorff dimension follows from the theory of almost minimal surfaces \cite[Proposition 2]{barchiesi16} \cite[Theorem 21.8]{maggi12}.
\end{proof}

\section{Eigenfunctions of L}\label{seceig}

Let $e_{1},\ldots,e_{\sdimn}$ be an orthonormal frame for an orientable $\sdimn$-dimensional hypersurface $\Sigma\subset\R^{\adimn}$ with $\partial\Sigma=\emptyset$.  Let $\Delta\colonequals\sum_{i=1}^{\sdimn}\nabla_{e_{i}}\nabla_{e_{i}}$ be the Laplacian associated to $\Sigma$. Let $\nabla\colonequals\sum_{i=1}^{\sdimn}e_{i}\nabla_{e_{i}}$ be the gradient associated to $\Sigma$.  (The symbol $\nabla_{\cdot}(\cdot)$ still denotes the Euclidean connection, and the meaning of the symbol $\nabla$ should be clear from context.)  For any $\sdimn\times\sdimn$ matrix $B=(b_{ij})_{1\leq i,j\leq\sdimn}$, define $\vnormt{B}^{2}\colonequals\sum_{i,j=1}^{\sdimn}b_{ij}^{2}$.

For any $f\in C^{\infty}(\Sigma)$, define
\begin{equation}\label{three4.3}
\mathcal{L}f\colonequals \Delta f-\langle x,\nabla f\rangle.
\end{equation}
\begin{equation}\label{three4.5}
L f\colonequals \Delta f-\langle x,\nabla f\rangle+f+\vnormt{A}^{2}f.
\end{equation}
Note that there is a factor of $2$ difference between our definition of $L$ and the definition of $L$ in \cite{colding12a}.

Below we often remove the $x$ arguments of the functions for brevity.  We extend $L$ to matrices so that $(LB)_{ij}\colonequals L(B_{ij})$ for all $1\leq i,j\leq\sdimn$.

%cite Uniqueness of blowups paper of CM
\begin{lemma}[\embolden{$H$ is almost an eigenfunction of $L$} {\cite[Proposition 1.2]{colding15}} {\cite[Lemma 2.1]{guang15}}]\label{lemma30}
Let $\Sigma\subset\R^{\adimn}$ be an orientable hypersurface.  Let $\scon\in\R$.  If
\begin{equation}\label{three0}
H(x)=\langle x,N(x)\rangle+\scon,\qquad\forall\,x\in\Sigma.
\end{equation}
Then
\begin{equation}\label{three9p}
LA=2A-\scon A^{2}.
\end{equation}
\begin{equation}\label{three9}
LH=2H+\scon\vnormt{A}^{2}.
\end{equation}
\end{lemma}
\begin{proof}
Let $1\leq i\leq \sdimn$ and let $x\in\Sigma$.  Then
\begin{equation}\label{three6}
\begin{aligned}
\nabla_{e_{i}}H(x)
&\stackrel{\eqref{three0}}{=}\nabla_{e_{i}}\langle x,N\rangle
=\langle x,\nabla_{e_{i}}N(x)\rangle+\langle e_{i},N(x)\rangle
=\langle x,\nabla_{e_{i}}N(x)\rangle\\
&=\sum_{j=1}^{\sdimn}\langle\nabla_{e_{i}} N(x),e_{j}\rangle\langle x,e_{j}\rangle
\stackrel{\eqref{three2}}{=}-\sum_{j=1}^{\sdimn}a_{ij}\langle x,e_{j}\rangle.
\end{aligned}
\end{equation}
Fix $x\in\Sigma$.  Choosing the frame such that $\nabla_{e_{k}}^{T}e_{j}=0$ at $x$ for every $1\leq j,k\leq \sdimn$, we then have $\nabla_{e_{k}}e_{j}=a_{kj}N$ at $x$ by \eqref{three1}, so
\begin{equation}\label{three5}
\sum_{j=1}^{\sdimn}a_{ij}\langle x,\nabla_{e_{k}}e_{j}\rangle
=\sum_{j=1}^{\sdimn}a_{ij}a_{kj}\langle x,N\rangle
\stackrel{\eqref{three0}}{=}\sum_{j=1}^{\sdimn}a_{ij}a_{kj}(H-\scon).
\end{equation}
So, $\forall$ $1\leq i,k\leq\sdimn$,
\begin{equation}\label{three5.5}
\begin{aligned}
\nabla_{e_{k}}\nabla_{e_{i}}H(x)
&\stackrel{\eqref{three6}}{=}
-\sum_{j=1}^{\sdimn}\nabla_{e_{k}}a_{ij}\langle x,e_{j}\rangle
-\sum_{j=1}^{\sdimn}a_{ij}\langle e_{k},e_{j}\rangle
-\sum_{j=1}^{\sdimn}a_{ij}\langle x,\nabla_{e_{k}}e_{j}\rangle\\
&\stackrel{\eqref{three5}}{=}
-\sum_{j=1}^{\sdimn}\nabla_{e_{k}}a_{ij}\langle x,e_{j}\rangle
-\sum_{j=1}^{\sdimn}a_{ij}\langle e_{k},e_{j}\rangle
-\sum_{j=1}^{\sdimn}a_{ij}a_{kj}(H-\scon).
\end{aligned}
\end{equation}

Also, for any hypersurface, and for any $1\leq i,k\leq\sdimn$, (see \cite[Lemma B.8]{simon83a} where $A$ has the opposite sign),
$$(\Delta A)_{ik}=-\vnormt{A}^{2}a_{ik}-\sum_{j=1}^{\sdimn}Ha_{ij}a_{kj}-\nabla_{e_{i}}\nabla_{e_{k}}H.$$
So, using the Codazzi equation ($\nabla_{e_{i}}a_{kj}=\nabla_{e_{j}}a_{ki}$) and that $A$ is a symmetric matrix,
\begin{flalign*}
(\Delta A)_{ik}
%&=-\vnormt{A}^{2}a_{ik}-\sum_{j=1}^{\sdimn}(\langle x,N\rangle +c)a_{ij}a_{kj}-\nabla_{e_{i}}\nabla_{e_{k}}H\\
&\stackrel{\eqref{three5.5}}{=}-\vnormt{A}^{2}a_{ik}
+\sum_{j=1}^{\sdimn}\nabla_{e_{j}}a_{ki}\langle x,e_{j}\rangle
+\sum_{j=1}^{\sdimn}a_{ki}\langle e_{i},e_{j}\rangle
-\scon\sum_{j=1}^{\sdimn}a_{ij}a_{kj}\\
&=-\vnormt{A}^{2}a_{ik}+\langle x,\nabla a_{ik}\rangle+a_{ik} - \scon(A^{2})_{ik}.
\end{flalign*}

Therefore,
$$LA\stackrel{\eqref{three4.5}}{=}2A-\scon A^{2}.$$ % \Delta H -<x,\nabla H>+H+|A|^2 H

Finally, summing the diagonal entries of this equality and applying \eqref{three4} proves \eqref{three9}.

\end{proof}

\begin{lemma}[\embolden{Linear Eigenfunction of $L$}, {\cite{mcgonagle15,barchiesi16}}]\label{lemma45}
Let $\Sigma\subset\R^{\adimn}$ be an orientable hypersurface.  Let $\scon\in\R$.  Suppose
\begin{equation}\label{three0n}
H(x)=\langle x,N\rangle+\scon,\qquad\forall\,x\in\Sigma.
\end{equation}
Let $v\in\R^{\adimn}$.  Then
$$L\langle v,N\rangle=\langle v,N\rangle.$$
\end{lemma}
\begin{proof}
Let $1\leq i\leq \sdimn$.  Then
\begin{equation}\label{three5g}
\nabla_{e_{i}}\langle v,N\rangle
=\langle v,\nabla_{e_{i}}N\rangle
\stackrel{\eqref{three2}}{=}-\sum_{j=1}^{\sdimn}a_{ij}\langle v,e_{j}\rangle.
\end{equation}
Fix $x\in\Sigma$.  Choosing the frame such that $\nabla_{e_{k}}^{T}e_{j}=0$ at $x$ for every $1\leq j,k\leq \sdimn$, we then have $\nabla_{e_{k}}e_{j}=a_{kj}N$ at $x$ by \eqref{three1}, so using also Codazzi's equation,
\begin{equation}\label{three6.5}
\begin{aligned}
\nabla_{e_{i}}\nabla_{e_{i}}\langle v,N\rangle
&=-\sum_{j=1}^{\sdimn}\nabla_{e_{i}}a_{ij}\langle v,e_{j}\rangle
-\sum_{j=1}^{\sdimn}a_{ij}\langle v,\nabla_{e_{i}}e_{j}\rangle\\
&=-\sum_{j=1}^{\sdimn}\nabla_{e_{j}}a_{ii}\langle v,e_{j}\rangle
-\sum_{j=1}^{\sdimn}a_{ij}a_{ij}\langle v,N\rangle.
\end{aligned}
\end{equation}
Therefore,
$$\Delta\langle v,N\rangle
=\sum_{i=1}^{\sdimn}\nabla_{e_{i}}\nabla_{e_{i}}\langle v,N\rangle
\stackrel{\eqref{three4}\wedge\eqref{three6.5}}{=}\langle v,\nabla H\rangle-\vnormt{A}^{2}\langle v,N\rangle.$$
So far, we have not used any of our assumptions.  Using now \eqref{three6}, and that $A$ is symmetric,
$$
\langle v,\nabla H\rangle
=-\sum_{i,j=1}^{\sdimn}\langle x,e_{j}\rangle a_{ij}\langle v,e_{i}\rangle
\stackrel{\eqref{three5g}}{=}\langle x,\nabla\langle v,N\rangle\rangle.
$$
In summary,
$$\Delta\langle v,N\rangle
=\langle x,\nabla\langle v,N\rangle\rangle-\vnormt{A}^{2}\langle v,N\rangle.
$$
We conclude by \eqref{three4.5}.
\end{proof}
\begin{remark}\label{rk20}
Let $f,g\in C^{\infty}(\Sigma)$.  Using \eqref{three4.5}, we get the following product rule for $L$.
\begin{flalign*}
L(fg)
&=f\Delta g+g\Delta f+2\langle \nabla f,\nabla g\rangle-f\langle x,\nabla g\rangle-g\langle x,\nabla f\rangle+\vnormt{A}^{2}fg+fg\\
&=fLg+gLf+2\langle \nabla f,\nabla g\rangle-\vnormt{A}^{2}fg-fg.
\end{flalign*}
So, by Lemma \ref{lemma45}, if $v,w\in\R^{\adimn}$,
$$L\langle v,N\rangle\langle w,N\rangle=
\langle v,N\rangle\langle w,N\rangle
+2\langle \nabla \langle v,N\rangle,\nabla \langle w,N\rangle\rangle
-\vnormt{A}^{2}\langle v,N\rangle\langle w,N\rangle.$$
\end{remark}

The following Lemma follows from Stokes' Theorem.

\begin{lemma}[\embolden{Integration by Parts}, {\cite[Corollary 3.10]{colding12a}, \cite[Lemma 5.4]{zhu16}}]\label{lemma39.7}
Let $\Sigma\subset\R^{\adimn}$ be an $\sdimn$-dimensional hypersurface.  Let $f,g\colon\Sigma\to\R$.  Assume that $f$ is a $C^{2}$ function and $g$ is a $C^{2}$ function with compact support.  Then
$$\int_{\Sigma}f\mathcal{L}g\gamma_{\sdimn}(x)dx=\int_{\Sigma}g\mathcal{L}f\gamma_{\sdimn}(x)dx=-\int_{\Sigma}\langle\nabla f,\nabla g\rangle\gamma_{\sdimn}(x)dx.$$
%$$\int_{\Sigma}fLg\gamma_{\sdimn}(x)dx=\int_{\Sigma}gLf\gamma_{\sdimn}(x)dx.$$
\end{lemma}

\section{Curvature Bounds}\label{seccurv}

In this rather technical section, we show that the derivatives of the curvature have finite integrals.  This will be used later on to justify a more general version of Lemma \ref{lemma39.7}.  Many of the results of this section are unnecessary if we assume that $\partial\Omega$ is a $C^{\infty}$ manifold.  However, from Lemma \ref{lemma51}, if $\Omega$ minimizes Problem \ref{prob1}, it may occur that $\partial\Omega\setminus\redA$ is a nonempty set with Hausdorff dimension $\sdimn-7$.  And indeed, due to e.g. the existence of Simons-Lawson cones, this is the best possible.  If the singular set $\partial\Omega\setminus\redA$ is nonempty, we then have to be careful about integrals of curvature blowing up near the singular set.

For any hypersurface $\Sigma\subset\R^{\adimn}$, we define
\begin{equation}\label{seven0}
\pcon=\pcon(\Sigma)
\colonequals\sup_{f\in C_{0}^{\infty}(\Sigma)}\frac{\int_{\Sigma}fLf\gamma_{\sdimn}(x)dx}{\int_{\Sigma}f^{2}\gamma_{\sdimn}(x)dx}.
\end{equation}
By the definition of $\pcon$,
\begin{equation}\label{seven1}
\Sigma_{1}\subset\Sigma_{2}\qquad\Longrightarrow\qquad\pcon(\Sigma_{1})\leq\pcon(\Sigma_{2}).
\end{equation}

We say an $\sdimn$-dimensional hypersurface $\Sigma\subset\R^{\adimn}$ has \embolden{polynomial volume growth} if there exists $c>0$ such that, and for any $r>1$, $\int_{\{x\in\Sigma\colon \vnormt{x}\leq r\}}dx\leq c r^{\sdimn}$.

\begin{lemma}[\embolden{Existence of an Eigenfunction}, {\cite[Lemma 6.5]{zhu16}}]\label{lemma40.9}
Let $\Sigma\subset\R^{\adimn}$ be a symmetric, connected, orientable hypersurface with polynomial volume growth.  Assume that $\Sigma$ is a $C^{\infty}$ hypersurface with possibly nonempty boundary.  Assume that $\pcon(\Sigma)<\infty$.  Then there exists a positive $C^{2}$ function $g$ on $\Sigma$ such that $Lg=\pcon(\Sigma)g$, and such that $g(x)=g(-x)$ for all $x\in\Sigma$.
\end{lemma}
\begin{proof}
Fix $x\in\Sigma$.  Let $\Sigma_{1}\subset\Sigma_{2}\subset\ldots$ be a sequence of compact $C^{\infty}$ hypersurfaces such that $\cup_{i=1}^{\infty}\Sigma_{i}=\Sigma$.  For each $i\geq1$, let $g_{i}$ be a positive Dirichlet eigenfunction of $L$ on $\Sigma_{i}$ such that $Lg_{i}=\pcon(\Sigma_{i})g_{i}$.  By multiplying by a constant, we may assume $g_{i}(x)=1$ for all $i\geq1$. Since $\pcon(\Sigma_{i})$ increases to $\pcon(\Sigma)<\infty$ as $i\to\infty$ by \eqref{seven0}, the Harnack inequality implies that there exists $c=c(\Sigma_{i},\pcon(\Sigma))$ such that $1\leq\sup_{x\in\Sigma_{i}}g_{i}(x)\leq c\inf_{x\in\Sigma_{i}}g_{i}(x)\leq c$.  Elliptic theory then gives uniform $C^{2,\sigma}$ bounds for the functions $g_{1},g_{2},\ldots$ on each compact subset of $\Sigma$.  So, by Arzel\`{a}-Ascoli there exists a uniformly convergent subsequence of $g_{1},g_{2},\ldots$ which converges to a nonnegative solution of $Lg=\pcon(\Sigma)g$ on $\Sigma$ with $g(x)=1$.  The Harnack inequality then implies that $g>0$ $\Sigma$.  Finally, the definition of $L$ \eqref{three4.5} and symmetry of $\Sigma$ implies that $L(g(x)+g(-x))=\pcon(\Sigma)(g(x)+g(-x))$.  That is, we may assume that $g$ itself satisfies $g(x)=g(-x)$ for all $x\in\Sigma$.
\end{proof}

\begin{lemma}\label{lemma95.1}
If $\Omega\subset\R^{\adimn}$ minimizes Problem \ref{prob1}, then $\pcon(\redA)<\infty$.
\end{lemma}
\begin{proof}
We argue by contradiction.  Suppose $\pcon(\redA)=\infty$.  Then, for any $m>0$, there exists a compact symmetric subset $\Sigma_{m}\subset\redA$ and there exists a Dirichlet eigenfunction $g_{m}>0$ on $\Sigma_{m}$ (by Lemma \ref{lemma40.9}, or by applying spectral theory to the compact operator $(L-\vnormt{A}^{2}-2)^{-1}$ on $\Sigma_{m}$) such that $Lg_{m}=\pcon_{m}g_{m}$, $\pcon_{m}>m$, and such that $g_{m}(x)=g_{m}(-x)$ for all $x\in\Sigma_{m}$.  From Lemma \ref{varlem}, $\exists$ $\scon\in\R$ such that $H(x)=\langle x,N(x)\rangle+\scon$ on $\redA\supset\Sigma_{m}$.  To conclude, it suffices by Lemma \ref{varlem2} to find a function $f$ on $\Sigma_{m}$ (extended to be zero on $\redA\setminus\Sigma_{m}$) such that $\int_{\Sigma_{m}}f\gamma_{\sdimn}(x)dx=0$, $f(x)=f(-x)$ for all $x\in\Sigma_{m}$, and $\int_{\Sigma_{m}}fLf\gamma_{\sdimn}(x)dx>0$.

Let $f\colonequals g_{m}+H-\scon$, where we multiply $g_{m}$ by a constant so that $\int_{\Sigma_{m}}f\gamma_{\sdimn}(x)dx=0$.  (In the case that $\int_{\Sigma_{m}}g_{m}\gamma_{\sdimn}(x)dx=0$ for some $m\geq1$, we conclude by choosing $f=g_{m}$.  So, we may assume that $\int_{\Sigma_{m}}g_{m}\gamma_{\sdimn}(x)dx\neq0$ for all $m\geq1$.)  By construction $f(x)=f(-x)$ for all $x\in\Sigma_{m}$. It remains to bound $\int_{\Sigma_{m}}fLf\gamma_{\sdimn}(x)dx$.  From \eqref{three9} and \eqref{three4.5},
\begin{equation}\label{two18}
L(H-\scon)=2H+\scon\vnormt{A}^{2}-\scon\vnormt{A}^{2}-\scon=2H-\scon.
\end{equation}
Integrating by parts with Lemma \ref{lemma39.7}, we have
\begin{flalign*}
\int_{\Sigma_{m}}fLf\gamma_{\sdimn}(x)dx
&=\int_{\Sigma_{m}}(g_{m}+H-\scon)L(g_{m}+H-\scon)\gamma_{\sdimn}(x)dx\\
&=\int_{\Sigma_{m}}(g_{m}L g_{m}+2 g_{m}L(H-\scon)+(H-\scon)L(H-\scon))\gamma_{\sdimn}(x)dx\\
&\stackrel{\eqref{two18}}{=}\int_{\Sigma_{m}}(\pcon_{m}g_{m}^{2}+2 g_{m}(2H-\scon)+(H-\scon)(2H-\scon))\gamma_{\sdimn}(x)dx\\
&=\int_{\Sigma_{m}}(\pcon_{m}g_{m}^{2}+2 g_{m}(H-\scon)+2Hg_{m}+(H-\scon)^{2}+H(H-\scon))\gamma_{\sdimn}(x)dx\\
&=\int_{\Sigma_{m}}((\pcon_{m}-1)g_{m}^{2}+(g_{m}+H-\scon)^{2}+2Hg_{m}+H(H-\scon))\gamma_{\sdimn}(x)dx\\
&=\int_{\Sigma_{m}}((\pcon_{m}-2)g_{m}^{2}+(g_{m}+H-\scon)^{2}+(g_{m}+H)^{2}-\scon H)\gamma_{\sdimn}(x)dx.
\end{flalign*}
%If $\scon\leq0$, we are done by letting $m\to\infty$ so that $\lim_{m\to\infty}\pcon_{m}=\infty$.
%Since $\Omega$ is convex and $H-\scon=\langle x,N\rangle$, we have $H-\scon\geq0$, so that $g_{m}<0$ since $\int_{\Sigma_{m}}(g_{m}+H-\scon)\gamma_{\sdimn}(x)dx=0$.
By assumption, $\int_{\Sigma_{m}}(g_{m}+H-\scon)\gamma_{\sdimn}(x)dx=0$, so that $\int_{\Sigma_{m}}(-\scon H)\gamma_{\sdimn}(x)dx=\int_{\Sigma_{m}}(\scon g_{m}-\scon^{2})\gamma_{\sdimn}(x)dx$. From the Cauchy-Schwarz inequality, $\int_{\Sigma_{m}}g_{m}^{2}\gamma_{\sdimn}(x)dx\geq\frac{(\int_{\Sigma_{m}}g_{m}\gamma_{\sdimn}(x)dx)^{2}}{\int_{\Sigma_{m}}\gamma_{\sdimn}(x)dx}$.  So,
\begin{flalign*}
&\int_{\Sigma_{m}}fLf\gamma_{\sdimn}(x)dx\\
&\geq\int_{\Sigma_{m}}\Big(g_{m}\Big(\scon+(\pcon_{m}-2)\frac{\int_{\Sigma_{m}}g_{m}\gamma_{\sdimn}(y)dy}{\int_{\Sigma_{m}}\gamma_{\sdimn}(z)dz}\Big)
+(g_{m}+H-\scon)^{2}+(g_{m}+H)^{2}-\scon^{2}\Big)\gamma_{\sdimn}(x)dx\\
&=\int_{\Sigma_{m}}\Big(-\langle x,N\rangle\Big(\scon+(\pcon_{m}-2)\frac{\int_{\Sigma_{m}}-\langle y,N\rangle\gamma_{\sdimn}(y)dy}{\int_{\Sigma_{m}}\gamma_{\sdimn}(z)dz}\Big)\\
&\qquad\qquad\qquad+(g_{m}+H-\scon)^{2}+(g_{m}+H)^{2}-\scon^{2}\Big)\gamma_{\sdimn}(x)dx.
\end{flalign*}
In the last line, we used $\int_{\Sigma_{m}}g_{m}\gamma_{\sdimn}(x)dx=\int_{\Sigma_{m}}(\scon-H)\gamma_{\sdimn}(x)dx=\int_{\Sigma_{m}}-\langle x,N\rangle\gamma_{\sdimn}(x)dx$.  So, letting $m\to\infty$ and using $\lim_{m\to\infty}\pcon_{m}=\infty$ concludes the proof, in the case that $\lim_{m\to\infty}\int_{\Sigma_{m}}\langle x,N\rangle\gamma_{\sdimn}(x)dx\neq0$.  (Recall that $\int_{\Sigma_{m}}g_{m}\gamma_{\sdimn}(x)dx\neq0$ so $\int_{\Sigma_{m}}\langle x,N\rangle\gamma_{\sdimn}(x)dx\neq0$ for all $m\geq1$.)  It remains to address the case that $\lim_{m\to\infty}\int_{\Sigma_{m}}\langle x,N\rangle\gamma_{\sdimn}(x)dx=0$.  That is, it remains to address when $\int_{\redA} (H-\scon)\gamma_{\sdimn}(x)dx=0$.  In this case, we use $f\colonequals H-\scon$ to get
\begin{flalign*}
\int_{\redA} (H-\scon)L(H-\scon)\gamma_{\sdimn}(x)dx
&\stackrel{\eqref{two18}}{=}\int_{\redA} (H-\scon)(2H-\scon)\gamma_{\sdimn}(x)dx\\
&=\int_{\redA} (H-\scon)(2H-2\scon)\gamma_{\sdimn}(x)dx
=2\int_{\redA} (H-\scon)^{2}\gamma_{\sdimn}(x)dx.
\end{flalign*}
The latter quantity is positive, unless $H=\scon$ is constant on $\redA$.  Then $H-\scon=\langle x,N\rangle=0$ for all $x\in\redA$.  That is, $\redA$ is a cone.  If $\scon\neq0$, this is impossible, since $H(tx)=H(x)/t$ $\forall$ $x\in\redA$, $\forall$ $t>0$.  So, it remains to consider the case that $H=\langle x,N\rangle=0$ for all $x\in\redA$.  This case is eliminated in Lemma \ref{lemma52} below.
\end{proof}

\begin{itemize}
\item We say that $f\in L_{2}(\Sigma,\gamma_{\sdimn})$ if $\int_{\Sigma}\abs{f}^{2}\gamma_{\sdimn}(x)dx<\infty$.
\item We say that $f\in W_{1,2}(\Sigma,\gamma_{\sdimn})$ if $\int_{\Sigma}(\abs{f}^{2}+\vnormt{\nabla f}^{2})\gamma_{\sdimn}(x)dx<\infty$.
\item We say that $f\in W_{2,2}(\Sigma,\gamma_{\sdimn})$ if $\int_{\Sigma}(\abs{f}^{2}+\vnormt{\nabla f}^{2}+\abs{\mathcal{L}f}^{2})\gamma_{\sdimn}(x)dx<\infty$.
\end{itemize}

\begin{lemma}[{\cite[Lemma 9.15(2)]{colding12a}}]\label{lemma28}
Let $\Sigma\subset\R^{\adimn}$ be a $C^{\infty}$ hypersurface, with possibly nonempty boundary.  Assume $\pcon(\Sigma)<\infty$.  Suppose $g\colon\Sigma\to\R$ is a $C^{2}$ function with $g>0$ and $Lg=\pcon g$.  If $\phi\in W_{1,2}(\Sigma,\gamma_{\sdimn})$, then

$$\int_{\Sigma}\phi^{2}(\vnormt{A}^{2}+\vnormt{\nabla\log g}^{2})\gamma_{\sdimn}(x)dx
\leq4\int_{\Sigma}(\vnormt{\nabla\phi}^{2}+(\pcon-1)\phi^{2})\gamma_{\sdimn}(x)dx.$$
\end{lemma}
\begin{proof}
\begin{flalign*}
\mathcal{L}\log g
&\stackrel{\eqref{three4.3}}{=}\sum_{i=1}^{\sdimn}\nabla_{e_{i}}\left(\frac{\nabla_{e_{i}}g}{g}\right)-\frac{\langle x,\nabla g\rangle}{g}
=\sum_{i=1}^{\sdimn}\frac{-(\nabla_{e_{i}}g)^{2}}{g^{2}}+\frac{\mathcal{L}g}{g}\\
&=-\vnormt{\nabla\log g}^{2}+\frac{\mathcal{L}g}{g}
\stackrel{\eqref{three4.5}}{=}-\vnormt{\nabla\log g}^{2}+\frac{Lg -\vnormt{A}^{2}g-g}{g}\\
&\stackrel{\eqref{three9}}{=}-\vnormt{\nabla\log g}^{2}+\frac{\pcon g -\vnormt{A}^{2}g-g}{g}\\
&=-\vnormt{\nabla\log g}^{2}+(\pcon-1)-\vnormt{A}^{2}.
\end{flalign*}
Let $\eta\in C_{0}^{\infty}(\Sigma)$.  By Lemma \ref{lemma39.7},
\begin{flalign*}
\int_{\Sigma}\langle\nabla \eta^{2},\nabla \log g\rangle\gamma_{\sdimn}(x)dx
&=-\int_{\Sigma}\eta^{2}\mathcal{L}\log g\gamma_{\sdimn}(x)dx\\
&=\int_{\Sigma}\eta^{2}\Big(\vnormt{\nabla\log g}^{2}+(1-\pcon)+\vnormt{A}^{2}\Big)\gamma_{\sdimn}(x)dx.
\end{flalign*}
By the arithmetic mean geometric mean inequality (AMGM), %2ab\leq (a^2+b^2)
$$\abs{\langle\nabla \eta^{2},\nabla \log g\rangle}
\leq2\vnormt{\nabla\eta}^{2}+\frac{1}{2}\eta^{2}\vnormt{\nabla \log g}^{2}.$$
So,
$$
\int_{\Sigma} \eta^{2}(\vnormt{A}^{2}+\vnormt{\nabla \log g}^{2})\gamma_{\sdimn}(x)dx
\leq4\int_{\Sigma}\Big(\vnormt{\nabla\eta}^{2}+(\pcon-1)\eta^{2}\Big)\gamma_{\sdimn}(x)dx.
$$
Letting $\eta$ approximate $\phi$ by cutoff functions and applying the monotone convergence theorem completes the proof.
\end{proof}

\begin{lemma}[{\cite[Lemma 6.2]{zhu16}}]\label{lemma28.5}
Let $\Omega\subset\R^{\adimn}$ and let $\Sigma\colonequals\partial\Omega$.  Assume $\pcon(\Sigma)<\infty$.  Suppose $g\colon\Sigma\to\R$ is a $C^{2}$ function with $g>0$ and $Lg=\pcon g$.  Assume that the Hausdorff dimension of $\partial\Omega\setminus\redA$ is at most $\sdimn-7$.  If $\phi\in W_{1,2}(\Sigma,\gamma_{\sdimn})$ and if $\int_{\Sigma}\abs{\phi}^{7/3}\gamma_{\sdimn}(x)dx<\infty$, then

$$\int_{\Sigma}\phi^{2}(\vnormt{A}^{2}+\vnormt{\nabla\log \abs{g}}^{2})\gamma_{\sdimn}(x)dx
\leq4\int_{\Sigma}(\vnormt{\nabla\phi}^{2}+(\pcon-1)\phi^{2})\gamma_{\sdimn}(x)dx.$$
\end{lemma}

\begin{cor}[{\cite[Lemma 6.2]{zhu16}}]\label{cor5}
Let $\Omega\subset\R^{\adimn}$ and let $\Sigma\colonequals\partial\Omega$.  Assume $\pcon(\Sigma)<\infty$.  Suppose $g\colon\Sigma\to\R$ is a $C^{2}$ function with $g>0$ and $Lg=\pcon g$.  Assume that the Hausdorff dimension of $\partial\Omega\setminus\redA$ is at most $\sdimn-7$.  Then for any $k\geq0$, $\vnormt{A}\vnormt{x}^{k}\in L_{2}(\Sigma,\gamma_{\sdimn})$
\end{cor}

\begin{lemma}[{\cite[Theorem 9.36]{colding12a}}]\label{lemma39.1}
Let $\Sigma\subset\R^{\adimn}$ be a connected, orientable $C^{\infty}$ hypersurface with polynomial volume growth and with possibly nonempty boundary.  Assume $\exists$ $\scon\in\R$ such that $H(x)=\langle x,N(x)\rangle+\scon$ for all $x\in\Sigma$.  Let $\pcon\colonequals\pcon(\Sigma)$.   Assume $\pcon(\Sigma)<\infty$.  Then
$$\vnormt{\nabla H}\in L_{2}(\Sigma,\gamma_{\sdimn}),\qquad\vnormt{A}\abs{H}\in L_{2}(\Sigma,\gamma_{\sdimn}).$$
$$\int_{\Sigma}\abs{H\mathcal{L}H}\gamma_{\sdimn}(x)dx<\infty.$$
\end{lemma}
\begin{proof}
As shown in \eqref{three6}, since $H(x)=\langle x,N\rangle+\scon$, for any $1\leq i\leq\sdimn$, $\nabla_{e_{i}}H(x)=-\sum_{j=1}^{\sdimn}a_{ij}\langle x,e_{j}\rangle$.  Therefore,
$$\vnormt{\nabla H}^{2}\leq\vnormt{A}^{2}\vnormt{x}^{2},\qquad \vnormt{A}^{2}H^{2}\leq2\vnormt{A}^{2}(\vnormt{x}^{2}+\scon^{2}).$$
So, Corollary \ref{cor5} implies that $\vnormt{\nabla H}\in L_{2}(\Sigma,\gamma_{\sdimn})$ and $\vnormt{A}\abs{H}\in L_{2}(\Sigma,\gamma_{\sdimn})$.

For the final assertion, note that  $\mathcal{L}H\stackrel{\eqref{three4.5}}{=}LH-\vnormt{A}^{2}H-H\stackrel{\eqref{three9}}{=}2H+\scon\vnormt{A}^{2}-\vnormt{A}^{2}H-H$.  So,
$$\abs{H\mathcal{L}H}\leq H^{2}+\abs{\scon}\abs{H}\vnormt{A}^{2}+\vnormt{A}^{2}H^{2}.$$
So, the polynomial volume growth (and $H^{2}\leq 2(\vnormt{x}^{2}+\scon^{2})$) and the above results show that $\int_{\Sigma}\abs{H\mathcal{L}H}\gamma_{\sdimn}(x)dx<\infty$.
\end{proof}

The following geometric inequality is essentially shown in \cite[Lemma 10.8]{colding12a},\cite[Lemma 7.1]{zhu16} and \cite[Lemma 4.1]{cheng15}, and it is inspired by an inequality of Simons \cite{simons68}.

When $A,B$ are $\sdimn\times\sdimn$ matrices, we use the notation $\langle A,B\rangle\colonequals\sum_{i,j=1}^{\sdimn}a_{ij}b_{ij}$.  Note that $\langle A,A\rangle=\vnormt{A}^{2}$.  Recall that we extend $L$ to matrices so that $(LA)_{ij}\colonequals L(A_{ij})$ for all $1\leq i,j\leq\sdimn$.

\begin{lemma}[\embolden{Simons-type inequality}, {\cite{simons68,colding12a,cheng15,zhu16}}]\label{lemma38}
Let $\Sigma$ be a $C^{\infty}$ orientable hypersurface.  Let $\scon\in\R$.  Suppose $H(x)=\langle x,N(x)\rangle+\scon$, $\forall\,x\in\Sigma$.  Then
\begin{equation}\label{three30}
\begin{aligned}
\vnormt{A}L\vnormt{A}
&=2\vnormt{A}^{2}-\scon\langle A^{2},A\rangle+\vnormt{\nabla A}^{2}-\vnormt{\nabla \vnormt{A}}^{2}\\
&\geq 2\vnormt{A}^{2}-\scon\langle A^{2},A\rangle.
\end{aligned}
\end{equation}
\end{lemma}
\begin{proof}
Using for now only \eqref{three4.5}, we have
\begin{flalign*}
L\vnormt{A}
&= L(\vnormt{A}^{2})^{1/2}
\stackrel{\eqref{three4.5}}{=}\sum_{i=1}^{\sdimn}\nabla_{e_{i}}\left(\frac{1}{2}\frac{\nabla_{e_{i}}(\vnormt{A}^{2})}{\vnormt{A}}\right)
-\frac{1}{2}\frac{\langle x,\nabla \vnormt{A}^{2}\rangle}{\vnormt{A}}+\vnormt{A}^{3}+\vnormt{A}\\
&=\frac{1}{2}\frac{\vnormt{A}\Delta \vnormt{A}^{2}-\sum_{i=1}^{\sdimn}[\nabla_{e_{i}}\vnormt{A}^{2}][\nabla_{e_{i}}\vnormt{A}]}{\vnormt{A}^{2}}
-\frac{1}{2}\frac{\langle x,\nabla \vnormt{A}^{2}\rangle}{\vnormt{A}}+\vnormt{A}^{3}+\vnormt{A}\\
&=\frac{1}{2}\frac{\Delta \vnormt{A}^{2}}{\vnormt{A}}-\frac{\vnormt{\nabla \vnormt{A}}^{2}}{\vnormt{A}}
-\frac{1}{2}\frac{\sum_{i=1}^{\sdimn}\langle x,e_{i}\rangle\nabla_{e_{i}}\vnormt{A}^{2}}{\vnormt{A}}+\vnormt{A}^{3}+\vnormt{A}\\
&=\frac{\langle A,\Delta A\rangle+\vnormt{\nabla A}^{2}}{\vnormt{A}}-\frac{\vnormt{\nabla \vnormt{A}}^{2}}{\vnormt{A}}
-\frac{\sum_{i,j,k=1}^{\sdimn}\langle x,e_{i}\rangle a_{jk}\nabla_{e_{i}}a_{jk}}{\vnormt{A}}+\vnormt{A}^{3}+\vnormt{A}\\
&\stackrel{\eqref{three4.5}}{=}\frac{\langle A, LA\rangle}{\vnormt{A}}+\frac{\vnormt{\nabla A}^{2}-\vnormt{\nabla \vnormt{A}}^{2}}{\vnormt{A}}.
\end{flalign*}
Now, using \eqref{three9p}, we get
$$\vnormt{A}L\vnormt{A}=2\vnormt{A}^{2}-\scon\langle A^{2},A\rangle+\vnormt{\nabla A}^{2}-\vnormt{\nabla \vnormt{A}}^{2}.$$
The proof is completed since $\vnormt{\nabla A}^{2}-\vnormt{\nabla \vnormt{A}}^{2}\geq0$, which follows by the Cauchy-Schwarz inequality.
\end{proof}

\begin{lemma}[{\cite[Lemma 10.2]{colding12a}}]\label{lemma39.9}
Let $\Sigma\subset\R^{\adimn}$ be any $\sdimn$-dimensional hypersurface.  Then
\begin{equation}\label{two12}
\Big(1+\frac{2}{\adimn}\Big)\vnormt{\nabla\vnormt{A}}^{2}
\leq \vnormt{\nabla A}^{2}+\frac{2\sdimn}{\adimn}\vnormt{\nabla H}^{2}.
\end{equation}
\end{lemma}

The following estimate is adapted from \cite{colding12a}, which itself was adapted from \cite{schoen75}.

\begin{lemma}[{\cite[Proposition 10.14]{colding12a}}]\label{lemma39.2}
Let $\Omega\subset\R^{\adimn}$ be a convex set and let $\Sigma\colonequals\redA$.  Assume that the Hausdorff dimension of $\partial\Omega\setminus\redA$ is at most $n-4$.  Assume $\pcon(\Sigma)<\infty$.

Assume $\exists$ $\scon\in\R$ such that $H(x)=\langle x,N(x)\rangle+\scon$ for all $x\in\Sigma$.  Assume $H>0$.  Then
$$\int_{\Sigma}(\vnormt{A}^{2}+\vnormt{A}^{4}+\vnormt{\nabla\vnormt{A}}^{2}+\vnormt{\nabla A}^{2})\gamma_{\sdimn}(x)dx<\infty.$$
\end{lemma}
\begin{proof}
Since $H>0$, $\log H$ is well-defined, so that
\begin{equation}\label{logeq}
\begin{aligned}
\mathcal{L}\log H
&\stackrel{\eqref{three4.3}}{=}\sum_{i=1}^{\sdimn}\nabla_{e_{i}}\left(\frac{\nabla_{e_{i}}H}{H}\right)-\frac{\langle x,\nabla H\rangle}{H}
=\sum_{i=1}^{\sdimn}\frac{-(\nabla_{e_{i}}H)^{2}}{H^{2}}+\frac{\mathcal{L}H}{H}\\
&=-\vnormt{\nabla\log H}^{2}+\frac{\mathcal{L}H}{H}
\stackrel{\eqref{three4.5}}{=}-\vnormt{\nabla\log H}^{2}+\frac{LH -\vnormt{A}^{2}H-H}{H}\\
&\stackrel{\eqref{three9}}{=}-\vnormt{\nabla\log H}^{2}+\frac{2H+\scon\vnormt{A}^{2} -\vnormt{A}^{2}H-H}{H}\\
&=-\vnormt{\nabla\log H}^{2}+1-\vnormt{A}^{2}+\scon\frac{\vnormt{A}^{2}}{H}.
\end{aligned}
\end{equation}%  a^2 + b^2 / a+b.  a+b\geq (a^2 +b^2)^1/2
%Since $\Omega$ is convex, the matrix $A$ is negative definite with nonnegative diagonal entries, so $H\geq\vnormt{A}$, so that $\vnormt{A}^{2}/H\leq\vnormt{A}$
%  Needed?

Note that $\int_{\Sigma}\vnormt{A}^{2}\gamma_{\sdimn}(x)dx<\infty$ by Corollary \ref{cor5}.  Let $\phi\in C_{0}^{\infty}(\Sigma)$.  Integrating by parts with Lemma \ref{lemma39.7},
\begin{flalign*}
\int_{\Sigma}\langle\nabla\phi^{2},\nabla\log H\rangle\gamma_{\sdimn}(x)dx
&=-\int_{\Sigma}\phi^{2}\mathcal{L}\log H\gamma_{\sdimn}(x)dx\\
&\stackrel{\eqref{logeq}}{=}\int_{\Sigma}\phi^{2}\Big(\vnormt{\nabla\log H}^{2}-1+\vnormt{A}^{2}-\scon\frac{\vnormt{A}^{2}}{H}\Big)\gamma_{\sdimn}(x)dx.
\end{flalign*}
From the AMGM inequality, $\abs{\langle\nabla\phi^{2},\nabla\log H\rangle}\leq\vnormt{\nabla\phi}^{2}+\phi^{2}\vnormt{\nabla\log H}^{2}$, so that
\begin{equation}\label{two57}
\int_{\Sigma}\phi^{2}\vnormt{A}^{2}\gamma_{\sdimn}(x)dx
\leq\int_{\Sigma}\Big[\vnormt{\nabla \phi}^{2}+\phi^{2}\Big(1+\scon\frac{\vnormt{A}^{2}}{H}\Big)\Big]\gamma_{\sdimn}(x)dx.
\end{equation}
%Since $\Omega$ is convex, the matrix $A$ is negative definite with nonnegative diagonal entries, so $H\geq\vnormt{A}$, so that $\vnormt{A}^{2}/H\leq\vnormt{A}$.  That is
%$$
%\int_{\Sigma}\phi^{2}\vnormt{A}^{2}\gamma_{\sdimn}(x)dx
%\leq\int_{\Sigma}\vnormt{\nabla \phi}^{2}+\phi^{2}(1+\abs{\scon}\vnormt{A})\gamma_{\sdimn}(x)dx.
%$$
Let $0<\epsilon<1/2$ to be chosen later.  Using now $\phi\colonequals\eta\vnormt{A}$ in \eqref{two57}, where $\eta\in C_{0}^{\infty}(\Sigma)$, $\eta\geq0$, and using the AMGM inequality in the form $2ab\leq \epsilon a^{2}+b^{2}/\epsilon$, $a,b>0$,
\begin{equation}\label{two13}
\begin{aligned}
&\int_{\Sigma}\eta^{2}\vnormt{A}^{4}\gamma_{\sdimn}(x)dx\\
&\leq\int_{\Sigma}\Big[\eta^{2}\vnormt{\nabla\vnormt{A}}^{2}+2\eta\vnormt{A}\vnormt{\nabla\vnormt{A}}\vnormt{\nabla\eta}
+\vnormt{A}^{2}\vnormt{\nabla\eta}^{2}+\eta^{2}\vnormt{A}^{2}\Big(1+\scon\frac{\vnormt{A}^{2}}{H}\Big)\Big]\gamma_{\sdimn}(x)dx\\
&\leq\int_{\Sigma}\Big[(1+\epsilon)\eta^{2}\vnormt{\nabla\vnormt{A}}^{2}
+\vnormt{A}^{2}\vnormt{\nabla\eta}^{2}(1+1/\epsilon)+\eta^{2}\vnormt{A}^{2}\Big(1+\scon\frac{\vnormt{A}^{2}}{H}\Big)\Big]\gamma_{\sdimn}(x)dx.
\end{aligned}
\end{equation}

Using the product rule for $\mathcal{L}$, and that $\mathcal{L}=L-\vnormt{A}^{2}-1$
\begin{flalign*}
\frac{1}{2}\mathcal{L}\vnormt{A}^{2}
&=\vnormt{\nabla\vnormt{A}}^{2}+\vnormt{A}\mathcal{L}\vnormt{A}
=\vnormt{\nabla\vnormt{A}}^{2}+\vnormt{A}\Big(L\vnormt{A}-\vnormt{A}^{3}-\vnormt{A}\Big)\\
&\stackrel{\eqref{three30}}{=}\vnormt{\nabla\vnormt{A}}^{2}+2\vnormt{A}^{2}-\scon\langle A^{2},A\rangle
+\vnormt{\nabla A}^{2}-\vnormt{\nabla \vnormt{A}}^{2}-\vnormt{A}^{4}-\vnormt{A}^{2}\\
&=\vnormt{\nabla A}^{2}+\vnormt{A}^{2}-\vnormt{A}^{4}-\scon\langle A^{2},A\rangle\\
&\stackrel{\eqref{two12}}{\geq}
\Big(1+\frac{2}{\adimn}\Big)\vnormt{\nabla\vnormt{A}}^{2}
-\frac{2\sdimn}{\adimn}\vnormt{\nabla H}^{2}
+\vnormt{A}^{2}-\vnormt{A}^{4}-\scon\langle A^{2},A\rangle.
\end{flalign*}
Multiplying this inequality by $\eta^{2}$ and integrating by parts with Lemma \ref{lemma39.7},
\begin{flalign*}
&-2\int_{\Sigma}\eta\vnormt{A}\langle\nabla\eta,\nabla\vnormt{A}\rangle\gamma_{\sdimn}(x)dx
=-\frac{1}{2}\int_{\Sigma}\langle\nabla\eta^{2},\nabla\vnormt{A}^{2}\rangle\gamma_{\sdimn}(x)dx
=\frac{1}{2}\int_{\Sigma}\eta^{2}\mathcal{L}\vnormt{A}^{2}\gamma_{\sdimn}(x)dx\\
&\qquad\geq\int_{\Sigma}\eta^{2}\Big(\Big(1+\frac{2}{\adimn}\Big)\vnormt{\nabla\vnormt{A}}^{2}
-\frac{2\sdimn}{\adimn}\vnormt{\nabla H}^{2}
-\vnormt{A}^{4}-\scon\langle A^{2},A\rangle\Big)\gamma_{\sdimn}(x)dx.
\end{flalign*}
(We removed the $\vnormt{A}^{2}$ term since doing so only decreases the quantity on the right.)  Rearranging this inequality and then using the AMGM inequality in the form $b^{2}/\epsilon-2ab\geq -\epsilon a^{2}$,
\begin{equation}\label{two14}
\begin{aligned}
&\int_{\Sigma}\Big(\eta^{2}\vnormt{A}^{4}+\scon\eta^{2}\langle A^{2},A\rangle
+\frac{2\sdimn}{\adimn}\eta^{2}\vnormt{\nabla H}^{2}+\frac{1}{\epsilon}\eta^{2}\vnormt{A}^{2}\vnormt{\nabla \eta}^{2}\Big)\gamma_{\sdimn}(x)dx\\
&\qquad\qquad\geq \Big(1+\frac{2}{\adimn}-\epsilon\Big)\int_{\Sigma}\eta^{2}\vnormt{\nabla\vnormt{A}}^{2}\gamma_{\sdimn}(x)dx.
\end{aligned}
\end{equation}

Substituting \eqref{two14} into \eqref{two13},
\begin{flalign*}
&\int_{\Sigma}\eta^{2}\vnormt{A}^{4}\gamma_{\sdimn}(x)dx
\leq \frac{1+\epsilon}{1+\frac{2}{\adimn}-\epsilon}\int_{\Sigma}\eta^{2}\vnormt{A}^{4}\gamma_{\sdimn}(x)dx\\
&\quad
+10\int_{\Sigma}\Big[\eta^{2}\vnormt{\nabla H}^{2}+\eta^{2}(1+1/\epsilon)(1+\vnormt{\nabla\eta}^{2})\vnormt{A}^{2}+\abs{\scon}\eta^{2}\Big(\abs{\langle A,A^{2}\rangle}+\frac{\vnormt{A}^{4}}{H}\Big)\Big]\gamma_{\sdimn}(x)dx.
\end{flalign*}
%%%  2ab\leq a^2 /eps +b^2 eps
%%%  2a a^2 \leq  a^2 /eps + a^4 eps
Since $\Omega$ is convex, $A$ is negative definite with nonpositive diagonal entries, so that $H\geq\vnormt{A}$, i.e. $\vnormt{A}^{4}/H\leq\vnormt{A}^{3}$.  Also $\abs{\langle A,A^{2}\rangle}\leq\vnormt{A}^{3}$.  Using the AMGM inequality in the form $2a^{3}\leq a^{4}\epsilon+a^{2}/\epsilon$, we then get
\begin{flalign*}
\int_{\Sigma}\eta^{2}\vnormt{A}^{4}\gamma_{\sdimn}(x)dx
&\leq \Big(10\epsilon\abs{\scon}+\frac{1+\epsilon}{1+\frac{2}{\adimn}-\epsilon}\Big)\int_{\Sigma}\eta^{2}\vnormt{A}^{4}\gamma_{\sdimn}(x)dx\\
&\qquad
+10\int_{\Sigma}\eta^{2}\Big(\vnormt{\nabla H}^{2}+(1+(1+\abs{\scon})/\epsilon)(1+\vnormt{\nabla\eta}^{2})\vnormt{A}^{2}\Big)\gamma_{\sdimn}(x)dx.
\end{flalign*}
%
%$$
%\int_{\Sigma}\eta^{2}\vnormt{A}^{4}\gamma_{\sdimn}(x)dx
%\leq \frac{1+\epsilon}{1+\frac{2}{\adimn}-\epsilon}\int_{\Sigma}\eta^{2}\vnormt{A}^{4})\gamma_{\sdimn}(x)dx
%+c_{\epsilon,1}\int_{\Sigma}\vnormt{\nabla H}^{2}+\vnormt{A}^{2}\Big)\gamma_{\sdimn}(x)dx.
%$$

% Tr(A^3)\leq (Tr A^2)^3/2 ?
%  \|x\|_{3}^{3}\leq \|x\|_{2}^{3}
Now, choose $\epsilon<1/(20(n+1)(\abs{\scon}+1))$, so that $10\abs{\scon}\epsilon+\frac{1+\epsilon}{1+\frac{2}{\adimn}-\epsilon}<1$.  We then can move the $\eta^{2}\vnormt{A}^{4}$ term on the right side to the left side to get some $c_{\epsilon}>0$ such that
\begin{flalign*}
\int_{\Sigma}\eta^{2}\vnormt{A}^{4}\gamma_{\sdimn}(x)dx
&\leq c_{\epsilon}\int_{\Sigma}\eta^{2}\Big[\vnormt{\nabla H}^{2}+\vnormt{A}^{2}(1+\vnormt{\nabla\eta}^{2})\Big]\gamma_{\sdimn}(x)dx\\
&\leq c_{\epsilon}\int_{\Sigma}\eta^{2}\Big[\vnormt{A}^{2}(1+\vnormt{x}^{2}+\vnormt{\nabla\eta}^{2})\Big]\gamma_{\sdimn}(x)dx.
\end{flalign*}
In the last line, we used the inequality $\vnormt{\nabla H}^{2}\leq\vnormt{A}^{2}\vnormt{x}^{2}$.  This follows by \eqref{three6}, since $H(x)=\langle x,N\rangle+\scon$, so for any $1\leq i\leq\sdimn$, $\nabla_{e_{i}}H(x)=-\sum_{j=1}^{\sdimn}a_{ij}\langle x,e_{j}\rangle$.

We now choose a sequence of $\eta=\eta_{r}$ increasing to $1$ as $r\to\infty$ so that the $\vnormt{\nabla\eta}^{2}$ term vanishes.  This is possible due to the assumptions that $\delta<\infty$ and the Hausdorff dimension of $\partial\Omega\setminus\redA$ is at most $n-4$.  Such functions are constructed and this estimate is made in \cite[Lemma 6.4]{zhu16}:
\begin{equation}\label{two16}
\int_{\Sigma}\vnormt{A}^{2}\vnormt{\nabla\eta_{r}}^{2}\eta_{r}^{2}\gamma_{\sdimn}(x)dx
\leq c(\delta)(r^{\sdimn}e^{-(r-4)^{2}/4}+r^{-1}),\qquad\forall r>1.
\end{equation}

It therefore follows from Corollary \ref{cor5} applied to $\Sigma=\redA$ that
\begin{equation}\label{two15}
\int_{\Sigma}\vnormt{A}^{4}\gamma_{\sdimn}(x)dx<\infty.
\end{equation}
It then follows from \eqref{two14} that $\int_{\Sigma}\vnormt{\nabla\vnormt{A}}^{2}\gamma_{\sdimn}(x)dx<\infty.$

Finally, multiplying the above equality
$\mathcal{L}\vnormt{A}^{2}=2\vnormt{\nabla A}^{2}+2\vnormt{A}^{2}-2\vnormt{A}^{4}-2\scon\langle A^{2},A\rangle$ by $\eta^{2}$ and integrating by parts with Lemma \ref{lemma39.7}, we get
\begin{flalign*}
&2\int_{\Sigma}\eta^{2}(\vnormt{\nabla A}^{2}-\vnormt{A}^{4}-\scon\langle A^{2},A\rangle)\gamma_{\sdimn}(x)dx
\leq \int_{\Sigma}\eta^{2}\mathcal{L}\vnormt{A}^{2}\gamma_{\sdimn}(x)dx\\
&\qquad=-4\int_{\Sigma}\eta\vnormt{A}\langle\nabla\eta,\nabla\vnormt{A}\rangle\gamma_{\sdimn}(x)dx
\leq2\int_{\Sigma}\eta^{2}\vnormt{\nabla\vnormt{A}}^{2}+\vnormt{A}^{2}\vnormt{\nabla\eta}^{2}\rangle\gamma_{\sdimn}(x)dx.
\end{flalign*}
Then the $\vnormt{A}^{4}$ integral is finite by \eqref{two15}, the $\vnormt{\nabla\vnormt{A}}^{2}$ integral is finite, the last term has a finite integral by \eqref{two16}, so the integral of $\vnormt{\nabla A}^{2}$ is also finite.
\end{proof}

In the following Corollary, the Hausdorff dimension condition is needed to construct functions that converge to $1$ while being zero in a neighborhood of the singular set $\partial\Omega\setminus\redA$.  For details on this construction, see e.g. \cite[Section 5.2]{zhu16}.

\begin{cor}[\embolden{Integration by Parts, Version 2} {\cite[Lemma 5.4]{zhu16}}]\label{lemma39.79}
Let $\Omega\subset\R^{\adimn}$.  Let $f,g\colon\redA\to\R$ be $C^{2}$ functions.  Suppose the Hausdorff dimension of $\partial\Omega\setminus\redA$ is at most $\sdimn-7$.  Assume that
$$\int_{\Sigma}\vnormt{\nabla f}\vnormt{\nabla g}\gamma_{\sdimn}(x)dx<\infty,\qquad
\int_{\Sigma}\abs{f\mathcal{L}g}\gamma_{\sdimn}(x)dx<\infty,\qquad
\int_{\Sigma}\vnormt{f\nabla g}^{2}\gamma_{\sdimn}(x)dx<\infty.
$$
Then
$$\int_{\Sigma}f\mathcal{L}g\gamma_{\sdimn}(x)dx
%=\int_{\Sigma}g\mathcal{L}f\gamma_{\sdimn}(x)dx
=-\int_{\Sigma}\langle\nabla f,\nabla g\rangle\gamma_{\sdimn}(x)dx.
$$
\end{cor}

\begin{lemma}\label{lemma39pre}
Let $\Sigma\subset\R^{\adimn}$ be a $C^{\infty}$ hypersurface, with possibly nonempty boundary.  Assume $\exists$ $\scon\in\R$ such that $H(x)=\langle x,N(x)\rangle+\scon$ for all $x\in\Sigma$.  Let $\pcon\colonequals\pcon(\Sigma)$.  Assume that $\pcon<\infty$ and $\partial\Sigma\setminus\redA$ has Hausdorff dimension at most $\sdimn-7$.  Let $g$ be the eigenfunction guaranteed to exist by Lemma \ref{lemma40.9}.  Then for any $\epsilon>0$, we have
$$\abs{\int_{\Sigma} (HLg-gLH)\gamma_{\sdimn}(x)dx}<\epsilon(\int_{\Sigma} g^{2}\gamma_{\sdimn}(x)dx)^{1/2}(\int_{\Sigma} H^{2}\gamma_{\sdimn}(x)dx)^{1/2}.$$
\end{lemma}
\begin{proof}
Let $R>10$ let $\epsilon\colonequals R^{-3}$ and let $0<r_{1},\ldots,r_{k}<1$, $x^{(1)},\ldots,x^{(k)}\in\Sigma$ such that $\cup_{i=1}^{k}B(x^{(i)},r_{i})\supset B(0,r)\cap (\partial\Sigma\setminus\redA)$ and such that $\sum_{i=1}^{k}r_{i}^{\sdimn-4}<\epsilon$.  Such $r_{1},\ldots,r_{k}$ exist since $\sdimn-7>\sdimn-4$, and $\partial\Sigma\setminus\redA$ has Hausdorff dimension at most $\sdimn-7$.  For each $1\leq i\leq k$, we let $\phi_{i}\colon\Sigma\to[0,1]$ such that $\phi_{i}=1$ outside $B(x^{(i)},3r_{i})$,  $\phi_{i}=0$ inside $B(x^{(i)},2r_{i})$, and $\vnorm{\nabla\phi_{i}}\leq 2/r_{i}$ in $B(x^{(i)},3r_{i})\setminus B(x^{(i)},2r_{i})$.  Let also $\psi_{R}$ so that $\psi_{R}=0$ outside $B(0,R+2)$,  $\psi_{R}=1$ inside $B(0,R)$, and $\vnorm{\nabla\psi_{R}}\leq 2$ in $B(0,R+2)\setminus B(0,R)$.  Finally, define $\phi=\phi_{R}\colon\Sigma\to[0,1]$  so that $\phi\colonequals\min(\psi_{R},\min_{1\leq i\leq k}\phi_{i})$.  Note that $\phi$ is Lipschitz with compact support and $\vnorm{\nabla\phi}\leq \max(\vnorm{\nabla\psi_{R}},\max_{1\leq i\leq k}\vnorm{\nabla\psi_{i}})$.

We now integrate by parts with Lemma \ref{lemma39.7} to get
\begin{flalign*}
&\int_{\Sigma} (HLg-gLH)\gamma_{\sdimn}(x)dx
=\int_{\Sigma} (H\mathcal{L}g-g\mathcal{L}H)\gamma_{\sdimn}(x)dx\\
&\qquad\qquad=\int_{\Sigma} \phi^{2}(H\mathcal{L}g-g\mathcal{L}H)\gamma_{\sdimn}(x)dx+\int_{\Sigma} (1-\phi^{2})(H\mathcal{L}g-g\mathcal{L}H)\gamma_{\sdimn}(x)dx\\
&\qquad\qquad=2\int_{\Sigma} \phi\langle\nabla\phi, H\nabla g-g\nabla H\rangle\gamma_{\sdimn}(x)dx+\int_{\Sigma} (1-\phi^{2})(H\mathcal{L}g-g\mathcal{L}H)\gamma_{\sdimn}(x)dx
\end{flalign*}
The second term is made small by choosing $R$ large, so we focus on the first term.  We bound the first by
\begin{flalign*}
&\int_{\Sigma} \phi\abs{\langle\nabla\phi, H\nabla g-g\nabla H\rangle}\gamma_{\sdimn}(x)dx
=\int_{\Sigma} \phi \abs{g}\abs{\langle\nabla\phi, H\frac{\nabla g}{g}-\nabla H\rangle}\gamma_{\sdimn}(x)dx\\
&\qquad\leq\int_{\Sigma} \phi\abs{g}\Big(\vnorm{\nabla\phi}\abs{H}\vnorm{\nabla \log g}+\vnorm{\nabla H}\Big)\gamma_{\sdimn}(x)dx\\
&\qquad\leq(\int_{\Sigma} g^{2}\gamma_{\sdimn}(x)dx)^{1/2}\int_{\Sigma} \phi^{2}\vnorm{\nabla\phi}^{2}H^{2}\vnorm{\nabla \log g}^{2}+\phi^{2}\vnorm{\nabla H}^{2}\gamma_{\sdimn}(x)dx)^{1/2}\\
&\qquad\leq(\int_{\Sigma} g^{2}\gamma_{\sdimn}(x)dx)^{1/2}\int_{\Sigma} \phi^{2}\vnorm{\nabla\phi}^{2}(\vnorm{x}+\abs{\lambda})\vnorm{\nabla \log g}^{2}+\phi^{2}\vnorm{A}^{2}\gamma_{\sdimn}(x)dx)^{1/2}\\
&\qquad\leq(\int_{\Sigma} g^{2}\gamma_{\sdimn}(x)dx)^{1/2}\int_{\Sigma} \phi^{2}\vnorm{\nabla\phi}^{2}(R+\abs{\lambda})\vnorm{\nabla \log g}^{2}+\phi^{2}\vnorm{A}^{2}\gamma_{\sdimn}(x)dx)^{1/2}.
\end{flalign*}

This term is then small by Lemma \ref{lemma28}.
\end{proof}

\begin{lemma}\label{lemma39}
Let $\Sigma\subset\R^{\adimn}$ be a $C^{\infty}$ hypersurface, with possibly nonempty boundary.  Assume $\exists$ $\scon\in\R$ such that $H(x)=\langle x,N(x)\rangle+\scon$ for all $x\in\Sigma$.  Let $\pcon\colonequals\pcon(\Sigma)$.  Assume that $\pcon<\infty$.  If $H\geq0$,  and if $\scon<0$, then $1\leq\pcon\leq2$.
\end{lemma}
\begin{proof}
First, $\pcon\geq1$ follows from the definition of $\pcon$ in \eqref{seven0} and Lemma \ref{lemma45}.  More specifically, if $\phi=\phi_{R}$ denotes the cutoff function constructed in Lemma \ref{lemma39pre}, we have by the product rule (Remark \ref{rk20})
\begin{flalign*}
\pcon
&\geq \frac{\int_{\Sigma}\phi\langle v,N\rangle L(\phi \langle v,N\rangle)\gamma_{\sdimn}(x)dx}{\int_{\Sigma}\phi^{2}\langle v,N\rangle^{2}\gamma_{\sdimn}(x)dx}\\
&= \frac{\int_{\Sigma}\Big(\phi^{2}\langle v,N\rangle^{2}+\phi\langle v,N\rangle^{2}\mathcal{L}\phi+2\phi\langle v,N\rangle\langle\nabla\langle v,N\rangle,\nabla\phi\rangle\Big)\gamma_{\sdimn}(x)dx}{\int_{\Sigma}\phi^{2}\langle v,N\rangle^{2}\gamma_{\sdimn}(x)dx}.
\end{flalign*}
As $R\to\infty$, the last two terms in the top of the integral converge to zero, by construction of $\phi$, so that $\pcon\geq1$.
%$$\pcon\colonequals\sup_{f\in C_{0}^{\infty}(\Sigma)}\frac{\int_{\Sigma}fLf\gamma_{\sdimn}(x)dx}{\int_{\Sigma}f^{2}\gamma_{\sdimn}(x)dx}.$$
Now, since $Lg=\pcon g$ from Lemma \ref{lemma40.9}, by Lemma \ref{lemma39pre},
\begin{equation}\label{three7.3}
\begin{aligned}
\pcon\int_{\Sigma}Hg\gamma_{\sdimn}(x)dx
&=\int_{\Sigma}HLg\gamma_{\sdimn}(x)dx=\int_{\Sigma}gLH\gamma_{\sdimn}(x)dx+O(\epsilon)\\
&\stackrel{\eqref{three9}}{=}\int_{\Sigma}(2Hg+\scon g\vnormt{A}^{2})\gamma_{\sdimn}(x)dx+O(\epsilon).
\end{aligned}
\end{equation}
Also, by Corollary \ref{lemma39.79}, $\int_{\Sigma}\mathcal{L}g\gamma_{\sdimn}(x)dx=\lim_{R\to\infty}\int_{\Sigma}\phi_{R}\mathcal{L}g\gamma_{\sdimn}(x)dx=0$, so
$$\pcon\int_{\Sigma}g\gamma_{\sdimn}(x)dx=\int_{\Sigma}Lg\gamma_{\sdimn}(x)dx\stackrel{\eqref{three4.5}}{=}\int_{\Sigma}g(1+\vnormt{A}^{2})\gamma_{\sdimn}(x)dx.$$
That is,
\begin{equation}\label{three7.4}
\int_{\Sigma}g\vnormt{A}^{2}\gamma_{\sdimn}(x)dx=(\pcon-1)\int_{\Sigma}g\gamma_{\sdimn}(x)dx.
\end{equation}
Finally, combining \eqref{three7.3} and \eqref{three7.4},
$$(\pcon-2)\int_{\Sigma} Hg\gamma_{\sdimn}(x)dx=\scon\int_{\Sigma}g\vnormt{A}^{2}\gamma_{\sdimn}(x)dx+O(\epsilon)=\scon(\pcon-1)\int_{\Sigma}g\gamma_{\sdimn}(x)dx+O(\epsilon).$$
So, by our assumptions on $\scon<0,g\geq0,H\geq0$, the right side is nonpositive.  In order for the left side to be nonpositive, we must have $\pcon\leq2$.
\end{proof}

\section{Perturbations using H or an Eigenfunction}\label{sech}

\begin{lemma}[\embolden{Perturbation using an Eigenfunction}]\label{lemma74}
Let $\Sigma$ be a symmetric, orientable $C^{\infty}$ hypersurface with $\pcon(\Sigma)<\infty$.  Assume $\exists$ $\scon\in\R$ such that $H(x)=\langle x,N(x)\rangle+\scon$ for all $x\in\Sigma$.  Let $g\colon\Sigma\to(0,\infty)$ from Lemma \ref{lemma40.9} so that $Lg=\pcon g$.  Then
$$\int_{\Sigma}(1+g)L(1+g)\gamma_{\sdimn}(x)dx
=\int_{\Sigma}\Big(\pcon(1+g)^{2}+\vnormt{A}^{2}+(1-\pcon)\Big)\gamma_{\sdimn}(x)dx.$$
\end{lemma}
\begin{proof}
Integrating by parts with Corollary \ref{lemma39.79},
\begin{flalign*}
&\int_{\Sigma}(1+g)L(1+g)\gamma_{\sdimn}(x)dx
=\int_{\Sigma}\Big(gLg+2Lg+L(1)\Big)\gamma_{\sdimn}(x)dx\\
&\qquad=\int_{\Sigma}\Big(\pcon g^{2}+2\pcon g+\vnormt{A}^{2}+1\Big)\gamma_{\sdimn}(x)dx
=\int_{\Sigma}\Big(\pcon(1+g)^{2}+\vnormt{A}^{2}+(1-\pcon)\Big)\gamma_{\sdimn}(x)dx.
\end{flalign*}
\end{proof}

\begin{cor}\label{hcor}
Let $\Omega\subset\R^{\adimn}$ minimize Problem \ref{prob1}.  Let $\Sigma\colonequals\redA$.   Then by Lemma \ref{varlem}, $\exists$ $\scon\in\R$ such that $H(x)=\langle x,N(x)\rangle+\scon$.  Assume that $H>0$, $\scon<0$ and
$$\int_{\Sigma}(\vnormt{A}^{2}-1)\gamma_{\sdimn}(x)dx>0.$$
Then, after rotating $\Omega$, $\exists$ $r>0$ and $\exists$ $0\leq k\leq \sdimn$ such that $\Sigma= rS^{k}\times\R^{\sdimn-k}$.
\end{cor}
\begin{proof}
From Lemma \ref{lemma39}, $1\leq\pcon\leq 2$.  Let $g\colon\Sigma\to(0,\infty)$ be the eigenfunction guaranteed to exist by Lemma \ref{lemma40.9}.  If $g$ is not constant, then we can multiply it by a constant as necessary so that $\int_{\Sigma}(1+g)\gamma_{\sdimn}(x)dx=0$.  Since $1\leq\pcon\leq2$ and since $\int_{\Sigma}(\vnormt{A}^{2}-1)\gamma_{\sdimn}(x)dx>0$, Lemma \ref{lemma74} contradicts Lemma \ref{varlem2}.  So, we must assume that $g$ is constant.

If $g$ is constant, then $L(1)=g(\vnormt{A}^{2}+1)=\pcon g$.  That is, $\vnormt{A}^{2}$ is equal to a constant $c\in\R$.  Choose $b\in\R$ such that $\int_{\Sigma}(H(x)+b)\gamma_{\sdimn}(x)dx=0$.  By \eqref{three9}, $LH=2H+\scon\vnormt{A}^{2}=2H+\scon c$.  So, using Lemma \ref{lemma39.1} and Corollary \ref{lemma39.79}, and also using the definition of $b$,
\begin{flalign*}
&\frac{1}{2}\int_{\Sigma}(H+b)L(H+b)\gamma_{\sdimn}(x)dx
=\frac{1}{2}\int_{\Sigma}(H+b)(2H+\scon c+b(c+1))\gamma_{\sdimn}(x)dx\\
&\qquad\qquad=\int_{\Sigma}(H+b)H\gamma_{\sdimn}(x)dx
=\int_{\Sigma}H^{2}\gamma_{\sdimn}(x)dx-\frac{[\int_{\Sigma}H\gamma_{\sdimn}(x)dx]^{2}}{\int_{\Sigma}\gamma_{\sdimn}(y)dy}.
\end{flalign*}
From the Cauchy-Schwarz inequality, $-[\int_{\Sigma}H\gamma_{\sdimn}(x)dx]^{2}\geq-\int_{\Sigma}H^{2}\gamma_{\sdimn}(x)dx\int_{\Sigma}\gamma_{\sdimn}(y)dy$, with equality only if $H$ is constant.  Therefore
$$\int_{\Sigma}(H+b)L(H+b)\gamma_{\sdimn}(x)dx\geq0,$$
with equality only if $H$ is constant.  So, if $H$ is not constant, then this inequality contradicts Lemma \ref{varlem2}.  So, we must assume that $H$ is constant.  If $H$ is constant and if $\vnormt{A}$ is constant, then $H/\vnormt{A}$ is constant.  It follows from \cite[Proof of Theorem 10.1]{colding12a} \cite[p.187-188]{huisken93} that $\Sigma$ is then a round cylinder.  (If $\vnormt{A}=0$ on $\Sigma$, then $\Sigma$ is a rotation of $rS^{0}\times\R^{\sdimn}$ for some $r>0$.)
\end{proof}
\begin{remark}
It seems difficult to classify self-shrinkers such that $\vnormt{A}$ is constant \cite{guang14}.
\end{remark}

\section{Huisken-type Classification}\label{secdif}

The following Lemma is a routine generalization of \cite[Lemma 10.14]{colding12a} and  \cite[Lemma 7.3]{zhu16}.
\begin{lemma}\label{lemma37}
Let $\Omega\subset\R^{\adimn}$ such that $\partial\Omega\setminus\redA$ has Hausdorff dimension at most $\sdimn-7$.  Let $\Sigma\colonequals\redA$.  Let $\scon\in\R$.  Assume $\pcon(\Sigma)<\infty$.  Assume $H(x)>0$ for all $x\in\Sigma$ and $H(x)=\langle x,N(x)\rangle+\scon$, $\forall\,x\in\Sigma$.  Then
$$\int_{\Sigma}\Big\|\vnormt{A}\nabla \log H -\nabla\vnormt{A}\Big\|^{2}\gamma_{\sdimn}(x)dx
\leq \scon\int_{\Sigma}\Big(\frac{\vnormt{A}^{4}}{H}+\langle A^{2},A\rangle\Big)\gamma_{\sdimn}(x)dx.$$
\end{lemma}
\begin{proof}
First, let $\mathcal{L}\colonequals\Delta-\langle x,\nabla\rangle$ as in \eqref{three4.3}.  Using \eqref{logeq} and integrating by parts by Lemma \ref{lemma39.2} and Corollaries \ref{lemma39.79} and \ref{cor5},

\begin{flalign*}
&-\int_{\Sigma}\langle\nabla\vnormt{A}^{2},\nabla\log H\rangle\gamma_{\sdimn}(x)dx
=\int_{\Sigma}\vnormt{A}^{2}\mathcal{L}\log H\gamma_{\sdimn}(x)dx\\
&\qquad\qquad=\int_{\Sigma}\vnormt{A}^{2}\Big(-\vnormt{\nabla\log H}^{2}+1-\vnormt{A}^{2}+\scon\frac{\vnormt{A}^{2}}{H}\Big)\gamma_{\sdimn}(x)dx.
\end{flalign*}

Note that
\begin{flalign*}
\Big\|\vnormt{A}\nabla \log H -\nabla\vnormt{A}\Big\|^{2}
&=\vnormt{A}^{2}\vnormt{\nabla\log H}^{2}+\vnormt{\nabla\vnormt{A}}^{2}-2\vnormt{A}\langle\nabla \log H ,\nabla\vnormt{A}\rangle\\
&=\vnormt{A}^{2}\vnormt{\nabla\log H}^{2}+\vnormt{\nabla\vnormt{A}}^{2}-\langle\nabla \log H ,\nabla\vnormt{A}^{2}\rangle.
\end{flalign*}

Integrating and combining the above,
\begin{equation}\label{three31}
\int_{\Sigma}\Big\|\vnormt{A}\nabla \log H -\nabla\vnormt{A}\Big\|^{2}\gamma_{\sdimn}(x)dx
=\int_{\Sigma}\Big(\vnormt{\nabla\vnormt{A}}^{2}+\vnormt{A}^{2}-\vnormt{A}^{4}+\scon\frac{\vnormt{A}^{4}}{H}\Big)\gamma_{\sdimn}(x)dx.
%&=\int_{\Sigma}\Big(\vnormt{}^{2}\Big)\gamma_{\sdimn}(x)dx\\
\end{equation}

We manipulate the first term on the right.  Integrating by parts with Corollary \ref{lemma39.79},
\begin{flalign*}
\int_{\Sigma}\vnormt{\nabla\vnormt{A}}^{2}\gamma_{\sdimn}(x)dx
&=-\int_{\Sigma}\vnormt{A}\mathcal{L}\vnormt{A}\gamma_{\sdimn}(x)dx\\
&\stackrel{\eqref{three4.5}}{=}\int_{\Sigma}\Big(-\vnormt{A}L\vnormt{A}+\vnormt{A}^{4}+\vnormt{A}^{2}\Big)\gamma_{\sdimn}(x)dx\\
&\stackrel{\eqref{three30}}{\leq}\int_{\Sigma}\Big(\vnormt{A}^{4}+\scon\langle A^{2},A\rangle-\vnormt{A}^{2}\Big)\gamma_{\sdimn}(x)dx.
\end{flalign*}

Combining this with \eqref{three31},
$$\int_{\Sigma}\Big\|\vnormt{A}\nabla \log H -\nabla\vnormt{A}\Big\|^{2}\gamma_{\sdimn}(x)dx
\leq \int_{\Sigma}\Big(\scon\frac{\vnormt{A}^{4}}{H}+\scon\langle A^{2},A\rangle\Big)\gamma_{\sdimn}(x)dx.
$$

\end{proof}
\begin{remark}
Repeating the above calculation and replacing $\log H$ with $\log(H-\scon)$ gives
$$
\mathcal{L}\log(H-\scon)
=-\vnormt{\nabla\log (H-\scon)}^{2}+1-\vnormt{A}^{2}+\frac{\scon}{H-\scon}.
$$
$$\int_{\Sigma}\Big\|\vnormt{A}\nabla \log (H-\scon) -\nabla\vnormt{A}\Big\|^{2}\gamma_{\sdimn}(x)dx
\leq \scon\int_{\Sigma}\Big(\frac{\vnormt{A}^{2}}{H-\scon}+\langle A^{2},A\rangle\Big)\gamma_{\sdimn}(x)dx.
$$
Recovering the main result of \cite[Theorem 4.1]{cheng15} with a slightly different proof.
\end{remark}

For curves in the plane, it is known that circles and lines are the only solutions of $H(x)=\langle x,N\rangle+\scon$ when $\scon>0$ \cite[Theorem 1.5]{guang15} \cite[Theorem 1.4]{chang17}.  For convex surfaces, we extend this argument to arbitrary dimensions.

\begin{cor}[\embolden{Huisken-type classification, $\scon>0$}]\label{huisk}
Let $\Omega\subset\R^{\adimn}$ be convex such that $\partial\Omega\setminus\redA$ has Hausdorff dimension at most $\sdimn-7$.  Let $\scon>0$.  Assume $\pcon(\Sigma)<\infty$.  Suppose $H(x)=\langle x,N(x)\rangle+\scon$, $\forall\,x\in\Sigma$.  Then $H/\vnormt{A}$ is constant on $\Sigma$.
\end{cor}
\begin{proof}
Let $B$ be a positive definite $\sdimn\times\sdimn$ matrix.  From logarithmic convexity of the $\ell_{p}$ norms (or H\"{o}lder's inequality),
$$\Big(\mathrm{Tr}(B^{2})\Big)^{1/2}\leq\Big(\mathrm{Tr}(B)\Big)^{1/4}\Big(\mathrm{Tr}(B^{3})\Big)^{1/4}.$$
Since $\Omega$ is convex, $A$ is negative definite, and all of the diagonal entries of $A$ are nonpositive.  So, taking the fourth power of this inequality,
$$\vnormt{A}^{4}\leq H\abs{\langle A^{2},A\rangle}.$$
That is,
$$\frac{\vnormt{A}^{4}}{H}\leq \abs{\langle A^{2},A\rangle}.$$
So, from Lemma \ref{lemma37}, $\scon>0$, and again using that $A$ is negative definite and $H>0$,
$$\int_{\Sigma}\Big\|\vnormt{A}\nabla \log H -\nabla\vnormt{A}\Big\|^{2}\gamma_{\sdimn}(x)dx
\leq \scon\int_{\Sigma}\Big(\frac{\vnormt{A}^{4}}{H}+\langle A^{2},A\rangle\Big)\gamma_{\sdimn}(x)dx
\leq0.$$
Therefore, $\nabla_{e_{i}}\log H=\nabla_{e_{i}}\log\vnormt{A}$ for all $1\leq i\leq \sdimn$.  That is, $\vnormt{A}/H$ is constant on $\Sigma$.
\end{proof}

\begin{cor}[\embolden{Huisken-type classification, $\scon>0$}]
Let $\Omega\subset\R^{\adimn}$ be convex such that $\partial\Omega\setminus\redA$ has Hausdorff dimension at most $\sdimn-7$.  Let $\scon>0$.  Assume $\pcon(\Sigma)<\infty$.  Suppose $H(x)=\langle x,N(x)\rangle+\scon$, $\forall\,x\in\Sigma$.  Then, after rotating $\Omega$, $\exists$ $r>0$ and $\exists$ $0\leq k\leq\sdimn$ such that $\Sigma=rS^{k}\times\R^{\sdimn-k}$.
\end{cor}
\begin{proof}
Since $\vnormt{A}/H$ is constant on $\Sigma$, it follows from \cite[Proof of Theorem 10.1]{colding12a} \cite[p.187-188]{huisken93} that $\Sigma$ is then a round cylinder.
\end{proof}

\section{Random Almost-Eigenfunctions}\label{secrand}

In this section, we let $\E$ denote the expected value of a random variable.

\begin{lemma}\label{lemma80}
Let $v$ be a uniformly distributed random vector in $S^{\sdimn}$.  Let $a,b\in S^{\sdimn}$.  Then
$$\E\langle v,a\rangle\langle v,b\rangle=\frac{\langle a,b\rangle}{\adimn}.$$
%$$\E\langle v,a\rangle^{2}\langle v,b\rangle^{2}=\frac{1+2\langle a,b\rangle^{2}}{(\adimn)(\sdimn+3)}.$$
\end{lemma}
\begin{proof}
Let $\theta\colonequals\cos^{-1}(\langle a,b\rangle)$.  In the case $n=1$ we have
\begin{flalign*}
\E\langle v,a\rangle\langle v,b\rangle
&=\frac{1}{2\pi}\int_{0}^{2\pi}\cos(t)\cos(t+\theta)dt
=\frac{1}{2\pi}\int_{0}^{2\pi}\cos(t)[\cos(t)\cos(\theta) - \sin(t)\sin(\theta)]dt\\
&=\cos\theta\frac{1}{2\pi}\int_{0}^{2\pi}\cos^{2}(t)dt
=\cos\theta\frac{1}{2\pi}\int_{0}^{2\pi}(1/2)(1+\cos(2t))dt
=\frac{\cos\theta}{2}
=\frac{\langle a,b\rangle}{2}.
\end{flalign*}
For any $n>1$, there then exists $c_{\adimn}\in\R$ such that $\E\langle v,a\rangle\langle v,b\rangle=c_{\adimn}\langle a,b\rangle$.  Choosing $a=b$, we get $c_{\adimn}=c_{\adimn}\vnormt{a}^{2}=\E\langle v,a\rangle^{2}$.  And summing over an orthonormal basis of $a$ in $\R^{\adimn}$ gives $(\adimn)c_{\adimn}=\E \vnormt{v}^{2}=1$.  That is, $c_{\adimn}=1/(\adimn)$.

\end{proof}

Let $A_{x}$ denote the matrix $A$ at the point $x$.  Let $\Pi=\Pi_{x}\colon\R^{\adimn}\to\R^{\sdimn}$ be the linear projection of $\R^{\adimn}$ onto the tangent space at $x$, viewed as $\R^{\sdimn}$ itself.  So, if $A$ is diagonal at $x$, then there exists a basis $u_{1},\ldots,u_{\adimn}$ of $\R^{\adimn}$ such that $\Pi_{x}u_{\adimn}=0$, $\Pi_{x}u_{i}=u_{i}$ and $A_{x}\Pi_{x}e_{i}=a_{ii}u_{i}$ for all $1\leq i\leq\sdimn$.  Let $p\colonequals\int_{\Sigma}\gamma_{\sdimn}(x)dx$.

\begin{lemma}[\embolden{Independent Bilinear Perturbation}]\label{lemma85}
Let $\Sigma$ be an orientable $C^{\infty}$ hypersurface.  Assume that there exists $\lambda\in\R$ such that for all $x\in\Sigma$, $H(x)=\langle x,N\rangle+\lambda$.  Then there exists $v,w\in S^{\sdimn}$ so that, if $m\colonequals\frac{1}{p}\int_{\Sigma}\langle v,N\rangle\langle w,N\rangle\gamma_{\sdimn}(y)dy$, then
\begin{flalign*}
&(\adimn)^{2}\int_{\Sigma}(\langle v,N\rangle\langle w,N\rangle-m)L(\langle v,N\rangle\langle w,N\rangle-m)\gamma_{\sdimn}(x)dx\\
&\geq\int_{\Sigma}(1-\vnormt{A_{x}}^{2})\gamma_{\sdimn}(x)dx\cdot\Big(1-\frac{1}{p^{2}}\int_{\Sigma}\int_{\Sigma}\langle N(y),N(z)\rangle^{2} \gamma_{\sdimn}(y)\gamma_{\sdimn}(z)dydz\Big)\\
&\qquad-\frac{2}{p}\int_{\Sigma}\int_{\Sigma}\vnormt{A_{x}\Pi_{x}(N(y))}^{2}\big)\gamma_{\sdimn}(y)\gamma_{\sdimn}(x)dydx.
\end{flalign*}
\end{lemma}
\begin{proof}
Let $v,w\in S^{\sdimn}$.  From Remark \ref{rk20},
$$
L\langle v,N\rangle\langle w,N\rangle
=\langle v,N\rangle\langle w,N\rangle
+2\langle\nabla \langle v,N\rangle,\nabla \langle w,N\rangle\rangle
-\vnormt{A}^{2}\langle v,N\rangle\langle w,N\rangle.
$$
So, if we define $m\colonequals\frac{1}{p}\int_{\Sigma}\langle v,N\rangle\langle w,N\rangle\gamma_{\sdimn}(y)dy$, then
$$
L(\langle v,N\rangle\langle w,N\rangle-m)
\stackrel{\eqref{three4.5}}{=}\langle v,N\rangle\langle w,N\rangle
+2\langle\nabla \langle v,N\rangle,\nabla \langle w,N\rangle\rangle
-\vnormt{A}^{2}(\langle v,N\rangle\langle w,N\rangle+m)
-m.
$$
Therefore,
\begin{flalign*}
&\int_{\Sigma}(\langle v,N\rangle\langle w,N\rangle-m)L(\langle v,N\rangle\langle w,N\rangle-m)\gamma_{\sdimn}(x)dx\\
&=\int_{\Sigma}\Big(\langle v,N\rangle^{2}\langle w,N\rangle^{2}+2(\langle v,N\rangle\langle w,N\rangle-m)\langle\nabla \langle v,N\rangle,\nabla \langle w,N\rangle\rangle
-2m\langle v,N\rangle\langle w,N\rangle\\
&\qquad\qquad-\vnormt{A}^{2}(\langle v,N\rangle^{2}\langle w,N\rangle^{2}-m^{2}) +m^{2}\Big)\gamma_{\sdimn}(x)dx\\
&=\int_{\Sigma}\Big(\langle v,N\rangle^{2}\langle w,N\rangle^{2}+2(\langle v,N\rangle\langle w,N\rangle-m)\langle\nabla \langle v,N\rangle,\nabla \langle w,N\rangle\rangle\\
&\qquad\qquad-\vnormt{A}^{2}(\langle v,N\rangle^{2}\langle w,N\rangle^{2}-m^{2})+m^{2}\Big)\gamma_{\sdimn}(x)dx -2pm^{2}.
\end{flalign*}

From Lemma \ref{lemma80}, if $v,w$ are uniformly distributed in $S^{\sdimn}$, then
$$\E \langle v,N\rangle^{2}\langle w,N\rangle^{2}=\frac{1}{(\adimn)^{2}}.$$
\begin{flalign*}
\E m^{2}
&=\frac{1}{p^{2}}\int_{\Sigma}\int_{\Sigma}\E \langle v,N(y)\rangle\langle w,N(y)\rangle\langle v,N(z)\rangle\langle w,N(z)\rangle
 \gamma_{\sdimn}(y)\gamma_{\sdimn}(z)dydz\\
&=\frac{1}{p^{2}}\int_{\Sigma}\int_{\Sigma}\frac{\langle N(y),N(z)\rangle^{2}}{(\adimn)^{2}} \gamma_{\sdimn}(y)\gamma_{\sdimn}(z)dydz.
\end{flalign*}

Let $e_{1},\ldots,e_{\sdimn}\in\R^{\adimn}$ be an orthonormal frame for $\Sigma$ (embedded into the tangent space of $\Sigma$ so that $\langle e_{i},N(x)\rangle=0$ for all $1\leq i\leq \sdimn$) such that $A$ is a diagonal matrix at $x$.  Then
$$\langle\nabla \langle v,N\rangle,\nabla \langle w,N\rangle\rangle
\stackrel{\eqref{three5g}}{=}\left\langle\sum_{i=1}^{\sdimn}a_{ii}\langle v,e_{i}\rangle e_{i},\sum_{j=1}^{\sdimn}a_{jj}\langle w,e_{j}\rangle e_{j}\right\rangle
=\sum_{i=1}^{\sdimn}a_{ii}^{2}\langle v,e_{i}\rangle\langle w,e_{i}\rangle.$$
So, using Lemma \ref{lemma80} and $\sum_{i=1}^{\sdimn}a_{ii}^{2}=\vnormt{A_{x}}^{2}$,
\begin{flalign*}
&\E \langle v,N(y)\rangle\langle w,N(y)\rangle\langle\nabla \langle v,N(x)\rangle,\nabla \langle w,N(x)\rangle\rangle
=\sum_{i=1}^{\sdimn}a_{ii}^{2}\E \langle v,N(y)\rangle\langle w,N(y)\rangle\langle v,e_{i}\rangle\langle w,e_{i}\rangle\\
&\qquad\qquad=\sum_{i=1}^{\sdimn}a_{ii}^{2}\frac{\langle e_{i},N(y)\rangle^{2}}{(\adimn)^{2}}
=\frac{\vnorm{A_{x} \Pi_{x}(N(y))}^{2}}{(\adimn)^{2}}.
\end{flalign*}
In particular, if $x=y$, we have $\Pi(N(y))=0$, so that
$$
\E\langle v,N(x)\rangle\langle w,N(x)\rangle\langle\nabla \langle v,N(x)\rangle,\nabla \langle w,N(x)\rangle\rangle
=0.
$$
Also,
\begin{flalign*}
\E m\langle\nabla \langle v,N(x)\rangle,\nabla \langle w,N(x)\rangle\rangle
&=\frac{1}{p}\int_{\Sigma}\E \langle v,N(y)\rangle\langle w,N(y)\rangle\langle\nabla \langle v,N(x)\rangle,\nabla \langle w,N(x)\rangle\rangle\gamma_{\sdimn}(y)dy\\
&=\frac{1}{p}\int_{\Sigma}\frac{\vnorm{A_{x} \Pi_{x}(N(y))}^{2}}{(\adimn)^{2}}\gamma_{\sdimn}(y)dy.
\end{flalign*}

Combining the above calculations,
\begin{flalign*}
&\E\int_{\Sigma}(\langle v,N\rangle\langle w,N\rangle-m)L(\langle v,N\rangle\langle w,N\rangle-m)\gamma_{\sdimn}(x)dx\\
&=\E\int_{\Sigma}\Big(\langle v,N\rangle^{2}\langle w,N\rangle^{2}+2(\langle v,N\rangle\langle w,N\rangle-m)
\langle\nabla \langle v,N\rangle,\nabla \langle w,N\rangle\rangle\\
&\qquad-\vnormt{A}^{2}(\langle v,N\rangle^{2}\langle w,N\rangle^{2}-m^{2})+m^{2}\Big)\gamma_{\sdimn}(x)dx-2p\E m^{2}\\
&=\frac{1}{(\adimn)^{2}}\int_{\Sigma}\Big(1-\frac{2}{p}\int_{\Sigma}\frac{\vnorm{A_{x} \Pi(N(y))}^{2}}{(\adimn)^{2}}\gamma_{\sdimn}(y)dy\\
&\qquad-\vnormt{A_{x}}^{2}+\frac{(1+\vnormt{A_{x}}^{2})}{p^{2}}\int_{\Sigma}\int_{\Sigma}\langle N(y),N(z)\rangle^{2} \gamma_{\sdimn}(y)\gamma_{\sdimn}(z)dydz
\Big)\gamma_{\sdimn}(x)dx\\
&\qquad-\frac{2}{(\adimn)^{2}}\frac{1}{p}\int_{\Sigma}\int_{\Sigma}\langle N(y),N(z)\rangle^{2} \gamma_{\sdimn}(y)\gamma_{\sdimn}(z)dydz.
\end{flalign*}

Simplifying a bit using the definition of $p$, we get
\begin{flalign*}
&(\adimn)^{2}\E\int_{\Sigma}(\langle v,N\rangle\langle w,N\rangle-m)L(\langle v,N\rangle\langle w,N\rangle-m)\gamma_{\sdimn}(x)dx\\
&=\int_{\Sigma}(1-\vnormt{A_{x}}^{2})\gamma_{\sdimn}(x)dx
+\int_{\Sigma}(\vnormt{A_{x}}^{2}-1)\gamma_{\sdimn}(x)dx\frac{1}{p^{2}}\int_{\Sigma}\int_{\Sigma}\langle N(y),N(z)\rangle^{2} \gamma_{\sdimn}(y)\gamma_{\sdimn}(z)dydz\\
&\qquad-\frac{2}{p}\int_{\Sigma}\int_{\Sigma}\vnormt{A_{x}\Pi(N(y))}^{2}\big)\gamma_{\sdimn}(y)\gamma_{\sdimn}(x)dydx.
\end{flalign*}
Therefore, $\exists$ $v,w\in S^{\sdimn}$ such that $\int_{\Sigma}(\langle v,N\rangle\langle w,N\rangle-m)L(\langle v,N\rangle\langle w,N\rangle-m)\gamma_{\sdimn}(x)dx$ exceeds or equals the above expected value.
\end{proof}

\begin{remark}
If we repeat the proof of Lemma \ref{lemma85} for $v,w$ which are conditioned to satisfy $\langle v,w\rangle=0$, then the result is the same.
\end{remark}

As above, let $A_{x}$ denote the matrix $A$ at the point $x$.  Let $\Pi=\Pi_{x}\colon\R^{\adimn}\to\R^{\sdimn}$ be the linear projection of $\R^{\adimn}$ onto the tangent space at $x$, viewed as $\R^{\sdimn}$ itself.  Let $p\colonequals\int_{\Sigma}\gamma_{\sdimn}(x)dx$.

\begin{cor}\label{cor7}
Let $\Sigma$ be an orientable $C^{\infty}$ hypersurface.  Assume that there exists $\lambda\in\R$ such that for all $x\in\Sigma$, $H(x)=\langle x,N\rangle+\lambda$.  Then there exists $v,w\in S^{\sdimn}$ so that, if $m\colonequals\frac{1}{p}\int_{\Sigma}\langle v,N\rangle\langle w,N\rangle\gamma_{\sdimn}(y)dy$, then
\begin{flalign*}
&(\adimn)^{2}\int_{\Sigma}(\langle v,N\rangle\langle w,N\rangle-m)L(\langle v,N\rangle\langle w,N\rangle-m)\gamma_{\sdimn}(x)dx\\
&\geq\frac{1}{p^{2}}\int_{\Sigma\times\Sigma\times\Sigma}(1-\vnormt{A_{x}}^{2}-2\vnorm{A_{y}}_{2\to2}^{2})\vnormt{\Pi_{y}(N(z))}^{2}
\gamma_{\sdimn}(x)\gamma_{\sdimn}(y)\gamma_{\sdimn}(z)dxdydz.
\end{flalign*}
\end{cor}
\begin{proof}
From Lemma \ref{lemma85}, there exists $v,w\in S^{\sdimn}$ such that
\begin{flalign*}
&(\adimn)^{2}\int_{\Sigma}(\langle v,N\rangle\langle w,N\rangle-m)L(\langle v,N\rangle\langle w,N\rangle-m)\gamma_{\sdimn}(x)dx\\
&\geq\int_{\Sigma}(1-\vnormt{A_{x}}^{2})\gamma_{\sdimn}(x)dx\cdot\Big(1-\frac{1}{p^{2}}\int_{\Sigma}\int_{\Sigma}\langle N(y),N(z)\rangle^{2} \gamma_{\sdimn}(y)\gamma_{\sdimn}(z)dydz\Big)\\
&\qquad-\frac{2}{p}\int_{\Sigma}\int_{\Sigma}\vnormt{A_{z}\Pi_{z}(N(y))}^{2}\big)\gamma_{\sdimn}(y)\gamma_{\sdimn}(z)dydz.
\end{flalign*}
Now, using the definition of $p$ and $\Pi_{z}$,
\begin{flalign*}
&\Big(1-\frac{1}{p^{2}}\int_{\Sigma}\int_{\Sigma}\langle N(y),N(z)\rangle^{2} \gamma_{\sdimn}(y)\gamma_{\sdimn}(z)dydz\Big)\\
&\qquad=\frac{1}{p^{2}}\int_{\Sigma}\int_{\Sigma}(1-\langle N(y),N(z)\rangle^{2})\gamma_{\sdimn}(y)\gamma_{\sdimn}(z)dydz\\
&\qquad=\frac{1}{p^{2}}\int_{\Sigma}\int_{\Sigma}\vnormt{\Pi_{z} N(y)}^{2}\gamma_{\sdimn}(y)\gamma_{\sdimn}(z)dydz.
\end{flalign*}

Recall that $\vnormt{A_{z}}_{2\to2}^{2}=\sup_{u\in S^{\sdimn}}\vnormt{A_{z}u}^{2}$.  Combining the above with $\vnormt{A_{z}\Pi_{z}(N(y))}^{2}\leq\vnorm{A_{z}}_{2\to2}^{2}\vnormt{\Pi_{z}(N(y))}^{2}$ gives
\begin{flalign*}
&(\adimn)^{2}\int_{\Sigma}(\langle v,N\rangle\langle w,N\rangle-m)L(\langle v,N\rangle\langle w,N\rangle-m)\gamma_{\sdimn}(x)dx\\
&\geq\int_{\Sigma}(1-\vnormt{A_{x}}^{2})\gamma_{\sdimn}(x)dx
\cdot\frac{1}{p^{2}}\int_{\Sigma}\int_{\Sigma}\vnormt{\Pi_{z} N(y)}^{2}\gamma_{\sdimn}(y)\gamma_{\sdimn}(z)dydz.\\
&\qquad-(\int_{\Sigma}\gamma_{\sdimn}(x)dx)
\frac{2}{p^{2}}\int_{\Sigma}\vnorm{A_{z}}_{2\to2}^{2}\int_{\Sigma}\vnormt{\Pi_{z}(N(y))}^{2}\gamma_{\sdimn}(y)\gamma_{\sdimn}(z)dydz\\
&=
\frac{1}{p^{2}}\int_{\Sigma}\int_{\Sigma}(1-\vnormt{A_{x}}^{2}-2\vnorm{A_{z}}_{2\to2}^{2})\int_{\Sigma}\vnormt{\Pi_{z}(N(y))}^{2}
\gamma_{\sdimn}(y)\gamma_{\sdimn}(x)\gamma_{\sdimn}(z)dydxdz.
\end{flalign*}
\end{proof}

We now prove the Main Theorem, Theorem \ref{mainthm}

\begin{proof}[Proof of Theorem \ref{mainthm}]
Combine Lemma \ref{varlem2} with Corollaries \ref{huisk}, \ref{cor7} and \ref{hcor} (note that if $\Omega$ is convex, then $H\geq0$).
\end{proof}

\section{Normal Variations with Dilations}\label{secdil}

\begin{lemma}[{\cite[Lemma 3.20]{colding12a}, \cite[Lemma 3.1]{cheng15}}]\label{lemma34}
Let $\Sigma\subset\R^{\adimn}$ be a $C^{\infty}$ hypersurface.  Let $\scon\in\R$.  Assume that
\begin{equation}\label{three0c}
H(x)=\langle x,N\rangle+\scon,\qquad\forall\,x\in\Sigma,
\end{equation}
then
\begin{equation}\label{three12.0}
\frac{1}{2}\Delta\vnormt{x}^{2}
=-H\langle N,x\rangle+\sdimn.
\end{equation}
\begin{equation}\label{three12}
\frac{1}{2}\mathcal{L}\vnormt{x}^{2}
=\sdimn-\vnormt{x}^{2}-\scon\langle x,N\rangle
=\sdimn-\vnormt{x}^{2}-\scon(H-\scon).
\end{equation}
\end{lemma}
\begin{proof}

Using that any hypersurface $\Sigma$ satisfies $\Delta x=-HN$,
$$
\frac{1}{2}\Delta\vnormt{x}^{2}
=\langle\Delta x,x\rangle+\vnormt{\nabla x}^{2}
=-H\langle N,x\rangle+\sdimn.
$$
Also,
\begin{flalign*}
\frac{1}{2}\mathcal{L}\vnormt{x}^{2}
&\stackrel{\eqref{three4.3}}{=}\frac{1}{2}\Delta\vnormt{x}^{2}-\frac{1}{2}\langle x,\nabla\vnormt{x}^{2}\rangle
=\frac{1}{2}\Delta\vnormt{x}^{2}-\vnormf{x^{T}}_{2}^{2}\\
&\stackrel{\eqref{three12.0}}{=}
-H\langle N,x\rangle+\sdimn-\vnormf{x^{T}}_{2}^{2}
\stackrel{\eqref{three0c}}{=}\sdimn-\langle N,x\rangle^{2}-\vnormf{x^{T}}_{2}^{2}-\scon\langle x,N\rangle\\
&= \sdimn-\vnormt{x}^{2}-\scon\langle x,N\rangle.
\stackrel{\eqref{three0c}}{=}\sdimn-\vnormt{x}^{2}-\scon(H-\scon).
\end{flalign*}
\end{proof}

\begin{lemma}[{\cite[Lemma 3.25]{colding12a} \cite[Lemma 3.3]{cheng15}}]
Let $\Sigma\subset\R^{\adimn}$ be a $C^{\infty}$ hypersurface with $\partial\Sigma=\emptyset$.  Let $\scon\in\R$.  Assume that $H(x)=\langle x,N\rangle+\scon$ $\forall$ $x\in\Sigma$, and that $\Sigma$ has polynomial volume growth.  Then
\begin{equation}\label{three15}
\int_{\Sigma}\Big(\sdimn-\vnormt{x}^{2} -\scon H+\scon^{2}\Big)\gamma_{\sdimn}(x)dx=0.
\end{equation}
%\begin{equation}\label{three16}
%\int_{\Sigma}\Big(\sdimn +\frac{c^{2}}{4}-\Big(H-\frac{c}{2}\Big)^{2}\Big)\gamma_{\sdimn}(x)dx\geq0.
%\end{equation}
%\snote{Delete the next one?}
%\begin{equation}\label{three17}
%\int_{\Sigma}H\gamma_{\sdimn}(x)dx=\scon\frac{\int_{\Sigma} \gamma_{\sdimn}(x)dx}{\int_{\Sigma}(\sdimn-\vnormt{x}^{2})\gamma_{\sdimn}(x)dx}
%\end{equation}
\begin{equation}\label{three18}
\int_{\Sigma}\Big((\sdimn+2)\vnormt{x}^{2}-\vnormt{x}^{4}-\scon\vnormt{x}^{2}(H-\scon)-2(H-\scon)^{2}\Big)=0.
\end{equation}
\begin{equation}\label{three19}
\int_{\Sigma}(\vnormt{x}^{2}-\sdimn)^{2}\gamma_{\sdimn}(x)dx=
\int_{\Sigma}\Big(2n+(H-\scon)\Big(-2H+\scon(\sdimn -\vnormt{x}^{2})\Big)\Big)\gamma_{\sdimn}(x)dx.
\end{equation}
\end{lemma}
%\snote{Use divergence theorem to compute $\int H\vnormt{x}^{2}$?}
\begin{proof}
Integrating by parts with Corollary \ref{lemma39.79} proves \eqref{three15} as follows.
$$
0=\int_{\Sigma}\Big(\mathcal{L}\vnormt{x}^{2}\Big)\gamma_{\sdimn}(x)dx
\stackrel{\eqref{three12}}{=}\int_{\Sigma}\Big(\sdimn-\vnormt{x}^{2}-\scon(H-\scon) \Big)\gamma_{\sdimn}(x)dx.
$$

Using Corollary \ref{lemma39.79} again,
\begin{flalign*}
-4\int_{\Sigma}\vnormt{x^{T}}^{2}\gamma_{\sdimn}(x)dx
&=-\int_{\Sigma}\vnormt{\nabla\vnormt{x}^{2}}^{2}\gamma_{\sdimn}(x)dx
=\int_{\Sigma}\vnormt{x}^{2}\mathcal{L}\vnormt{x}^{2}\gamma_{\sdimn}(x)dx\\
&\stackrel{\eqref{three12}}{=}2\int_{\Sigma}\Big(\sdimn\vnormt{x}^{2}-\vnormt{x}^{4}-\scon\vnormt{x}^{2}(H-\scon)\Big).
\end{flalign*}
Rearranging and using $\vnormf{x^{T}}_{2}^{2}=-\vnormf{x^{N}}_{2}^{2}+\vnormt{x}^{2}\stackrel{\eqref{three0c}}{=}-(H-\scon)^{2}+\vnormt{x}^{2}$, we get
$$\int_{\Sigma}\Big((\sdimn+2)\vnormt{x}^{2}-\vnormt{x}^{4}-\scon\vnormt{x}^{2}(H-\scon)-2(H-\scon)^{2}\Big)=0.$$
%Finally, using \eqref{three15} gives \eqref{three18}.

To prove \eqref{three19}, we write
\begin{flalign*}
&\int_{\Sigma}(\vnormt{x}^{2}-\sdimn)^{2}\gamma_{\sdimn}(x)dx
=\int_{\Sigma}\Big(\vnormt{x}^{4}-2\sdimn\vnormt{x}^{2}+n^{2}\Big)\gamma_{\sdimn}(x)dx\\
&\stackrel{\eqref{three18}}{=}
\int_{\Sigma}\Big((n+2 -2\sdimn)\vnormt{x}^{2}-\scon\vnormt{x}^{2}(H-\scon)-2(H-\scon)^{2}+n^{2}\Big)\gamma_{\sdimn}(x)dx\\
&=\int_{\Sigma}\Big((-\sdimn+2)\vnormt{x}^{2}-\scon\vnormt{x}^{2}(H-\scon)-2(H-\scon)^{2}+n^{2}\Big)\gamma_{\sdimn}(x)dx\\
&\stackrel{\eqref{three15}}{=}\int_{\Sigma}\Big((-\sdimn+2)[\sdimn-\scon(H-\scon)]-\scon\vnormt{x}^{2}(H-\scon)-2(H-\scon)^{2}+\sdimn^{2}\Big)\gamma_{\sdimn}(x)dx\\
&=\int_{\Sigma}2n+(H-\scon)\Big(\sdimn \scon-2\scon-\scon\vnormt{x}^{2}-2H+2\scon\Big)\gamma_{\sdimn}(x)dx\\
&=\int_{\Sigma}2n+(H-\scon)\Big(-2H+\scon(\sdimn -\vnormt{x}^{2})\Big)\gamma_{\sdimn}(x)dx.
\end{flalign*}
\end{proof}

Recall that if $X\colon\R^{\adimn}\to\R^{\adimn}$ is a given vector field, then we define $\Psi\colon\R^{\adimn}\times(-1,1)\to\R^{\adimn}$ so that $\Psi(x,0)=x$ and such that $\frac{d}{ds}|_{s=0}\Psi(x,s)=X(\Psi(x,s))$ $\forall$ $x\in\R^{\adimn}$ and $\forall$ $s\in(-1,1)$ as in \eqref{nine2.3}.  And for any $s\in(-1,1)$, we define $\Omega_{s}\colonequals\Psi(\Omega,s)$ and $\Sigma_{s}\colonequals\partial\Omega_{s}$.

Define $Z\colon\R^{\adimn}\to\R^{\adimn}$ so that $Z(x)\colonequals\frac{d^{2}}{ds^{2}}|_{s=0}\Psi(x,s)$ is the ``acceleration'' vector field.  Suppose we write $\Psi$ in its Taylor expansion (with respect to $s$) as
\begin{equation}\label{tayloreq}
\Psi(x,s)=x+sX(x)+\frac{s^{2}}{2}Z(x)+o(s^{2}),\qquad\forall\,x\in\R^{\adimn},s\in(-1,1).
\end{equation}
%Let $I\colon \R^{\adimn}\to\R^{\adimn}$ denote the identity map.  Also, given $A\subset\R^{\adimn}$, define sets
%$$A_{1}^{(t)}\colonequals (I+tX)(A)\in\R^{\adimn},\qquad A_{2}^{(t)}\colonequals (I+tX+(t^{2}/2)Z)(A)\in\R^{\adimn},\qquad\forall\,t\in(-1,1),$$
Note that $Z(x)\colonequals\frac{d^{2}}{ds^{2}}|_{s=0}\Psi(x,s)$ and \eqref{nine2.3} imply that
\begin{equation}\label{nine2.5}
Z(x)_{i}=\sum_{j=1}^{\adimn}X_{j}(x)\frac{\partial}{\partial x_{j}}X_{i}(x),\qquad\forall\,x\in\R^{\adimn},\,\forall\,1\leq i\leq \adimn.
\end{equation}
Let $\mathrm{Jac}$ denote the Jacobian determinant in $\R^{\adimn}$.  As shown in \cite[Lemma 12.2]{heilman15}\cite{chokski07},
\begin{equation}\label{Btwo2}
\mathrm{Jac}\Psi(x,0)=1.
\end{equation}
\begin{equation}\label{Btwo3}
(d/ds)\mathrm{Jac}\Psi|_{s=0}=\mathrm{div}(X).
\end{equation}
%\begin{equation}\label{Btwo4}
%\frac{d^{2}\Psi_{i}}{ds^{2}}|_{s=0}=\sum_{j=1}^{\adimn}X_{j}^{(i)}X^{(j)}.
%\end{equation}
\begin{equation}\label{Btwo5}
\frac{d^{2}}{ds^{2}}\mathrm{Jac}\Psi(x,s)|_{s=0}=\mathrm{div}((\mathrm{div}(X))X).
\end{equation}

For any $t_{s}>0$, define
$$F_{t_{s}}\colonequals(2\pi t_{s})^{-\frac{\sdimn}{2}}\int_{\Sigma_{s}}e^{-\frac{\vnormt{x}^{2}}{2t_{s}}}dx
=(2\pi)^{-\frac{\sdimn}{2}}\int_{\Sigma_{s}/\sqrt{t_{s}}}e^{-\frac{\vnormt{x}^{2}}{2}}dx.$$

\begin{lemma}[\embolden{First Variation of Surface area} {\cite[Lemma 3.1]{colding12a}}]\label{lemma107}
Let $\Sigma\subset\R^{\adimn}$ be a $C^{\infty}$ hypersurface.  Denote $f\colonequals\langle X,N\rangle$.  Then
$$\frac{\partial}{\partial s}F_{t_{s}}=\int_{\Sigma_{s}}\Big[f(H-\langle x,N\rangle/t_{s})+\frac{\frac{\partial}{\partial s}t_{s}}{2}
\Big(\frac{\vnormt{x}^{2}}{t_{s}^{2}}-\frac{\sdimn}{t_{s}}\Big)\Big]
(2\pi t_{s})^{-\frac{\sdimn}{2}}e^{-\frac{\vnormt{x}^{2}}{2t_{s}}}dx.$$
\end{lemma}
\begin{proof}
We use logarithmic differentiation.
$$\frac{\partial}{\partial t_{s}}\log\left((2\pi t_{s})^{-\frac{\sdimn}{2}}e^{-\frac{\vnormt{x}^{2}}{2t_{s}}}\right)=-\frac{\sdimn}{2}\frac{1}{t_{s}}+\frac{\vnormt{x}^{2}}{2t_{s}^{2}}.$$
%$$\int_{\Sigma_{s}}e^{-\frac{\vnormt{x}^{2}}{2t_{s}}}dx=\int_{\Sigma} e^{-\frac{\vnormt{\Psi(x,s)}^{2}}{2t_{s}}}\mathrm{Jac}_{\tau}\Psi(x,s)dx.$$
Taking the $s$ derivative and applying the chain rule, using $x'=fN$ and $(dx)'=fHdx$,
\begin{flalign*}
&\frac{\partial}{\partial s}\left((2\pi t_{s})^{-\frac{\sdimn}{2}}e^{-\frac{\vnormt{x}^{2}}{2t_{s}}}dx\right)\\
&=\left(fH-\frac{f}{t_{s}}\langle x,N\rangle+\Big(\frac{\partial}{\partial s}t_{s}\Big)\Big(-\frac{\sdimn}{2}\frac{1}{t_{s}}+\frac{\vnormt{x}^{2}}{2t_{s}^{2}}\Big)\right)
(2\pi t_{s})^{-\frac{\sdimn}{2}}e^{-\frac{\vnormt{x}^{2}}{2t_{s}}}dx.
\end{flalign*}
\end{proof}
\begin{remark}\label{firstvrk}
Let $\Sigma$ be a $C^{\infty}$ hypersurface.  If $\exists$ $\scon\in\R$ such that $H=\langle x,N\rangle+\scon$ $\forall$ $x\in\Sigma$, and if $h\colonequals \frac{\partial}{\partial s}|_{s=0}t_{s}$, then
\begin{flalign*}
\frac{\partial}{\partial s}|_{s=0}F_{t_{s}}
&=\int_{\Sigma}(f\scon+(h/2)(\vnormt{x}^{2}-\sdimn))\gamma_{\sdimn}dx\\
&\stackrel{\eqref{three15}}{=}\int_{\Sigma}[f\scon+(h/2)(-\scon)(H-\scon)]\gamma_{\sdimn}dx
=\scon\int_{\Sigma}[f-(h/2)\langle x,N\rangle]\gamma_{\sdimn}dx.
\end{flalign*}
%To preserve surface area to first order, must have $\int_{\Sigma}(cf+(h/2)(c^{2}-cH))\gamma_{\sdimn}(x)dx=0$/
\end{remark}

\begin{lemma}\label{lemma5}
Let $\Sigma\subset\R^{\adimn}$ be a $C^{\infty}$ hypersurface.  Then
$$f(x)\langle\nabla f(x),N(x)\rangle=\langle Z(x),N(x)\rangle,\qquad\forall\,x\in\Sigma.$$
\end{lemma}
\begin{proof}
Let $\mathrm{div}$ denote the divergence on $\R^{\adimn}$, and let $\mathrm{div}_{\tau}$ denote the divergence on $\Sigma$.  Let $DN$ denote the matrix of partial derivatives of $N$.  Using \eqref{nine2.5} and $\mathrm{div}(N(x))=\mathrm{div}_{\tau}(N(x))$ $\forall$ $x\in\Sigma$ to get
\begin{equation}
\begin{aligned}
&\langle N(x),Z(x)\rangle
=(f(x))^{2}\langle(DN(x))N(x),N(x)\rangle+f(x)\langle\nabla f(x),N(x)\rangle\\
&\quad=(f(x))^{2}(\mathrm{div}(N(x))-\mathrm{div}_{\tau}(N(x)))+f(x)\langle\nabla f(x),N(x)\rangle
=f(x)\langle\nabla f(x),N(x)\rangle.
\end{aligned}
\end{equation}
\end{proof}

\begin{lemma}[\embolden{Second Variation of Surface Area, with Dilations} {\cite[Theorem 4.1]{colding12a}}]\label{lemma40}
Let $\Sigma\subset\R^{\adimn}$ be a $C^{\infty}$ hypersurface with $\partial\Sigma=\emptyset$.  Let $t_{0}\colonequals1$, $h\colonequals\frac{\partial}{\partial s}|_{s=0}t_{s}$, $f\colonequals\langle X,N\rangle$.
\begin{flalign*}
\frac{\partial^{2}}{\partial s^{2}}|_{s=0}F_{t_{s}}
&=\int_{\Sigma}\Big(-fLf+2hf\langle x,N\rangle-h^{2}\Big(\vnormt{x}^{2}-\frac{\sdimn}{2}\Big)\\
&\qquad+\Big[f(H-\langle x,N\rangle)+\frac{h}{2}(\vnormt{x}^{2}-\sdimn)\Big]^{2}\\
&\qquad+f'(H-\langle x,N\rangle)+\frac{h'}{2}(\vnormt{x}^{2}-\sdimn)\Big)\gamma_{\sdimn}(x)dx.
\end{flalign*}
\end{lemma}
\begin{proof}
We let $'$ denote $\frac{\partial}{\partial s}|_{s=0}$.  Using Lemma \ref{lemma107} and $(ab)''=ab[\log(ab)]''+ab([\log(ab)]')^{2}$,
\begin{flalign*}
\frac{\partial^{2}}{\partial s^{2}}|_{s=0}F_{t_{s}}
&=\int_{\Sigma}\Big(f(H-\langle x,N\rangle/t_{s})'+\frac{h}{2}\Big(\frac{\vnormt{x}^{2}}{t_{s}^{2}}-\frac{\sdimn}{t_{s}}\Big)'\\
&\qquad+\Big[f(H-\langle x,N\rangle/t_{0})+\frac{h}{2}\Big(\frac{\vnormt{x}^{2}}{t_{0}^{2}}-\frac{\sdimn}{t_{0}}\Big)\Big]^{2}\\
&\qquad+f'(H-\langle x,N\rangle)+\frac{h'}{2}\Big(\frac{\vnormt{x}^{2}}{t_{0}^{2}}-\frac{\sdimn}{t_{0}}\Big)\Big)\gamma_{\sdimn}(x)dx.
\end{flalign*}
We use $H'=-\Delta f-\vnormt{A}^{2}f$, $N'=-\nabla f$, $x'=fN$ \cite[A.3, A.4]{colding12a} to get
$$\langle x,N\rangle'=f-\langle x,\nabla f\rangle.$$
Using $(t_{s}^{-1})'=-ht_{0}^{-2}$ and $t_{0}=1$,
$$
(H-\langle x,N\rangle/t_{s})'
=-\Delta f-\vnormt{A}^{2}f-\frac{(f-\langle x,\nabla f\rangle)}{t_{0}}+h\frac{\langle x,N\rangle}{t_{0}^{2}}
=\stackrel{\eqref{three4.5}}{=}-Lf+h\langle x,N\rangle.
$$
Using $(t_{s}^{-2})'=-h2t_{0}^{-3}$, $x'=fN$ and $t_{0}=1$,
$$
\left(\frac{\vnormt{x}^{2}}{t_{s}^{2}}-\frac{\sdimn}{ t_{s}}\right)'
=2\frac{\langle x,fN\rangle}{t_{0}^{2}}-2h\frac{\vnormt{x}^{2}}{t_{0}^{3}}+\frac{h\sdimn}{t_{0}^{2}}
=2f\langle x,N\rangle-2h\Big(\vnormt{x}^{2}-\frac{\sdimn}{2}\Big).
$$
Combining the above completes the proof.
\end{proof}

\begin{cor}
Let $\Sigma$ be a $C^{\infty}$ hypersurface with $\partial\Sigma=\emptyset$.  Let $\scon\in\R$.  Assume that $H(x)=\langle x,N\rangle+\scon$ for all $x\in\Sigma$.  Then
\begin{flalign*}
\frac{\partial^{2}}{\partial s^{2}}|_{s=0}F_{t_{s}}
&=\int_{\Sigma}\Big(-fLf+2hf(H-\scon)+h^{2}(H-\scon)(\scon-H/2 +\scon(\sdimn-\vnormt{x}^{2})/4)\\
&\qquad+f^{2}\scon^{2}+\scon hf(\vnormt{x}^{2}-\sdimn)
+\scon f'+\frac{h'}{2}(-\scon(H-\scon))\Big)\gamma_{\sdimn}(x)dx.
\end{flalign*}
\end{cor}
\begin{proof}
Using \eqref{three15} and \eqref{three19} in Lemma \ref{lemma40},
\begin{flalign*}
\frac{\partial^{2}}{\partial s^{2}}|_{s=0}F_{t_{s}}
&=\int_{\Sigma}\Big(-fLf+2hf\langle x,N\rangle-h^{2}(-\scon(H-\scon)+n/2)+f^{2}\scon^{2}+\scon hf(\vnormt{x}^{2}-\sdimn)\\
&\qquad+\frac{h^{2}}{4}[2n+(H-\scon)(-2H+\scon(n-\vnormt{x}^{2}))]+\scon f'+\frac{h'}{2}(-\scon(H-\scon))\Big)\gamma_{\sdimn}(x)dx\\
&=\int_{\Sigma}\Big(-fLf+2hf\langle x,N\rangle+h^{2}(H-\scon)(\scon-H/2 +\scon(\sdimn-\vnormt{x}^{2})/4)\\
&\qquad+f^{2}\scon^{2}+\scon hf(\vnormt{x}^{2}-\sdimn)
+\scon f'+\frac{h'}{2}(-\scon(H-\scon))\Big)\gamma_{\sdimn}(x)dx.
\end{flalign*}
\end{proof}

For any $t_{s}>0$, define
$$G_{t_{s}}\colonequals(2\pi t_{s})^{-\frac{\adimn}{2}}\int_{\Omega_{s}}e^{-\frac{\vnormt{x}^{2}}{2t_{s}}}dx
=(2\pi)^{-\frac{\adimn}{2}}\int_{\Omega_{s}/\sqrt{t_{s}}}e^{-\frac{\vnormt{x}^{2}}{2}}dx.$$
Repeating much of the reasoning of Lemma \ref{lemma107} (see e.g. \cite[Eq. (20)]{barchiesi16}), we get

\begin{lemma}[\embolden{First Variation of Volume}]\label{lemma41}
Let $\Sigma$ be a $C^{\infty}$ hypersurface.  Let $\mathrm{Jac}$ denote the Jacobian determinant in $\R^{\adimn}$.  Let $t_{0}\colonequals1$, $h\colonequals\frac{\partial}{\partial s}|_{s=0}t_{s}$, $f\colonequals\langle X,N\rangle$.
\begin{flalign*}
&\frac{\partial}{\partial s}G_{t_{s}}=(2\pi t_{s})^{-\frac{\adimn}{2}}\int_{\Omega}e^{-\frac{\vnormt{\Psi(x,s)}^{2}}{2t_{s}}}\mathrm{Jac}\Psi(x,s)\\
&\qquad\qquad\Big[\frac{h}{2}\Big(\frac{\vnormt{x}^{2}}{t_{s}^{2}}-\frac{\adimn}{t_{s}}\Big)+\frac{[\mathrm{Jac}\Psi(x,s)]'}{\mathrm{Jac}\Psi(x,s)}-\frac{\langle \Psi(x,s),\Psi'(x,s)\rangle}{t_{s}}\Big]dx.
\end{flalign*}
\begin{flalign*}
\frac{\partial}{\partial s}|_{s=0}G_{t_{s}}
&=(2\pi t_{s})^{-\frac{\adimn}{2}}\int_{\Omega}e^{-\frac{\vnormt{x}^{2}}{2}}
\Big[\frac{h}{2}\Big(\frac{\vnormt{x}^{2}}{t_{s}^{2}}-\frac{\adimn}{t_{s}}\Big)+\mathrm{div}(X)-\langle x,X\rangle\Big]dx\\
&=\int_{\Omega}\mathrm{div}\Big(\Big(-\frac{xh}{2}+X\Big)\gamma_{\adimn}(x)\Big)
=\int_{\Sigma}\Big(-\frac{h}{2}\langle x,N\rangle+\langle X,N\rangle\Big)\gamma_{\adimn}(x)dx\\
&=\int_{\Sigma}\Big(-\frac{h}{2}\langle x,N\rangle+f\Big)\gamma_{\adimn}(x)dx.
\end{flalign*}
\end{lemma}

\begin{lemma}[\embolden{Second Variation of Volume}]\label{lemma42}
Let $\Sigma\subset\R^{\adimn}$ be a $C^{\infty}$ hypersurface with $\partial\Sigma=\emptyset$.  Let $h\colonequals\frac{\partial}{\partial s}|_{s=0}t_{s}$, $f\colonequals\langle X,N\rangle$
\begin{flalign*}
\frac{\partial^{2}}{\partial s^{2}}|_{s=0}G_{t_{s}}
&=
\int_{\Sigma}\Big(f(\nabla_{N}f+fH-f\langle x,N\rangle)+\frac{1}{2}(h^{2}+h')\langle x,N\rangle\\
&\qquad\qquad+h(f-(h/4)\langle x,N\rangle)(\vnormt{x}^{2}-[\adimn])\Big)\gamma_{\adimn}(x)dx.
\end{flalign*}
\end{lemma}
\begin{proof}
For any $1\leq i\leq\adimn$, let $Z_{i}\colonequals\sum_{j=1}^{\adimn}X_{j}\frac{\partial}{\partial x_{j}}X_{i}$.  We let $'$ denote $\frac{\partial}{\partial s}|_{s=0}$.  Using $(ab)''=ab[\log(ab)]''+ab([\log(ab)]')^{2}$ in Lemma \ref{lemma41}, with $\Psi(x,0)=x$ \eqref{tayloreq}, \eqref{Btwo2} and \eqref{Btwo3},
\begin{flalign*}
&\frac{\partial^{2}}{\partial s^{2}}|_{s=0}G_{t_{s}}\\
&=\int_{\Omega}\Big(\frac{h}{2}\Big(\frac{\vnormt{\Psi(x,s)}^{2}}{t_{s}^{2}}-\frac{\adimn}{t_{s}}\Big)'+[\mathrm{Jac}\Psi(x,s)]''-\langle x,\Psi''(x,s)\rangle-\langle\Psi'(x,s),\Psi'(x,s)\rangle\\
&\qquad+\Big[\frac{h}{2}\Big(\frac{\vnormt{x}^{2}}{t_{0}^{2}}-\frac{\adimn}{t_{0}}\Big)+\mathrm{div}(X)-\langle x,X\rangle\Big]^{2}-[(\mathrm{Jac}\Psi(x,s))']^{2}+\langle x,X\rangle h\\
&\qquad+\frac{h'}{2}(\adimn-\vnormt{x}^{2})\Big)\gamma_{\adimn}(x)dx.
\end{flalign*}
Using \eqref{tayloreq}, $(t_{s}^{-1})'=-ht_{0}^{-2}$, $(t_{s}^{-2})'=-h2t_{0}^{-3}$ and $t_{0}=1$,
%\begin{flalign*}
$$
\left(\frac{\vnormt{\Psi(x,s)}^{2}}{t_{s}^{2}}-\frac{\adimn}{ t_{s}}\right)'
=2\frac{\langle x,X\rangle}{t_{0}^{2}}-2h\frac{\vnormt{x}^{2}}{t_{0}^{3}}+\frac{h(\adimn)}{t_{0}^{2}}
=2\langle x,X\rangle-2h\Big(\vnormt{x}^{2}-\frac{\adimn}{2}\Big).
$$
%\end{flalign*}
Combining the above with \eqref{tayloreq}, \eqref{Btwo3}, \eqref{Btwo5} and $t_{0}=1$,
\begin{flalign*}
\frac{\partial^{2}}{\partial s^{2}}|_{s=0}G_{t_{s}}
&=\int_{\Omega}\Big(\langle x,X\rangle h-h^{2}\Big(\vnormt{x}^{2}-\frac{\adimn}{2}\Big)+\mathrm{div}(X\mathrm{div}(X))-\langle Z,x\rangle-\langle X,X\rangle\\
&\qquad+\Big[\frac{h}{2}\Big(\frac{\vnormt{x}^{2}}{t_{0}^{2}}-\frac{\adimn}{t_{0}}\Big)+\mathrm{div}(X)-\langle x,X\rangle\Big]^{2}-(\mathrm{div}(X))^{2}+\langle x,X\rangle h\\
&\qquad+\frac{h'}{2}(\adimn-\vnormt{x}^{2})\Big)\gamma_{\adimn}(x)dx\\
&=\int_{\Omega}\Big(2\langle x,X\rangle h-h^{2}\Big(\vnormt{x}^{2}-\frac{\adimn}{2}\Big)+\frac{h^{2}}{4}(\vnormt{x}^{2}-[\adimn])^{2}\\
&\qquad+h[\mathrm{div}(X)-\langle x,X\rangle](\vnormt{x}^{2}-[\adimn])+\frac{h'}{2}(\adimn-\vnormt{x}^{2})\\
&\qquad+\mathrm{div}(X\mathrm{div}(X))-\langle Z,x\rangle-\langle X,X\rangle-2\langle x,X\rangle\mathrm{div}(X)+\langle x,X\rangle^{2}
\Big)\gamma_{\adimn}(x)dx.
\end{flalign*}

% <X,grad g>+g div(X) = ?  div(gX)

Then using Lemma \ref{lemma20b}, and also using

\begin{flalign*}
&\mathrm{div}\Big(X(\vnormt{x}^{2}-[\adimn])\gamma_{\adimn}(x)\Big)
=(\vnormt{x}^{2}-[\adimn])\mathrm{div}(X\gamma_{\adimn}(x))+2\langle x,X\rangle\gamma_{\adimn}(x)\\
&\qquad\qquad\qquad\qquad\qquad=(\vnormt{x}^{2}-[\adimn])[\mathrm{div}(X)-\langle x,X\rangle]\gamma_{\adimn}(x)+2\langle x,X\rangle\gamma_{\adimn}(x).
\end{flalign*}

\begin{flalign*}
\mathrm{div}\Big(x(\vnormt{x}^{2}-[\adimn])\gamma_{\adimn}(x)\Big)
&=(\vnormt{x}^{2}-[\adimn])\mathrm{div}(x\gamma_{\adimn}(x))+2\vnormt{x}^{2}\gamma_{\adimn}(x)\\
&=-(\vnormt{x}^{2}-[\adimn])^{2}\gamma_{\adimn}(x)+2\vnormt{x}^{2}\gamma_{\adimn}(x).
\end{flalign*}

We get

\begin{flalign*}
&\frac{\partial^{2}}{\partial s^{2}}|_{s=0}G_{t_{s}}\\
&=\int_{\Omega}\mathrm{div}(X\mathrm{div}(X\gamma_{\adimn}(x)))+\Big(2\langle x,X\rangle h-h^{2}\Big(\vnormt{x}^{2}-\frac{\adimn}{2}\Big)+\frac{h^{2}}{4}(\vnormt{x}^{2}-[\adimn])^{2}\\
&\qquad+h[\mathrm{div}(X)-\langle x,X\rangle](\vnormt{x}^{2}-[\adimn])+\frac{h'}{2}(\adimn-\vnormt{x}^{2})\Big)\gamma_{\adimn}(x)dx\\
&=\int_{\Omega}\mathrm{div}(X\mathrm{div}(X\gamma_{\adimn}(x)))+\Big(h^{2}[\mathrm{div}(x\gamma_{\adimn}(x))-[\adimn]/2]\\
&\qquad+\frac{h^{2}}{4}[-\mathrm{div}(x(\vnormt{x}^{2}-[\adimn])\gamma_{\adimn}(x))]+\frac{h^{2}}{2}\vnormt{x}^{2}\gamma_{\adimn}(x)\\
&\qquad+h[\mathrm{div}(X(\vnormt{x}^{2}-[\adimn])\gamma_{\adimn}(x))]
+\frac{h'}{2}\mathrm{div}(x\gamma_{\adimn}(x))\Big)dx.
\end{flalign*}

Applying the divergence theorem,

\begin{flalign*}
&\frac{\partial^{2}}{\partial s^{2}}|_{s=0}G_{t_{s}}=\int_{\Sigma}f(\mathrm{div}(X\gamma_{\adimn}(x)))\\
&+\Big(h^{2}\langle x,N\rangle-(h^{2}/4)\langle x,N\rangle(\vnormt{x}^{2}-[\adimn])+hf(\vnormt{x}^{2}-[\adimn])+(h'/2)\langle x,N\rangle\Big)\gamma_{\adimn}(x)dx\\
&+\int_{\Omega}\Big(-(h^{2}/2)(\adimn)+(h^{2}/2)\vnormt{x}^{2}\Big)\gamma_{\adimn}(x)dx\\
&=\int_{\Sigma}f(\mathrm{div}(X\gamma_{\adimn}(x)))
+\Big(h^{2}\langle x,N\rangle-(h^{2}/4)\langle x,N\rangle(\vnormt{x}^{2}-[\adimn])\\
&\qquad\qquad+hf(\vnormt{x}^{2}-[\adimn])+(h'/2)\langle x,N\rangle\Big)\gamma_{\adimn}(x)dx-\int_{\Omega}\frac{h^{2}}{2}\mathrm{div}(x\gamma_{\adimn}(x))dx\\
&=\int_{\Sigma}f(fH+\nabla_{N}f - f\langle x,N\rangle)+(h^{2}/2)\langle x,N\rangle\\
&\quad+\Big(-(h^{2}/4)\langle x,N\rangle(\vnormt{x}^{2}-[\adimn])+hf(\vnormt{x}^{2}-[\adimn])+(h'/2)\langle x,N\rangle\Big)\gamma_{\adimn}(x)dx.
\end{flalign*}

Above we used
$$\mathrm{div}(X\gamma_{\adimn}(x))=\mathrm{div}(fN\gamma_{\adimn}(x))=(\nabla_{N}f+fH-f\langle x,N\rangle)\gamma_{\adimn}(x).$$

\end{proof}

For technical reasons, we restrict the following Lemma to $f>0$.

\begin{lemma}[\embolden{Volume Preserving Second Variation of the Surface area}]\label{lemma100}
Let $\Sigma\subset\R^{\adimn}$ be a $\C^{\infty}$ hypersurface.  Let $\scon\in\R$.  Suppose $H(x)=\langle x,N\rangle+\scon$ for all $x\in\Sigma$.  Given any $f\colon\Sigma\to\R$ with $f>0$, $\exists$ $h,h',f'$ so that $\frac{\partial}{\partial s}|_{s=0}G_{t_{s}}=\frac{\partial^{2}}{\partial s^{2}}|_{s=0}G_{t_{s}}=\frac{\partial}{\partial s}|_{s=0}F_{t_{s}}=0$, and
$$
\frac{\partial^{2}}{\partial s^{2}}|_{s=0}F_{t_{s}}
=\int_{\Sigma}\Big(-fLf+2hf\langle x,N\rangle
-\frac{h^{2}}{2}\langle x,N\rangle^{2} +\scon\frac{h^{2}}{4}\langle x,N\rangle\Big)\gamma_{\sdimn}(x)dx.
$$
\end{lemma}
\begin{proof}
As above, we let $'$ denote $\frac{\partial}{\partial s}|_{s=0}$.  Using $N'=-\nabla f$ \cite[A.3]{colding12a}, and that $X$ is parallel to $N$ on $\Sigma$, and Lemma \ref{lemma5},
$$f'=\langle X,N\rangle'=\langle Z,N\rangle=f\nabla_{N}f.$$
Using this fact, we let $h'\colonequals0$, we choose $h$ so that $\frac{\partial}{\partial s}|_{s=0}G_{t_{s}}=0$ in Lemma \ref{lemma41}.  That is, we choose $h$ so that
\begin{equation}\label{three50}
\int_{\Sigma}f\gamma_{\sdimn}(x)dx=\frac{h}{2}\int_{\Sigma}\langle x,N\rangle\gamma_{\sdimn}(x)dx.
\end{equation}
By Remark \ref{firstvrk}, $\frac{\partial}{\partial s}|_{s=0}F_{t_{s}}=0$ as well.  We then choose $f\nabla_{N}f$ in Lemma \ref{lemma42} so that $\frac{\partial^{2}}{\partial s^{2}}G_{t_{s}}=0$.  That is, we choose $\nabla_{N}f$ so that
$$
f\nabla_{N}f
\colonequals
-f^{2}H+f^{2}\langle x,N\rangle-(h^{2}/2)\langle x,N\rangle
-h(f-(h/4)\langle x,N\rangle)(\vnormt{x}^{2}-[\adimn]).
$$

We then substitute these choices of $h,h',f'$ into Lemma \ref{lemma40}.

\begin{flalign*}
&\frac{\partial^{2}}{\partial s^{2}}|_{s=0}F_{t_{s}}\\
&=\int_{\Sigma}\Big(-fLf+2hf\langle x,N\rangle-h^{2}\Big(\vnormt{x}^{2}-\frac{\sdimn}{2}\Big)
+\Big[f\scon+\frac{h}{2}(\vnormt{x}^{2}-\sdimn)\Big]^{2}
+f'\scon\Big)\gamma_{\sdimn}(x)dx\\
&=\int_{\Sigma}\Big(-fLf+2hf\langle x,N\rangle-h^{2}\Big(\vnormt{x}^{2}-\frac{\sdimn}{2}\Big)
+\Big[f\scon+\frac{h}{2}(\vnormt{x}^{2}-\sdimn)\Big]^{2}\\
&\quad+\scon\Big[-f^{2}H+f^{2}\langle x,N\rangle-(h^{2}/2)\langle x,N\rangle
-h(f-(h/4)\langle x,N\rangle)(\vnormt{x}^{2}-[\adimn])\Big]\Big)\gamma_{\sdimn}(x)dx.
\end{flalign*}

Applying \eqref{three15} and \eqref{three19}, along with repeated use of $H=\langle x,N\rangle+\scon$ gives
\begin{flalign*}
&\frac{\partial^{2}}{\partial s^{2}}|_{s=0}F_{t_{s}}
=\int_{\Sigma}\Big(-fLf+2hf\langle x,N\rangle-h^{2}\Big(\vnormt{x}^{2}-\sdimn+(\sdimn/2)\Big)\\
&\qquad\qquad\qquad
+f^{2}\scon^{2}+\scon fh(\vnormt{x}^{2}-\sdimn)+\frac{h^{2}}{4}(\vnormt{x}^{2}-\sdimn)^{2}\\
&\quad+\scon\Big[-f^{2}H+f^{2}\langle x,N\rangle-(h^{2}/2)\langle x,N\rangle
-h(f-(h/4)\langle x,N\rangle)(\vnormt{x}^{2}-[\adimn])\Big]\Big)\gamma_{\sdimn}(x)dx\\
&=\int_{\Sigma}\Big(-fLf+2hf\langle x,N\rangle-h^{2}\scon(\scon-H)-h^{2}(\sdimn/2)\Big)
+f^{2}\scon^{2}+\scon fh(\vnormt{x}^{2}-\sdimn)\\
&\quad+\frac{h^{2}}{4}\Big(2\sdimn+\langle x,N\rangle(-2H+\scon(\sdimn-\vnormt{x}^{2}))\Big)\\
&\quad+\scon\Big[-f^{2}H+f^{2}\langle x,N\rangle-(h^{2}/2)\langle x,N\rangle
-h(f-(h/4)\langle x,N\rangle)(\vnormt{x}^{2}-[\adimn])\Big]\Big)\gamma_{\sdimn}(x)dx\\
&=\int_{\Sigma}\Big(-fLf+2hf\langle x,N\rangle+h^{2}\scon\langle x,N\rangle-h^{2}(\sdimn/2)\Big)
+f^{2}\scon^{2}+\scon fh(\vnormt{x}^{2}-\sdimn)\\
&\quad+\frac{h^{2}}{4}\Big(2\sdimn+\langle x,N\rangle(-2H+\scon(\sdimn-\vnormt{x}^{2}))\Big)\\
&\quad+\scon\Big[-f^{2}H+f^{2}\langle x,N\rangle-(h^{2}/2)\langle x,N\rangle
-h(f-(h/4)\langle x,N\rangle)(\vnormt{x}^{2}-[\adimn])\Big]\Big)\gamma_{\sdimn}(x)dx\\
\end{flalign*}
\begin{flalign*}
&=\int_{\Sigma}\Big(-fLf+2hf\langle x,N\rangle+h^{2}\scon\langle x,N\rangle-h^{2}(\sdimn/2)\Big)
+f^{2}\scon^{2}+\scon fh\\
&\qquad+\frac{h^{2}}{4}\Big(2\sdimn+\langle x,N\rangle(-2H+\scon(\sdimn-\vnormt{x}^{2}))\Big)\\
&\qquad+\scon\Big[-f^{2}H+f^{2}\langle x,N\rangle-(h^{2}/2)\langle x,N\rangle
+(h^{2}/4)\langle x,N\rangle(\vnormt{x}^{2}-[\adimn])\Big]\Big)\gamma_{\sdimn}(x)dx\\
&=\int_{\Sigma}\Big(-fLf+2hf\langle x,N\rangle-h^{2}(\sdimn/2)\Big)
+f^{2}\scon^{2}+\scon fh\\
&\qquad+\frac{h^{2}}{4}\Big(2\sdimn+\langle x,N\rangle(-2H+\scon(\sdimn-\vnormt{x}^{2}))\Big)\\
&\qquad+\scon\Big[-f^{2}H+f^{2}\langle x,N\rangle
+(h^{2}/4)\langle x,N\rangle(\vnormt{x}^{2}-[\sdimn-1])\Big]\Big)\gamma_{\sdimn}(x)dx\\
&=\int_{\Sigma}\Big(-fLf+2hf\langle x,N\rangle-h^{2}(\sdimn/2)\Big)
+f^{2}\scon^{2}+\scon fh\\
&\qquad+\frac{h^{2}}{4}\Big(2\sdimn+\langle x,N\rangle(-2H)\Big)
+\scon\Big[-f^{2}H+f^{2}\langle x,N\rangle
+(h^{2}/4)\langle x,N\rangle\Big]\Big)\gamma_{\sdimn}(x)dx\\
&=\int_{\Sigma}\Big(-fLf+2hf\langle x,N\rangle\Big)
+f^{2}\scon^{2}+\scon fh\\
&\qquad-\frac{h^{2}}{2}\langle x,N\rangle H+\scon\Big[-f^{2}H+f^{2}\langle x,N\rangle
+(h^{2}/4)\langle x,N\rangle\Big]\Big)\gamma_{\sdimn}(x)dx\\
&=\int_{\Sigma}\Big(-fLf+2hf\langle x,N\rangle\Big)+\scon fh
-\frac{h^{2}}{2}\langle x,N\rangle H+\scon(h^{2}/4)\langle x,N\rangle\Big)\gamma_{\sdimn}(x)dx.
\end{flalign*}

Applying $H=\langle x,N\rangle+\scon$ one more time
$$
\frac{\partial^{2}}{\partial s^{2}}|_{s=0}F_{t_{s}}
=\int_{\Sigma}\Big(-fLf+2hf\langle x,N\rangle+\scon fh
-\frac{h^{2}}{2}\langle x,N\rangle^{2}-\scon\frac{h^{2}}{4}\langle x,N\rangle\Big)\gamma_{\sdimn}(x)dx.
$$
Finally, by \eqref{three50}, we then get % ch\int f = ch(h/2) \int <x,N>
$$
\frac{\partial^{2}}{\partial s^{2}}|_{s=0}F_{t_{s}}
=\int_{\Sigma}\Big(-fLf+2hf\langle x,N\rangle
-\frac{h^{2}}{2}\langle x,N\rangle^{2} +\scon\frac{h^{2}}{4}\langle x,N\rangle\Big)\gamma_{\sdimn}(x)dx.
$$

\end{proof}

%to get a better idea, compute these quantities for sphere

Using Lemma \ref{lemma100} for $H-\scon+t$, and then differentiating the resulting expression twice at $t=0$, we get the following interesting inequality.

\begin{lemma}\label{lemma101}
Let $\Sigma\subset\R^{\adimn}$ be a $\C^{\infty}$ hypersurface.  Let $\scon\in\R$.  Suppose $H(x)=\langle x,N\rangle+\scon$ for all $x\in\Sigma$.  Let $t\in\R$.  Let $f\colonequals H-\scon+t$.  Define $h\colonequals2\frac{\int_{\Sigma}f\gamma_{\adimn}(x)dx}{\int_{\Sigma}\langle x,N\rangle\gamma_{\adimn}(x)dx}$ as in \eqref{three50}.  Then
\begin{flalign*}
&\int_{\Sigma}\Big(-fLf+2hf\langle x,N\rangle
-\frac{h^{2}}{2}\langle x,N\rangle^{2} +\scon\frac{h^{2}}{4}\langle x,N\rangle\Big)\gamma_{\adimn}(x)dx\\
&\qquad\qquad\leq t^{2}\int_{\Sigma}\Big(-\vnormt{A}^{2}
+H\frac{\int_{\Sigma}\gamma_{\adimn}(x)dx}{\int_{\Sigma}\langle x,N\rangle\gamma_{\adimn}(x)dx}
\Big)\gamma_{\adimn}(x)dx.
\end{flalign*}
\end{lemma}
\begin{proof}
From \eqref{three9}, and \eqref{three4.5},
$$L(H-\scon)=2H+\scon\vnormt{A}^{2}-\scon\vnormt{A}^{2}-\scon=2H-\scon.$$
Also, by the definition of $H$ and $f$, and using $H-\scon=\langle x,N\rangle$,
$$h\colonequals2\frac{\int_{\Sigma}f\gamma_{\adimn}(x)dx}{\int_{\Sigma}\langle x,N\rangle\gamma_{\adimn}(x)dx}
=2\frac{\int_{\Sigma}(\langle x,N\rangle+t)\gamma_{\adimn}(x)dx}{\int_{\Sigma}\langle x,N\rangle\gamma_{\adimn}(x)dx}
=2+2\frac{\int_{\Sigma}t\gamma_{\adimn}(x)dx}{\int_{\Sigma}\langle x,N\rangle\gamma_{\adimn}(x)dx}.$$
For brevity, we define
$$s\colonequals2\frac{\int_{\Sigma}t\gamma_{\adimn}(x)dx}{\int_{\Sigma}\langle x,N\rangle\gamma_{\adimn}(x)dx}.$$
Then $h=2+s$, and
\begin{flalign*}
&\int_{\Sigma}\Big(-fLf+2hf\langle x,N\rangle
-\frac{h^{2}}{2}\langle x,N\rangle^{2} +\scon\frac{h^{2}}{4}\langle x,N\rangle\Big)\gamma_{\adimn}(x)dx\\
&=\int_{\Sigma}\Big(-(H-\scon+t)(2H-\scon+t\vnormt{A}^{2}+t)+2(2+s)(\langle x,N\rangle +t)\langle x,N\rangle\\
&\qquad\qquad-\frac{(2+s)^{2}}{2}\langle x,N\rangle^{2} +\scon\frac{(2+s)^{2}}{4}\langle x,N\rangle\Big)\gamma_{\adimn}(x)dx\\
&=\int_{\Sigma}\Big(-(H-\scon+t)^{2}-(\langle x,N\rangle+t)(H+t\vnormt{A}^{2})+(4+2s)\langle x,N\rangle^{2}+(4+2s)t\langle x,N\rangle\\
&\qquad\qquad-\frac{(2+s)^{2}}{2}\langle x,N\rangle^{2} +\scon\frac{(2+s)^{2}}{4}\langle x,N\rangle\Big)\gamma_{\adimn}(x)dx\\
&=\int_{\Sigma}\Big(-(\langle x,N\rangle+t)^{2}-(\langle x,N\rangle+t)(H+t\vnormt{A}^{2})+(4+2s)\langle x,N\rangle^{2}+(4+2s)t\langle x,N\rangle\\
&\qquad\qquad-\frac{(2+s)^{2}}{2}\langle x,N\rangle^{2} +(H-\langle x,N\rangle)\frac{(2+s)^{2}}{4}\langle x,N\rangle\Big)\gamma_{\adimn}(x)dx\\
&=\int_{\Sigma}\Big(-\langle x,N\rangle^{2}-2t\langle x,N\rangle-t^{2}-H\langle x,N\rangle-tH-(\langle x,N\rangle+t)t\vnormt{A}^{2}\\
&\qquad+4\langle x,N\rangle^{2}+2s\langle x,N\rangle^{2}+(4+2s)t\langle x,N\rangle-2\langle x,N\rangle^{2}-2s\langle x,N\rangle^{2}-\frac{s^{2}}{2}\langle x,N\rangle^{2}\\
&\qquad-\langle x,N\rangle^{2}+H\langle x,N\rangle+s(H-\langle x,N\rangle)\langle x,N\rangle+\frac{s^{2}}{4}(H-\langle x,N\rangle)\langle x,N\rangle\Big)\gamma_{\adimn}(x)dx\\
&=\int_{\Sigma}\Big(-2t\langle x,N\rangle-t^{2}-H\langle x,N\rangle-tH-(\langle x,N\rangle+t)t\vnormt{A}^{2}\\
&\qquad+2s\langle x,N\rangle^{2}+(4+2s)t\langle x,N\rangle-2s\langle x,N\rangle^{2}-\frac{s^{2}}{2}\langle x,N\rangle^{2}\\
&\qquad+H\langle x,N\rangle+s(H-\langle x,N\rangle)\langle x,N\rangle+\frac{s^{2}}{4}(H-\langle x,N\rangle)\langle x,N\rangle.\\
&=\int_{\Sigma}\Big(-2t\langle x,N\rangle-t^{2}-tH-(\langle x,N\rangle+t)t\vnormt{A}^{2}\\
&\qquad+(4+2s)t\langle x,N\rangle-\frac{s^{2}}{2}\langle x,N\rangle^{2}
+s(H-\langle x,N\rangle)\langle x,N\rangle+\frac{s^{2}}{4}(H-\langle x,N\rangle)\langle x,N\rangle.
\end{flalign*}
Let $F(t)$ be the above expression, as a function of $t$.  Then
\begin{flalign*}
&F'(0)
=\int_{\Sigma}\Big(-2\langle x,N\rangle-H-\langle x,N\rangle\vnormt{A}^{2}+4\langle x,N\rangle\\
&\qquad\qquad\qquad
+2\frac{\int_{\Sigma}\gamma_{\adimn}(x)dx}{\int_{\Sigma}\langle x,N\rangle\gamma_{\adimn}(x)dx}(H-\langle x,N\rangle)\langle x,N\rangle\Big)\gamma_{\adimn}(x)dx\\
&=\int_{\Sigma}\Big(2\langle x,N\rangle-H-(H-\scon)\vnormt{A}^{2}
+2\frac{\int_{\Sigma}\gamma_{\adimn}(x)dx}{\int_{\Sigma}\langle x,N\rangle\gamma_{\adimn}(x)dx}(H-\langle x,N\rangle)\langle x,N\rangle\Big)\gamma_{\adimn}(x)dx.
\end{flalign*}
Integrating by parts to get $\int_{\Sigma}\mathcal{L}H\gamma_{\sdimn}(x)dx=0$ by Corollary \ref{lemma39.79},
$$
\int_{\Sigma}2H\gamma_{n}(x)dx
\stackrel{\eqref{three9}}{=}\int_{\Sigma} (LH -\scon\vnormt{A}^{2})\gamma_{n}(x)dx
=\int_{\Sigma} (H(\vnormt{A}^{2}+1) -\scon\vnormt{A}^{2})\gamma_{n}(x)dx.
$$
So,
\begin{equation}\label{haeq}
\int_{\Sigma}H\vnormt{A}^{2}\gamma_{n}(x)dx=\int_{\Sigma}H\gamma_{n}(x)dx+\scon\int_{\Sigma}\vnormt{A}^{2}\gamma_{n}(x)dx.
\end{equation}
From \eqref{haeq}, $\int_{\Sigma}((H-\scon)\vnormt{A}^{2}-H)\gamma_{\adimn}(x)dx=0$.  So,
\begin{flalign*}
F'(0)
&=\int_{\Sigma}\Big(2\langle x,N\rangle-2H
+2\frac{\int_{\Sigma}\gamma_{\adimn}(x)dx}{\int_{\Sigma}\langle x,N\rangle\gamma_{\adimn}(x)dx}(H-\langle x,N\rangle)\langle x,N\rangle\Big)\gamma_{\adimn}(x)dx\\
&=\int_{\Sigma}\Big(-2\scon
+2\frac{\int_{\Sigma}\gamma_{\adimn}(x)dx}{\int_{\Sigma}\langle x,N\rangle\gamma_{\adimn}(x)dx}\scon\langle x,N\rangle\Big)\gamma_{\adimn}(x)dx
=0.
\end{flalign*}  %  2(<x,N>-H)=-\scon
\begin{flalign*}
&F''(0)
=\int_{\Sigma}\Big(-2-2\vnormt{A}^{2}
+8\frac{\int_{\Sigma}\gamma_{\adimn}(x)dx}{\int_{\Sigma}\langle x,N\rangle\gamma_{\adimn}(x)dx}\langle x,N\rangle
-4\frac{\left[\int_{\Sigma}\gamma_{\adimn}(x)dx\right]^{2}}{\left[\int_{\Sigma}\langle x,N\rangle\gamma_{\adimn}(x)dx\right]^{2}}\langle x,N\rangle^{2}\\
&\qquad
+2\frac{\left[\int_{\Sigma}\gamma_{\adimn}(x)dx\right]^{2}}{\left[\int_{\Sigma}\langle x,N\rangle\gamma_{\adimn}(x)dx\right]^{2}}(H-\langle x,N\rangle)\langle x,N\rangle
\Big)\gamma_{\adimn}(x)dx\\
&=\int_{\Sigma}\Big(6-2\vnormt{A}^{2}
+2\scon\frac{\int_{\Sigma}\gamma_{\adimn}(x)dx}{\int_{\Sigma}\langle x,N\rangle\gamma_{\adimn}(x)dx}
-4\frac{\left[\int_{\Sigma}\gamma_{\adimn}(x)dx\right]^{2}}{\left[\int_{\Sigma}\langle x,N\rangle\gamma_{\adimn}(x)dx\right]^{2}}\langle x,N\rangle^{2}\Big)\gamma_{\adimn}(x)dx.
\end{flalign*}
Then using the Cauchy-Schwarz inequality,
\begin{flalign*}
F''(0)
&\leq\int_{\Sigma}\Big(6-2\vnormt{A}^{2}
+2\scon\frac{\int_{\Sigma}\gamma_{\adimn}(x)dx}{\int_{\Sigma}\langle x,N\rangle\gamma_{\adimn}(x)dx}
-4\frac{\int_{\Sigma}\gamma_{\adimn}(x)dx}{\int_{\Sigma}\langle x,N\rangle^{2}\gamma_{\adimn}(x)dx}\langle x,N\rangle^{2}\Big)\gamma_{\adimn}(x)dx\\
&=2\int_{\Sigma}\Big(1-\vnormt{A}^{2}
+\scon\frac{\int_{\Sigma}\gamma_{\adimn}(x)dx}{\int_{\Sigma}\langle x,N\rangle\gamma_{\adimn}(x)dx}
\Big)\gamma_{\adimn}(x)dx\\
&=2\int_{\Sigma}\Big(-\vnormt{A}^{2}
+H\frac{\int_{\Sigma}\gamma_{\adimn}(x)dx}{\int_{\Sigma}\langle x,N\rangle\gamma_{\adimn}(x)dx}
\Big)\gamma_{\adimn}(x)dx.
\end{flalign*}

\end{proof}

\begin{proof}[Proof of Theorem \ref{lastthm}]
Combine Lemmas \ref{lemma100} and \ref{lemma101}.
\end{proof}

\begin{lemma}\label{lemma20b}
Let $X\in C_{0}^{\infty}(\R^{\adimn},\R^{\adimn})$.  Then for any $x\in\R^{\adimn}$,
\begin{flalign*}
\mathrm{div}\Big(X(x)\mathrm{div}(X(x)\gamma_{\adimn}(x))\Big)
&=\Big(\mathrm{div}(X(x)\mathrm{div}X(x))-\langle X(x),x\rangle\mathrm{div}(X(x))-\langle X(x),X(x)\rangle\\
&\qquad\qquad-\langle x,Z(x)\rangle-\langle x,X(x)\rangle\mathrm{div}(X)+\langle X(x),x\rangle^{2}\Big)\gamma_{\adimn}(x).
\end{flalign*}
\end{lemma}
\begin{proof}
Below, we let $\nabla$ denote the gradient on $\R^{\adimn}$.
\begin{flalign*}
&\mathrm{div}\Big(X(x)\mathrm{div}(X(x)\gamma_{\adimn}(x))\Big)\\
&=\mathrm{div}\Big(X(x)\mathrm{div}(X(x))\gamma_{\adimn}(x)-X(x)\langle X(x),x\rangle\gamma_{\adimn}(x))\Big)\\
&=\Big(\mathrm{div}(X(x)\mathrm{div}(X(x)))-\langle X(x),x\rangle\mathrm{div}(X(x))\Big)\gamma_{\adimn}(x)
-\mathrm{div}(X(x)\langle X(x),x\rangle\gamma_{\adimn}(x))\\
&=\Big(\mathrm{div}(X(x)\mathrm{div}(X(x)))-\langle X(x),x\rangle\mathrm{div}(X(x))-\mathrm{div}(X(x))\langle X(x),x\rangle\\
&\qquad\qquad+\langle X(x),x\rangle^{2}-\langle X(x),\nabla\langle X(x),x\rangle\rangle\Big)\gamma_{\adimn}(x)\\
&=\Big(\mathrm{div}(X(x)\mathrm{div}(X(x)))-\langle X(x),x\rangle\mathrm{div}(X(x))-\mathrm{div}(X(x))\langle X(x),x\rangle\\
&\qquad\qquad+\langle X(x),x\rangle^{2}-\langle X(x),X(x)\rangle
-\sum_{i,j=1}^{\adimn}x_{i}X_{j}(x)\frac{\partial}{\partial x_{j}}X_{i}(x)\Big)\gamma_{\adimn}(x).
\end{flalign*}
We then conclude by \eqref{nine2.5}.
\end{proof}

\section{Open Questions}

\begin{itemize}
\item  If $\Omega$ is convex, is it possible to improve the eigenvalue bound of Lemma \ref{lemma39} in certain specific cases?  If so, then Lemma \ref{lemma74} would give an improvement to Theorem \ref{thm2}.
\item  Is it possible to classify convex $\Omega\subset\R^{\adimn}$ with $\Sigma=\partial\Omega$ such that $\exists$ $\scon\in\R$ so that $H(x)=\langle x,N(x)\rangle+\scon$, and such that $\Sigma$ satisfies the inequality in Theorem \ref{lastthm}?
\end{itemize}

\section{Comments on the Non-Convex Case}\label{secal}

As mentioned previously, the second condition of the Main Theorem, Theorem \ref{mainthm}, holds without the assumption of convexity.  However, the first condition of Theorem \ref{mainthm} requires $\Omega$ or $\Omega^{c}$ to be convex.  In the current section, we therefore try to find a result similar to the first part of Theorem \ref{mainthm} without the assumption of convexity.

\begin{lemma}[Perturbation using $H$]\label{lemma31}
Let $\Omega\subset\R^{\adimn}$ such that $\partial\Omega\setminus\redA$ has Hausdorff dimension at most $\sdimn-7$.  Let $\Sigma\colonequals\redA$.  Suppose $\delta(\Sigma)<\infty$ and
\begin{equation}\label{three0p}
H(x)=\langle x,N\rangle+\scon,\qquad\forall\,x\in\Sigma.
\end{equation}
Let $b\in\R$ so that $\int_{\Sigma}(H(x)+b)\gamma_{\sdimn}(x)dx=0$.  If $H$ is not constant, then
\begin{equation}\label{three87}
\begin{aligned}
&\int_{\Sigma}(H+b)L(H+b)\gamma_{\sdimn}(x)dx\\
&\qquad=\int_{\Sigma}\Big(2(H+b)^{2}+(b+\scon)^{2}\vnormt{A}^{2}\Big)\gamma_{\sdimn}(x)dx
+b\int_{\Sigma}\langle x,N\rangle\gamma_{\sdimn}(x)dx\\
&\qquad=\int_{\Sigma}\Big(2(H+b)^{2}+(b+\scon)^{2}[\vnormt{A}^{2}-1]\Big)\gamma_{\sdimn}(x)dx
-\scon\int_{\Sigma}\langle x,N\rangle\gamma_{\sdimn}(x)dx\\
\end{aligned}
\end{equation}
\end{lemma}
\begin{proof}
Using Lemma \ref{lemma39.1}, Corollary \ref{lemma39.79} and \eqref{three4.5}
\begin{flalign*}
&\int_{\Sigma}(H+b)L(H+b)\gamma_{\sdimn}(x)dx
=\int_{\Sigma}(HLH+2HLb+b^{2}L(1))\gamma_{\sdimn}(x)dx\\
&\qquad\qquad\stackrel{\eqref{three9}\wedge\eqref{three4.5}}{=}
\int_{\Sigma}\Big(H(2H+\scon\vnormt{A}^{2})+2Hb(\vnormt{A}^{2}+1)+b^{2}(\vnormt{A}^{2}+1)\Big)\gamma_{\sdimn}(x)dx.
\end{flalign*}
From \eqref{three9} and Corollary \ref{lemma39.79} again,
$$\int_{\Sigma}(2H+\scon\vnormt{A}^{2})\gamma_{\sdimn}(x)dx
=\int_{\Sigma}LH\gamma_{\sdimn}(x)dx
=\int_{\Sigma}HL(1)\gamma_{\sdimn}(x)dx
=\int_{\Sigma}H(\vnormt{A}^{2}+1)\gamma_{\sdimn}(x)dx.$$
So, $\int_{\Sigma}H\vnormt{A}^{2}\gamma_{\sdimn}(x)dx=\int_{\Sigma}(H+\scon\vnormt{A}^{2})\gamma_{\sdimn}(x)dx$.  In summary,
\begin{flalign*}
&\int_{\Sigma}(H+b)L(H+b)\gamma_{\sdimn}(x)dx\\
&\qquad=\int_{\Sigma}\Big(2H^{2}+ (\scon+2b)(H+\scon\vnormt{A}^{2})+2Hb+b^{2}(\vnormt{A}^{2}+1)\Big)\gamma_{\sdimn}(x)dx\\
&\qquad=\int_{\Sigma}\Big(2(H+b)^{2}-2b^{2}+H\scon +\scon^{2}\vnormt{A}^{2}+2b\scon\vnormt{A}^{2}+b^{2}(\vnormt{A}^{2}+1)\Big)\gamma_{\sdimn}(x)dx\\
&\qquad=\int_{\Sigma}\Big(2(H+b)^{2}-b^{2}-b\scon +(\scon+b)^{2}\vnormt{A}^{2}\Big)\gamma_{\sdimn}(x)dx\\
&\qquad=\int_{\Sigma}\Big(2(H+b)^{2}+(b+\scon)^{2}\vnormt{A}^{2}\Big)\gamma_{\sdimn}(x)dx
-b(b+\scon)\int_{\Sigma}\gamma_{\sdimn}(x)dx.
\end{flalign*}
In the penultimate line, we used the definition of $b$, so that $\int_{\Sigma}H\gamma_{\sdimn}(x)dx=-\int_{\Sigma}b\gamma_{\sdimn}(x)dx$. Finally, $\int_{\Sigma}(b+\scon)\gamma_{\sdimn}(x)dx=\int_{\Sigma}(-H+\scon)\gamma_{\sdimn}(x)dx\stackrel{\eqref{three0p}}{=}-\int_{\Sigma}\langle x,N\rangle\gamma_{\sdimn}(x)dx$.  The first part of \eqref{three87} is proven.  The second follows by writing $b=b+\scon-\scon$.
\end{proof}

Lemmas \ref{lemma31}, \ref{varlem2} and \ref{lemma95.1} have the following corollary.

\begin{cor}\label{cor9}
Let $\Omega\subset\R^{\adimn}$ minimize Problem \ref{prob1}.  Let $\Sigma\colonequals\redA$.  From Lemma \ref{varlem}, $\exists$ $\scon\in\R$ such that $H(x)=\langle x,N(x)\rangle+\scon$ for all $x\in\Sigma$.

If either
\begin{itemize}
\item[(i)] $-\int_{\Sigma}H(x)\gamma_{\sdimn}(x)dx\cdot \int_{\Sigma}\langle x,N\rangle\gamma_{\sdimn}(x)dx>0$, or
\item[(ii)] $\int_{\Sigma}(\vnormt{A}^{2}-1)\gamma_{\sdimn}(x)dx>0$ and $-\scon\int_{\Sigma}\langle x,N\rangle\gamma_{\sdimn}(x)dx>0$,
\end{itemize}
then $H$ must be constant on $\redA$.
\end{cor}

By Corollary \ref{cor9}(ii), if we want a condition resembling the second condition of Theorem \ref{mainthm} to hold without the assumption of convexity, we must investigate the case that $H$ is constant in Problem \ref{prob1}.

\subsection{The case of constant mean curvature}

\begin{prop}\label{thm10}
Let $\Omega\subset\R^{\adimn}$ minimize Problem \ref{prob1}.  Let $\Sigma\colonequals\redA$.  Assume $H$ is constant on $\Sigma$.  Then either $\Sigma$ is a round cylinder, or $H=\langle x,N\rangle=0$ for all $x\in\Sigma$.
\end{prop}
\begin{proof}
%From Corollary \ref{cor3}, $H$ is constant.
By \eqref{three9} and \eqref{three4.5},
$$LH=H(\vnormt{A}^{2}+1)=2H+\scon\vnormt{A}^{2}.$$
So, $\vnormt{A}^{2}(H-\scon)=H$.  If $H\neq0$, then $\scon\neq H$.  So, if $H\neq0$, then $\vnormt{A}$ is constant.  So, if $H\neq0$, it follows from \cite[Proof of Theorem 10.1]{colding12a} \cite[p.187-188]{huisken93} that $\Sigma$ is a round cylinder.

If $H=0$, then either $\scon=H=0$ or $\vnormt{A}=0$.  If $\scon=0$, then $H=\langle x,N\rangle=0$ by Lemma \ref{varlem}.  If $\vnormt{A}$ is constant, then it follows from \cite[Proof of Theorem 10.1]{colding12a} \cite[p.187-188]{huisken93} that $\Sigma$ is a round cylinder.
\end{proof}

\begin{remark}\label{dimrk}
By Lemma \ref{lemma51}, if $\sdimn\leq6$ and if $H$ is constant on $\Sigma$, then in the setting of Corollary \ref{cor9}, it cannot occur that $H=\langle x,N\rangle=0$ $\forall$ $x\in\Sigma$.
\end{remark}

By Proposition \ref{thm10}, the only remaining case to consider in Corollary \ref{cor9} is when $H=\langle x,N\rangle=0$ for all $x\in\Sigma$.  That is, we must consider when $\Sigma$ is a cone with mean curvature zero.   This case is eliminated by the following Lemma.
 %In order to handle this case, we rely on the work of \cite{simons68} (see also \cite{zhu16,zhu16b}), where normal variations of the Gaussian surface area of cones is considered (or equivalently, normal variations of submanifolds of the sphere).

\begin{lemma}\label{lemma52}
Let $\Omega\subset\R^{\adimn}$ minimize Problem \ref{prob1}.  Let $\Sigma\colonequals\redA$.  Then it cannot occur that $H(x)=\langle x, N(x)\rangle=0$ for all $x\in\Sigma$.
\end{lemma}
\begin{proof}
At every point $x\in\Sigma$, suppose we label the $\sdimn$ eigenvalues of $A$ in order as $\lambda_{1}(x)\leq\cdots\leq\lambda_{\sdimn}(x)$.  Except on a set of Hausdorff dimension at most $\sdimn-1$ on $\Sigma$, these eigenvalues are $C^{\infty}$ functions \cite[Theorem II.5.4, p. 111]{kato66}.  For any $1\leq i\leq \sdimn$, we claim that $L\lambda_{i}=2\lambda_{i}$.  This follows from \eqref{three9p} with $\scon=0$, since $A$ can be diagonalized in a neighborhood of any $x\in\Sigma$ (except on a set of $x$ of Hausdorff dimension at most $\sdimn-1$.)  We claim that there exists a function $h\colon\Sigma\to\R$ and there exist constants $c_{1},\ldots,c_{\sdimn}\in\R$ such that
\begin{equation}\label{thiseq}
\lambda_{i}(x)=c_{i}h(x),\qquad\forall\,x\in\Sigma.
\end{equation}
To see this, let $\alpha,\beta$ be any two distinct eigenvalues of $A$.  Let $t\in\R$ such that $\int_{\Sigma}(\alpha+t\beta)\gamma_{\sdimn}(x)dx=0$.  Let $f\colonequals\alpha+t\beta$.  Then $\int_{\Sigma}f(x)\gamma_{\sdimn}(x)dx=0$, $f(x)=f(-x)$ for all $x\in\Sigma$, and $Lf=2f$.  So, $\int_{\Sigma}f(x)Lf(x)\gamma_{\sdimn}(x)=2\int_{\Sigma}(f(x))^{2}\gamma_{\sdimn}(x)$.  From Lemma \ref{varlem2}, we conclude that $f(x)=0$ for all $x\in\Sigma$.  Equation \eqref{thiseq} follows.  Since all eigenvalues of $A$ are multiples of the same function, $\vnormtf{A}$ is also an eigenfunction of $L$, with eigenvalue $2$.  Then Lemma \ref{lemma38} with $\scon=0$ says that
$$2\vnormtf{A}^{2}=\vnormtf{A}L\vnormtf{A}=2\vnormtf{A}^{2}+\vnormtf{\nabla A}^{2}-\vnormtf{\nabla\vnormtf{A}}^{2}.$$
That is, $\vnormtf{\nabla A}^{2}=\vnormtf{\nabla\vnormtf{A}}^{2}$ for all $x\in\Sigma$.  As shown in \cite[Eq. (10.33)]{colding12a}, if $A$ is diagonal at $x\in\Sigma$ and if $\vnormtf{\nabla A}^{2}=\vnormtf{\nabla\vnormtf{A}}^{2}$, then $\nabla_{e_{\sdimn}}a_{ii}=0$ for all $1\leq i\leq\sdimn-1$.  Since $\Sigma$ is a cone, if we choose $e_{\sdimn}$ such that $e_{\sdimn}$ is invariant under a dilation of the cone, we have $\abs{\nabla_{e_{\sdimn}}a_{ii}}=\abs{a_{ii}/\vnormt{x}}$ at $x\in\Sigma$.  That is, $a_{ii}=0$ for all $1\leq i\leq \sdimn-1$.  And $a_{\sdimn\sdimn}=0$ as well, by the choice of $e_{\sdimn}$.  We conclude that $A=0$ everywhere, so that $\Sigma$ is a plane through the origin.  But this finally contradicts that $\Omega$ is a symmetric set and $\Sigma=\redA$.  We conclude that $H=\langle x,N\rangle=0$ cannot occur on a set of positive $\sdimn$-dimensional Hausdorff measure.
\end{proof}

We can now finally improve the conclusion of Corollary \ref{cor9}.

\begin{cor}\label{cor10}
Let $\Omega\subset\R^{\adimn}$ minimize Problem \ref{prob1}.  Let $\Sigma\colonequals\redA$.  From Lemma \ref{varlem}, $\exists$ $\scon\in\R$ such that $H(x)=\langle x,N(x)\rangle+\scon$ for all $x\in\Sigma$.

If either
\begin{itemize}
\item[(i)] $-\int_{\Sigma} H(x)\gamma_{\sdimn}(x)dx\cdot \int_{\Sigma}\langle x,N\rangle\gamma_{\sdimn}(x)dx>0$, or
\item[(ii)] $\int_{\Sigma}(\vnormt{A}^{2}-1)\gamma_{\sdimn}(x)dx>0$ and $-\scon\int_{\Sigma}\langle x,N\rangle\gamma_{\sdimn}(x)dx>0$,
\end{itemize}
Then, after rotating $\Omega$, $\exists$ $r>0$ and $\exists$ $0\leq k\leq\sdimn$ such that $\Sigma=rS^{k}\times\R^{\sdimn-k}$.
\end{cor}
\begin{proof}
Let $\Omega$ minimize Problem \ref{prob1}.  From Proposition \ref{thm10}, $\Sigma$ is either a round cylinder, or $H(x)=\langle x,N(x)\rangle=0$ for all $x\in\Sigma$.  The latter case is eliminated by Lemma \ref{lemma52}.
\end{proof}

\begin{proof}[Proof of Theorem \ref{mainthm2}]
Combine Corollaries \ref{cor2} and \ref{cor10}.
\end{proof}

\subsection{Infinitesimal Rotations}

Corollary \ref{cor10} can eliminate e.g. star-shaped sets with $\scon<0$ such as the interior of a hyperboloid.  Below we present an argument that also eliminates sets with ``lumpy'' boundary as candidates for minimizers in Problem \ref{prob1}.  This argument is a modification of one from \cite{hutchings02}, which was itself inspired by the standard proof of the Courant Nodal Domain Theorem.

\begin{lemma}[\embolden{Infinitesimal Rotations as Eigenfunctions of $L$}]\label{lemma30f}
Let $\Sigma\subset\R^{\adimn}$ be an orientable hypersurface.  Let $\scon\in\R$.  Assume that
\begin{equation}\label{three0f}
H(x)=\langle x,N(x)\rangle+\scon,\qquad\forall\,x\in\Sigma.
\end{equation}
Let $Q$ be an $(\adimn)\times(\adimn)$ real antisymmetric matrix with $Q^{*}=-Q$.  Then
\begin{equation}\label{three9f}
L\langle Qx,N\rangle=0,\qquad\forall\,x\in\Sigma.
\end{equation}
\end{lemma}
\begin{proof}
Let $1\leq i\leq \sdimn$ and let $x\in\Sigma$.  Then
\begin{equation}\label{three6f}
\begin{aligned}
\nabla_{e_{i}}\langle Qx,N\rangle
&=\langle Qe_{i},N\rangle+\langle Qx,\nabla_{e_{i}}N\rangle\\
&=\langle Qe_{i},N\rangle+\sum_{j=1}^{\sdimn}\langle\nabla_{e_{i}} N,e_{j}\rangle\langle Qx,e_{j}\rangle
\stackrel{\eqref{three2}}{=}\langle Qe_{i},N\rangle-\sum_{j=1}^{\sdimn}a_{ij}\langle Qx,e_{j}\rangle.
\end{aligned}
\end{equation}
\begin{equation}\label{three5f}
\nabla H
=\sum_{i=1}^{\sdimn}\nabla_{e_{i}}H e_{i}
\stackrel{\eqref{three0}}{=}\sum_{i=1}^{\sdimn}\nabla_{e_{i}}\langle x,N\rangle e_{i}
=\sum_{i=1}^{\sdimn}\langle x,\nabla_{e_{i}}N\rangle e_{i}\\
\stackrel{\eqref{three2}}{=}-\sum_{i,j=1}^{\sdimn}a_{ij}\langle x,e_{j}\rangle e_{i}.
\end{equation}
Using that $A$ is a symmetric matrix,
\begin{equation}\label{three5.2f}
\begin{aligned}
\langle\nabla H,Qx\rangle
&\stackrel{\eqref{three5f}}{=}-\sum_{i,j=1}^{\sdimn}a_{ij}\langle x,e_{j}\rangle \langle e_{i},Qx\rangle
=-\sum_{i,j=1}^{\sdimn}a_{ij}\langle x,e_{i}\rangle \langle e_{j},Qx\rangle\\
&\stackrel{\eqref{three6f}}{=}\langle\nabla\langle Qx,N\rangle,x\rangle-\sum_{i=1}^{\sdimn}\langle x,e_{i}\rangle\langle Qe_{i},N\rangle.
\end{aligned}
\end{equation}
Writing $x=\langle x,N\rangle N+\sum_{i=1}^{\sdimn}\langle x,e_{i}\rangle e_{i}$, so $\langle Qx,N\rangle=\langle x,N\rangle \langle QN,N\rangle+\sum_{i=1}^{\sdimn}\langle x,e_{i}\rangle \langle Qe_{i},N\rangle$,
\begin{equation}\label{three5.3f}
\langle\nabla H,Qx\rangle
\stackrel{\eqref{three5.2f}}{=}\langle\nabla\langle Qx,N\rangle,x\rangle+\langle x,N\rangle \langle QN,N\rangle-\langle Qx,N\rangle.
\end{equation}

Choose the frame such that $\nabla_{e_{k}}^{T}e_{j}=0$ at $x$ for every $1\leq j,k\leq \sdimn$, we then have $\nabla_{e_{k}}e_{j}=a_{kj}N$ at $x$ by \eqref{three1}.  So, $\forall$ $1\leq i\leq\sdimn$,
\begin{equation}\label{three5.5f}
\begin{aligned}
&\nabla_{e_{i}}\nabla_{e_{i}}\langle Qx,N\rangle\\
&\stackrel{\eqref{three6f}}{=}
\langle Q\nabla_{e_{i}}e_{i},N\rangle
+\langle Qe_{i},\nabla_{e_{i}}N\rangle
-\sum_{j=1}^{\sdimn}\nabla_{e_{i}}a_{ij}\langle Qx,e_{j}\rangle
-\sum_{j=1}^{\sdimn}a_{ij}\langle Qe_{i},e_{j}\rangle
-\sum_{j=1}^{\sdimn}a_{ij}\langle Qx,\nabla_{e_{i}}e_{j}\rangle\\
&\stackrel{\eqref{three5f}}{=}
+a_{ii}\langle QN,N\rangle
-\sum_{j=1}^{\sdimn}a_{ij}\langle Qe_{i},e_{j}\rangle
-\sum_{j=1}^{\sdimn}\nabla_{e_{i}}a_{ij}\langle Qx,e_{j}\rangle
-\sum_{j=1}^{\sdimn}a_{ij}\langle Qe_{i},e_{j}\rangle
-\sum_{j=1}^{\sdimn}a_{ij}^{2}\langle Qx, N\rangle.
\end{aligned}
\end{equation}

So, using the Codazzi equation ($\nabla_{e_{i}}a_{kj}=\nabla_{e_{j}}a_{ki}$ $\forall$ $1\leq i,j,k\leq\sdimn$),
\begin{flalign*}
&\Delta \langle Qx,N\rangle
=\sum_{i=1}^{\sdimn}\nabla_{e_{i}}\nabla_{e_{i}}\langle Qx,N\rangle\\
&\!\stackrel{\eqref{three5.5f}\wedge\eqref{three4}}{=}
-H\langle QN,N\rangle
-2\sum_{i,j=1}^{\sdimn}a_{ij}\langle Qe_{i},e_{j}\rangle
-\sum_{i,j=1}^{\sdimn}\nabla_{e_{j}}a_{ii}\langle x,e_{j}\rangle
-\sum_{ij=1}^{\sdimn}a_{ij}^{2}\langle Qx, N\rangle\\
&\!\stackrel{\eqref{three4}}{=}-H\langle QN,N\rangle
-2\sum_{i,j=1}^{\sdimn}a_{ij}\langle Qe_{i},e_{j}\rangle
+\langle\nabla H, Qx\rangle
-\sum_{ij=1}^{\sdimn}a_{ij}^{2}\langle Qx, N\rangle\\
&\!\stackrel{\eqref{three5.3f}}{=}(-H+\langle x,N\rangle)\langle QN,N\rangle
-2\sum_{i,j=1}^{\sdimn}a_{ij}\langle Qe_{i},e_{j}\rangle
+\langle\nabla \langle Qx,N\rangle, x\rangle
-\langle Qx,N\rangle
-\vnormt{A}^{2}\langle Qx, N\rangle.
\end{flalign*}

In summary,
\begin{equation}\label{three5.7f}
L\langle Qx,N\rangle
\stackrel{\eqref{three4.5}\wedge\eqref{three0f}}{=}
-\scon\langle QN,N\rangle
-2\sum_{i,j=1}^{\sdimn}a_{ij}\langle Qe_{i},e_{j}\rangle.
\end{equation}
It remains to show that the right side of \eqref{three5.7f} is zero.  Note that we have not yet used any property of $Q$.  Since $Q^{*}=-Q$, for any $v,w\in\R^{\adimn}$, we have
\begin{equation}\label{fix1}
\langle Qv,w\rangle=w^{*}Qv=(w^{*}Qv)^{*}=v^{*}Q^{*}w=-v^{*}Qw=-\langle Qw,v\rangle.
\end{equation}
In particular, choosing $v=w=N$, we get
\begin{equation}\label{fix2}
\langle Qv,v\rangle=0,\qquad\forall\,v\in\R^{\adimn}.
\end{equation}
So, the first term on the right of \eqref{three5.7f} is zero.  Lastly, using that $A$ is a symmetric matrix,
$$
\sum_{i,j=1}^{\sdimn}a_{ij}\langle Qe_{i},e_{j}\rangle
\stackrel{\eqref{fix1}}{=}-\sum_{i,j=1}^{\sdimn}a_{ij}\langle Qe_{j},e_{i}\rangle
=-\sum_{i,j=1}^{\sdimn}a_{ij}\langle Qe_{i},e_{j}\rangle.
$$
So, the last term on the right of \eqref{three5.7f} is zero.  That is, \eqref{three9f} holds.
\end{proof}

\begin{lemma}\label{lemma94}
Let $\Omega\subset\R^{\adimn}$ with $-\Omega=\Omega$.  Let $\Sigma\colonequals\redA$.  Let $Q$ be a real antisymmetric $(\adimn)\times(\adimn)$ matrix.  Then
$$\int_{\Sigma}\langle Qx,N(x)\rangle\gamma_{\sdimn}(x)dx=0.$$
\end{lemma}
\begin{proof}
By the divergence theorem,
$$\int_{\Sigma}\langle Qx,N(x)\rangle\gamma_{\sdimn}(x)dx=\int_{\Omega}(-\langle Qx,x\rangle+\mathrm{Tr}(Q))\gamma_{\sdimn}(x)dx\stackrel{\eqref{fix2}}{=}0.$$
Here $\mathrm{Tr}$ denotes the trace of a matrix.
\end{proof}

We define the number of nodal domains of a function $f\colon\Sigma\to\R$ to be the number of connected components of the set $\{x\in\Sigma\colon f(x)\neq0\}$.

\begin{cor}\label{cor11}
Let $\Omega\subset\R^{\adimn}$ with $-\Omega=\Omega$.  Let $\Sigma\colonequals\redA$.  Suppose there exists a real antisymmetric $(\adimn)\times(\adimn)$ matrix $Q$ with $Q^{*}=-Q$ such that the function $f\colon\Sigma\to\R$ defined by $f(x)\colonequals\langle Qx,N(x)\rangle$ $\forall$ $x\in\Sigma$ has more than four nodal domains.  Then $\Omega$ does not minimize Problem \ref{prob1}.
\end{cor}
\begin{remark}
Since $\Omega=-\Omega$ and $\int_{\Sigma}f(x)\gamma_{\sdimn}(x)dx=0$ by Lemma \ref{lemma94}, $f$ cannot have exactly two nodal domains.  If $\Sigma$ is a non-spherical ellipsoid aligned with the coordinate axes, and if we choose $Q$ to have all zero entries other than the upper left corner of $\begin{pmatrix} 0 & 1\\ -1 & 0\end{pmatrix}$, then $f$ has four connected components.  And if $\Sigma$ is a curve in the plane with many oscillations, then $f$ has many connected components.  So, the assumption of the theorem implies that $\Omega$ has a ``lumpy'' boundary.
\end{remark}
\begin{proof}
Label two of the nodal domains as $D_{1},D_{2}\subset\Sigma$, so that $D_{1}\neq D_{2}$, $D_{1}\neq -D_{2}$.  Let $\alpha\in\R$.  Define $g\colon\Sigma\to\R$ so that $g(x)\colonequals f(x)$ for any $x\in D_{1}\cup(-D_{1})$, $g(x)\colonequals \alpha f(x)$ for any $x\in D_{2}\cup(-D_{2})$, and $g(x)\colonequals 0$ otherwise.  Choose $\alpha\in\R$, such that $\int_{\Sigma}g(x)\gamma_{\sdimn}(x)dx=0$. Then $g(x)=g(-x)$ for all $x\in\Sigma$, $g$ vanishes on an open subset of $\Sigma$, and $g$ vanishes on the set where $\nabla g$ is discontinuous.  Also, $Lg=0$ on $\Sigma\setminus[(\partial D_{1})\cup(\partial (- D_{1}))\cup(\partial D_{2})\cup(\partial (- D_{2}))]$.  Also $\int_{\Sigma}gLg\gamma_{\sdimn}(x)dx=0$ by Lemma \ref{lemma30f}.

Assume for the sake of contradiction that $\Omega$ minimizes Problem \ref{prob1}.  From Lemma \ref{varlem2}, if $h\colon\Sigma\to\R$ is any $C^{\infty}$ function such that $\int_{\Sigma}h(x)\gamma_{\sdimn}(x)dx$ and $h(x)=h(-x)$ for all $x\in\Sigma$, then for any $t\in\R$,
$$\int_{\Sigma}(g+th)L(g+th)\gamma_{\sdimn}(x)dx\leq0.$$
Since this holds for all $t\in\R$ and $\int_{\Sigma}gLg\gamma_{\sdimn}(x)dx=0$, we conclude that
$$\int_{\Sigma}(gLh+hLg)\gamma_{\sdimn}(x)dx=0.$$
Integrating by parts with Lemma \ref{lemma39.79} (which is valid since $g$ vanishes on the set where $\nabla g$ is discontinuous and $g\in W_{1,2}(\Sigma,\gamma_{\sdimn})$, as defined before Lemma \ref{lemma28}),
$$2\int_{\Sigma}hLg\gamma_{\sdimn}(x)dx=0.$$
Since this equation holds for any mean zero symmetric $C^{\infty}$ function $h$, we conclude that $Lg=0$ in the distributional sense.  By elliptic regularity, $Lg=0$ on all of $\Sigma$.  By the unique continuation property, since $g$ vanishes on an open subset of $\Sigma$, we conclude that $g=0$ on $\Sigma$.  This contradicts the existence of more than one nodal domain of $f$.  We conclude that $\Omega$ does not minimize Problem \ref{prob1}.
\end{proof}

\medskip
\noindent\textbf{Acknowledgement}.  Thanks to Vesa Julin for helpful discussions, especially concerning volume preserving extensions of a function.  Thanks also to Domenico La Manna for sharing his preprint \cite{lamanna17}.

\bibliographystyle{amsalpha}
\bibliography{12162011}

\end{document}